\documentclass[11pt]{article}
\usepackage[numbers,square]{natbib}
\usepackage{amsmath}
\usepackage{amsthm}
\usepackage{amsfonts}
\usepackage{amssymb}
\usepackage{commath}
\usepackage{graphicx}
\usepackage{xspace}
\usepackage{color}
\usepackage{tikz}
\usetikzlibrary{math,arrows,arrows.meta,calc}
\usepackage[countmax]{subfloat}
\usepackage{caption}
\usepackage{subcaption}
\usepackage{stmaryrd}
\usepackage{url}
\usepackage{amsthm}
\usepackage{footnote}
\makesavenoteenv{tabular}
\makesavenoteenv{table}
\usepackage[hidelinks]{hyperref}
\hypersetup{colorlinks = true, urlcolor = blue,
  linkcolor = blue, citecolor = blue}
\usepackage{cleveref}

\usepackage[margin=1.0in]{geometry}
\makeatletter
\newcommand{\tnorm}{\@ifstar\@tnorms\@tnorm}
\newcommand{\@tnorms}[1]{%
  \left|\mkern-1.5mu\left|\mkern-1.5mu\left|
   #1
  \right|\mkern-1.5mu\right|\mkern-1.5mu\right|
}
\newcommand{\@tnorm}[2][]{%
  \mathopen{#1|\mkern-1.5mu#1|\mkern-1.5mu#1|}
  #2
  \mathclose{#1|\mkern-1.5mu#1|\mkern-1.5mu#1|}
}
\makeatother

\newtheorem{theorem}{Theorem}[section]
\newtheorem{lemma}[theorem]{Lemma}

\newtheorem{remark}{Remark}
\numberwithin{equation}{section}
\title{Space-time hybridizable
  discontinuous Galerkin method for advection-diffusion on deforming
  domains: \\The advection-dominated regime}
\author{Y. Wang\thanks{Department of Applied Mathematics, University of
    Waterloo, ON, Canada (\url{yuan.wang@uwaterloo.ca}),
    \url{http://orcid.org/0009-0006-8092-4378}
 } \and
  S. Rhebergen\thanks{Department of Applied Mathematics, University of
    Waterloo, ON, Canada (\url{srheberg@uwaterloo.ca}),
    \url{http://orcid.org/0000-0001-6036-0356}}}
\begin{document}
\maketitle
\begin{abstract}
	We analyze a space-time hybridizable discontinuous Galerkin
	method to solve the time-dependent advection-diffusion equation on
	deforming domains. We prove stability of the discretization in the
	advection-dominated regime by using weighted test functions and
	derive a priori space-time error estimates. A numerical example
	illustrates the theoretical results.
\end{abstract}
\section{Introduction}
\label{sec:introduction}

In this paper we analyze a space-time hybridizable discontinuous
Galerkin (HDG) method for the time-dependent advection-diffusion
problem on a time-dependent polygonal $(d=2$) or polyhedral ($d=3$)
domain $\Omega(t) \subset\mathbb{R}^d$, that evolves continuously in
the time interval $t\in\sbr{0,T}$. This problem is given by
\begin{equation}
  \label{eq:adr}
  \partial_t u
  +
  \overline{\nabla}\cdot\del[0]{\bar{\beta}u}
  -
  \varepsilon\overline{\nabla}^2u
  =f
  \quad
  \text{in }
  \Omega(t),
  \;
  0< t\leq T,
\end{equation}
in which
$\overline{\nabla} = \del[0]{\partial_{x_1}, \partial_{x_2}, \dots,
  \partial_{x_d}}$ denotes the spatial gradient, $\bar{\beta}$ is the
given divergence-free advective field, $\varepsilon>0$ is the constant
diffusion coefficient, and $f$ is a forcing term. The focus in this
paper is the advection-dominated regime ($\varepsilon \ll 1$).

Many different finite element methods have been designed and analyzed
for advection-diffusion problems. These include, for example,
streamline upwind Petrov--Galerkin (SUPG)
\citep{Brooks:1982,Burman:2010,John:2011,Novo:2016}, continuous
interior penalty (CIP) \citep{Burman:2004,Burman:2005,Burman:2009},
discontinuous Galerkin (DG)
\citep{Cockburn:1998,Cockburn:2001,Ern:2006,Reed:1973}, and HDG
\citep{Chen:2012,Egger:2010,Nguyen:2009,Wells:2011} methods. Of
particular interest to the current work is the analysis presented by
\citet{Ayuso:2009}. They introduce a weighted test function to analyze
a DG method for the stationary advection-diffusion-reaction
problem. For this they assume that $\bar{\beta}$ has no closed curves
nor stationary points which implies the existence of a smooth function
$\psi$ such that $\bar{\beta} \cdot \overline{\nabla} \psi(x) \ge b_0$
for some constant $b_0 > 0$ depending on the inverse of the diameter
of the domain. The function $\psi$ is used to define a weighting
function $\varphi := \exp(-\psi) + \chi$, with $\chi$ a free to choose
positive constant, which is used to show stability of the DG method.
They are able to show stability for the advection-diffusion-reaction
problem independent of the diffusion parameter $\varepsilon$, and
therefore stability of the DG method in the advection-dominated
regime, in a norm that provides control of the streamline
derivative. Furthermore, their analysis also holds for the
advection-diffusion problem with solenoidal advective field, thereby
relaxing the usual coercivity condition that
$\mu(x) + \tfrac{1}{2}\overline{\nabla}\cdot\bar{\beta} > 0$, where
$\mu(x)$ is the reaction coefficient. The same idea is also used to
analyze the HDG method for the stationary advection-diffusion problem
in the advection-dominated regime in \citet{Fu:2015}.

To discretize the advection-diffusion problem on a time-dependent
domain we consider a fully discrete space-time formulation of
\cref{eq:adr} using DG time stepping \citep{Jamet:1975,Jamet:1978}. DG
time stepping can be combined with different spatial
discretizations. On fixed domains, for example, DG time stepping
combined with SUPG was analyzed for the advection-diffusion equation
in \citet{Hughes:1987}, while space-time DG, in which DG is applied
both in space and time, was analyzed for a nonlinear
advection-diffusion problem in \citet{Feistauer:2011}. The space-time
DG method for the (linear) advection-diffusion problem on a
time-dependent domain was analyzed in \citet{Sudirham:2006} by
considering the space-time discretization on a space-time mesh
consisting of anisotropic (in space and time) elements. This enabled
them to obtain error estimates in terms of the spatial mesh size and
the time step, extending the analysis of DG methods for the stationary
advection-diffusion problem on anisotropic spatial meshes
\citep{Georgoulis:thesis} to space-time.

DG methods are known to be expensive due to the large number of
degrees-of-freedom (dofs) compared to, for example, a continuous
finite element method on the same mesh. This resulted in the
introduction of the HDG method by \citet{Cockburn:2009}, a DG method
that uses static condensation to reduce the number of globally coupled
dofs. The extension of HDG to space-time, in which HDG is used to
discretize a PDE in both space and time, was presented in
\citep{Rhebergen:2012,Rhebergen:2013}. Furthermore, it was
demonstrated in \citet{Sivas:2021} that nonsymmetric algebraic
multigrid, based on approximate ideal restriction
\citep{Manteuffel:2019, Manteuffel:2018}, is an effective
preconditioner for space-time HDG discretizations of the
advection-diffusion problem in the advection-dominated regime.

The space-time HDG method for \cref{eq:adr} was analyzed in
\citet{Kirk:2019} following the space-time anisotropic framework used
in \citet{Sudirham:2006}. However, despite the space-time HDG method
performing well in practice for $\varepsilon \ll 1$, the
well-posedness result proven in \citet{Kirk:2019} does not hold in the
advection dominated regime. To address this discrepancy between
practice and theory, we revisit the analysis in \citet{Kirk:2019} and
focus on its extension to the advection-dominated regime. Like the
analysis of \citet{Ayuso:2009} discussed above, we will use a weighted
test function. However, our weighting function is constructed
explicitly for the time-dependent advection-diffusion equation and
only depends on the time variable $t$ and final time $T$, not on the
space variable $x$. We prove stability of the space-time HDG method in
a norm providing control on the streamline derivative and present
error estimates for the advection-dominated regime that are also valid
on moving meshes.

The paper is organized as follows. In \cref{s:advecdiffuprob}, we
introduce the space-time formulation of the time-dependent
advection-diffusion problem. In \cref{s:spacetimehdg} we describe the
finite element spaces, present inequalities that will be used in the
analysis, and the space-time HDG discretization of the
advection-diffusion problem. The main result of \cref{s:stability} is
inf-sup stability of the discretization. The a priori error analysis
is then presented in \cref{s:conv_anal}. We provide pre-asymptotic and
asymptotic convergence rates, with the transition from the former to
the latter featuring a drop in the rate of convergence from $p+1/2$ to
$p$, with $p$ the degree of the polynomial approximation. In
\cref{s:numerics}, we numerically confirm the error analysis by
solving a time-dependent advection-diffusion problem on a deforming
mesh that contains hanging nodes in both spatial and temporal
directions.

\section{The space-time formulation of the advection-diffusion problem}
\label{s:advecdiffuprob}

In this section we formulate \cref{eq:adr} as a problem in
$(d+1)$-dimensional space-time. For this, we define the
$(d+1)$-dimensional polyhedral space-time domain as
$\mathcal{E}:=\{ (t,x): x\in\Omega(t), 0<t<T
\}\subset\mathbb{R}^{d+1}$. Its boundary, $\partial\mathcal{E}$,
consists of $\Omega(0):=\{(t,x) \in \partial\mathcal{E}\,:\, t=0\}$,
$\Omega(T):=\{(t,x) \in \partial\mathcal{E}\,:\, t=T\}$, and
$\mathcal{Q}_{\mathcal{E}}:=\{(t,x) \in \partial\mathcal{E}\,:\,
0<t<T\}$. The outward space-time normal vector to
$\partial\mathcal{E}$ is denoted by ${n}:=\del[0]{n_t,\bar{n}}$, where
$n_t$ and $\bar{n}$ are the temporal and spatial components of the
space-time normal vector, respectively. Introducing the space-time
advective field $\beta:=\del[0]{1,\bar{\beta}}$ and space-time
gradient operator $\nabla:=\del[0]{\partial_t,\overline{\nabla}}$, the
space-time formulation of \cref{eq:adr} is given by
\begin{subequations}
  \label{eq:st_adr}
  \begin{equation}
    \nabla\cdot\del{\beta u}
    -\varepsilon\overline{\nabla}^2u
    =f
    \text{ in }
    \mathcal{E}.
  \end{equation}
  We consider a nonoverlapping partition of the domain boundary,
  $\partial\mathcal{E} = \partial\mathcal{E}_D \cup
  \partial\mathcal{E}_N$, and impose the boundary conditions
  \begin{equation}
    -\zeta^- u \beta \cdot{n}
    +\varepsilon\overline{\nabla}u\cdot\bar{n}
    =g
    \text{ on }
    \partial\mathcal{E}_N,
    \qquad
    u=0
    \text{ on }
    \partial\mathcal{E}_D.
  \end{equation}
\end{subequations}
The Dirichlet $\partial\mathcal{E}_D$ and Neumann
$\partial\mathcal{E}_N$ boundaries are defined here by:
\begin{equation*}
  \partial\mathcal{E}_D
  :=
  \cbr[0]{
    \del[0]{t,x}:
    x\in\Gamma_D(t),
    0<t\leq T
  },
  \qquad
  \partial\mathcal{E}_N
  :=
  \cbr[0]{
    \del[0]{t,x}:
    x\in\Gamma_N(t)\cup\Omega(0)\cup\Omega(T),
    0<t\leq T
  },
\end{equation*}
where we also prescribe a nonoverlapping partition of the boundary of
$\Omega(t)$, i.e.,
$\partial\Omega(t)=\Gamma_D(t)\cup\Gamma_N(t)$. Furthermore, $\zeta^-$
is an indicator function for the inflow (where $\beta \cdot n < 0$)
part of the boundary of $\mathcal{E}$. Therefore, the boundary
condition on $\partial\mathcal{E}_N$ also imposes the initial
condition $u(x,0)=g(x)$ on $\Omega(0)$.

We assume that the forcing term $f$ lies in $L^2(\mathcal{E})$ and
that the Neumann boundary data $g$ lies in
$L^2(\partial\mathcal{E}_N)$. Furthermore, we assume that
$\bar{\beta} \in \sbr[0]{W^{1,\infty}(\mathcal{E})}^{d}$,
$\norm[0]{\bar{\beta}}_{L^{\infty}(\mathcal{E})}\le 1$ and, following
\citet{Ayuso:2009}, that
$\norm[0]{\bar{\beta}}_{W^{1,\infty}(\mathcal{E})}\leq
c\norm[0]{\bar{\beta}}_{L^{\infty}(\mathcal{E})}\le c$.

\section{The space-time hybridizable discontinuous Galerkin method}
\label{s:spacetimehdg}

\subsection{Description of space-time slabs, faces, and elements}
\label{ss:description-stslabsfaceselements}

An initial partition of the space-time domain $\mathcal{E}$ consists
of dividing the time interval $\sbr[0]{0,T}$ into time levels
$0=t_0<t_1<\cdots<t_N=T$ and defining the $n$th time interval as
$I_n=\del[0]{t_n,t_{n+1}}$. The space-time domain is divided into
space-time slabs
$\mathcal{E}^n := \mathcal{E}\cap\del[0]{I_n\times\mathbb{R}^d}$,
which are then divided into space-time elements,
$\mathcal{E}^n=\cup_j\mathcal{K}^n_j$. To construct the space-time
element $\mathcal{K}^n_j$, we divide the domain $\Omega(t_n)$ into
nonoverlapping spatial elements $K^n_j$ so that
$\Omega(t_n)=\cup_jK^n_j$. Let $\Upsilon$ be the transformation
describing the deformation of the domain. The spatial elements
$K^{n+1}_j$ at $t_{n+1}$ are obtained by mapping the nodes of the
elements $K^n_j$ into their new position via the transformation
$\Upsilon$. Each space-time element $\mathcal{K}^n_j$ is obtained by
connecting the elements $K^n_j$ and $K^{n+1}_j$ via linear
interpolation in time following \citet{Vegt:2002}. We denote the set of
all space-time elements tessellating the space-time domain by
$\mathcal{T}_h$.

The boundary of a space-time element $\mathcal{K}$ with
$t \in (t_*,t^*)$ is partitioned as
$\partial \mathcal{K} = \mathcal{Q}_{\mathcal{K}} \cup
\mathcal{R}_{\mathcal{K}}$ where
$\mathcal{R}_{\mathcal{K}} := K_* \cup K^*$,
$\mathcal{Q}_{\mathcal{K}} \cap \mathcal{R}_{\mathcal{K}} =
\emptyset$, and where $K_*$ denotes the face of $\mathcal{K}$ at time
$t=t_*$ and $K^*$ denotes the face of $\mathcal{K}$ at time $t^*$. On
$\partial\mathcal{K}$, the outward unit space-time normal vector is
denoted by
$n^{\mathcal{K}} = \del[0]{n_t^{\mathcal{K}},\bar{n}^{\mathcal{K}}}$,
where $n_t^{\mathcal{K}}$ and $\bar{n}^{\mathcal{K}}$ are the temporal
and spatial components of the space-time normal vector,
respectively. Note that $\bar{n}$ is the zero vector on an
$\mathcal{R}$-face, i.e., that $\bar{n}=0$ on $K^*$ and $K_*$, and
that $\bar{n} \ne 0$ on a $\mathcal{Q}$-face.

The set of all faces in the mesh shared between at most two elements
is denoted by $\mathcal{F}_h$. Within this set, we denote by
$\mathcal{F}_h^i$, $\mathcal{F}_h^b$, $\mathcal{F}_{\mathcal{Q},h}$,
and $\mathcal{F}_{\mathcal{R},h}$ all the interior faces, boundary
faces, $\mathcal{Q}$-faces, and $\mathcal{R}$-faces respectively. When
the mesh is viewed from an element's perspective, we denote by
$\partial\mathcal{T}_h$ the set of element boundaries. Within this
set, we denote by $\mathcal{Q}_h$ and $\mathcal{R}_h$ the set of
element boundary $\mathcal{Q}$-faces and element boundary
$\mathcal{R}$-faces, respectively. Finally, the union of all faces in
$\mathcal{F}_h$ is denoted by $\Gamma$.  An illustration of a
$(d+1)$-dimensional space-time mesh in slab $\mathcal{E}^n$, with
$d=2$, is shown in \cref{fig:st_dom_ex}.
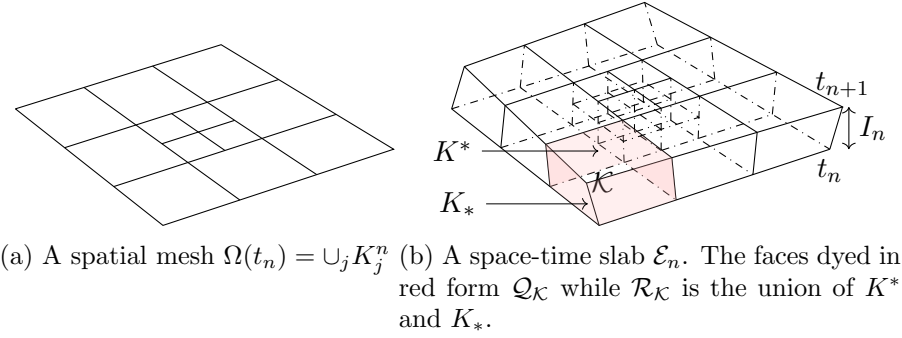
\begin{figure}[tbp]
  \centering
  \subfloat[A spatial mesh $\Omega(t_n)=\cup_{j}K_j^n$]{
    \tikzmath{
      %
      coordinate \s{01}, \s{02}, \s{03}, \s{04};
      \s{01} = (0,0);
      \s{04} = (3,-0.5);
      \s{13} = (0.8,-2.7);
      \s{16} = (4.2,-3.2);
      coordinate \x{01}, \x{02}, \x{03}, \x{04};
      coordinate \x{05}, \x{06}, \x{07}, \x{08};
      coordinate \x{09}, \x{10}, \x{11}, \x{12};
      coordinate \x{13}, \x{14}, \x{15}, \x{16};
      \a = 30; 
      \t = 1.2; 
      \x{01} =
      ({\t*(\sx{01}*cos(\a)-\sy{01}*sin(\a))},{\sx{01}*sin(\a)+\sy{01}*cos(\a)});
      \x{04} =
      ({\t*(\sx{04}*cos(\a)-\sy{04}*sin(\a))},{\sx{04}*sin(\a)+\sy{04}*cos(\a)});
      \x{13} =
      ({\t*(\sx{13}*cos(\a)-\sy{13}*sin(\a))},{\sx{13}*sin(\a)+\sy{13}*cos(\a)});
      \x{16} =
      ({\t*(\sx{16}*cos(\a)-\sy{16}*sin(\a))},{\sx{16}*sin(\a)+\sy{16}*cos(\a)});
      \x{02} = ({(2*\xx{01}+\xx{04})/3},{(2*\xy{01}+\xy{04})/3});
      \x{14} = ({(2*\xx{13}+\xx{16})/3},{(2*\xy{13}+\xy{16})/3});
      \x{03} = ({(2*\xx{04}+\xx{01})/3},{(2*\xy{04}+\xy{01})/3});
      \x{15} = ({(2*\xx{16}+\xx{13})/3},{(2*\xy{16}+\xy{13})/3});
      \x{05} = ({(2*\xx{01}+\xx{13})/3},{(2*\xy{01}+\xy{13})/3});
      \x{08} = ({(2*\xx{04}+\xx{16})/3},{(2*\xy{04}+\xy{16})/3});
      \x{09} = ({(2*\xx{13}+\xx{01})/3},{(2*\xy{13}+\xy{01})/3});
      \x{12} = ({(2*\xx{16}+\xx{04})/3},{(2*\xy{16}+\xy{04})/3});
      \x{06} = ({(\xx{05}+\xx{08})/2},{(\xy{05}+\xy{08})/2});
      \x{07} = ({(\xx{09}+\xx{12})/2},{(\xy{09}+\xy{12})/2});
      \x{10} = ({(\xx{02}+\xx{14})/2},{(\xy{02}+\xy{14})/2});
      \x{11} = ({(\xx{03}+\xx{15})/2},{(\xy{03}+\xy{15})/2});
    }
    \begin{tikzpicture}[scale=0.8]
      \draw[] (\x{01}) -- (\x{04}) -- (\x{16}) -- (\x{13}) -- (\x{01});
      \draw[] (\x{02}) -- (\x{14});
      \draw[] (\x{03}) -- (\x{15});
      \draw[] (\x{05}) -- (\x{08});
      \draw[] (\x{09}) -- (\x{12});
      \draw[] (\x{06}) -- (\x{07});
      \draw[] (\x{10}) -- (\x{11});
    \end{tikzpicture}
  }
  \subfloat[A space-time slab $\mathcal{E}_n$. The
  faces dyed in red form $\mathcal{Q}_{\mathcal{K}}$ while
  $\mathcal{R}_{\mathcal{K}}$ is the union of $K^*$ and $K_*$.
  \label{fig:localtimestep}
  ]{
    \tikzmath{
      coordinate \s{01}, \s{02}, \s{03}, \s{04};
      \s{01} = (0,0);
      \s{04} = (3,-0.5);
      \s{13} = (0.8,-2.7);
      \s{16} = (4.2,-3.2);
      coordinate \x{01}, \x{02}, \x{03}, \x{04};
      coordinate \x{05}, \x{06}, \x{07}, \x{08};
      coordinate \x{09}, \x{10}, \x{11}, \x{12};
      coordinate \x{13}, \x{14}, \x{15}, \x{16};
      coordinate \x{17}, \x{18}, \x{19}, \x{20}, \x{21};
      coordinate \x{22}; 
      coordinate \x{23}; 
      \a = 30; 
      \t = 1.2; 
      \x{01} =
      ({\t*(\sx{01}*cos(\a)-\sy{01}*sin(\a))},{\sx{01}*sin(\a)+\sy{01}*cos(\a)});
      \x{04} =
      ({\t*(\sx{04}*cos(\a)-\sy{04}*sin(\a))},{\sx{04}*sin(\a)+\sy{04}*cos(\a)});
      \x{13} =
      ({\t*(\sx{13}*cos(\a)-\sy{13}*sin(\a))},{\sx{13}*sin(\a)+\sy{13}*cos(\a)});
      \x{16} =
      ({\t*(\sx{16}*cos(\a)-\sy{16}*sin(\a))},{\sx{16}*sin(\a)+\sy{16}*cos(\a)});
      \x{02} = ({(2*\xx{01}+\xx{04})/3},{(2*\xy{01}+\xy{04})/3});
      \x{14} = ({(2*\xx{13}+\xx{16})/3},{(2*\xy{13}+\xy{16})/3});
      \x{03} = ({(2*\xx{04}+\xx{01})/3},{(2*\xy{04}+\xy{01})/3});
      \x{15} = ({(2*\xx{16}+\xx{13})/3},{(2*\xy{16}+\xy{13})/3});
      \x{05} = ({(2*\xx{01}+\xx{13})/3},{(2*\xy{01}+\xy{13})/3});
      \x{08} = ({(2*\xx{04}+\xx{16})/3},{(2*\xy{04}+\xy{16})/3});
      \x{09} = ({(2*\xx{13}+\xx{01})/3},{(2*\xy{13}+\xy{01})/3});
      \x{12} = ({(2*\xx{16}+\xx{04})/3},{(2*\xy{16}+\xy{04})/3});
      \x{06} = ({(\xx{05}+\xx{08})/2},{(\xy{05}+\xy{08})/2});
      \x{07} = ({(\xx{09}+\xx{12})/2},{(\xy{09}+\xy{12})/2});
      \x{10} = ({(\xx{02}+\xx{14})/2},{(\xy{02}+\xy{14})/2});
      \x{11} = ({(\xx{03}+\xx{15})/2},{(\xy{03}+\xy{15})/2});
      \x{17} = ({(2*\xx{02}+\xx{14})/3},{(2*\xy{02}+\xy{14})/3});
      \x{19} = ({(2*\xx{14}+\xx{02})/3},{(2*\xy{14}+\xy{02})/3});
      \x{18} = ({(2*\xx{03}+\xx{15})/3},{(2*\xy{03}+\xy{15})/3});
      \x{20} = ({(2*\xx{15}+\xx{03})/3},{(2*\xy{15}+\xy{03})/3});
      \x{21} = ({(\xx{10}+\xx{11})/2},{(\xy{10}+\xy{11})/2});
      \x{22} = ({(\xx{09}+\xx{14})/2},{(\xy{09}+\xy{14})/2});
      coordinate \y{01}, \y{02}, \y{03}, \y{04};
      coordinate \y{05}, \y{06}, \y{07}, \y{08};
      coordinate \y{09}, \y{10}, \y{11}, \y{12};
      coordinate \y{13}, \y{14}, \y{15}, \y{16};
      coordinate \y{17}, \y{18}, \y{19}, \y{20}, \y{21};
      coordinate \y{22}; 
      coordinate \y{23}; 
      \y{01} = ({\xx{01}+6},{\xy{01}+20});
      \y{04} = ({\xx{04}-6},{\xy{04}+20});
      \y{13} = ({\xx{13}-6},{\xy{13}+20});
      \y{16} = ({\xx{16}+6},{\xy{16}+20});
      \y{02} = ({(2*\yx{01}+\yx{04})/3},{(2*\yy{01}+\yy{04})/3});
      \y{14} = ({(2*\yx{13}+\yx{16})/3},{(2*\yy{13}+\yy{16})/3});
      \y{03} = ({(2*\yx{04}+\yx{01})/3},{(2*\yy{04}+\yy{01})/3});
      \y{15} = ({(2*\yx{16}+\yx{13})/3},{(2*\yy{16}+\yy{13})/3});
      \y{05} = ({(2*\yx{01}+\yx{13})/3},{(2*\yy{01}+\yy{13})/3});
      \y{08} = ({(2*\yx{04}+\yx{16})/3},{(2*\yy{04}+\yy{16})/3});
      \y{09} = ({(2*\yx{13}+\yx{01})/3},{(2*\yy{13}+\yy{01})/3});
      \y{12} = ({(2*\yx{16}+\yx{04})/3},{(2*\yy{16}+\yy{04})/3});
      \y{06} = ({(\yx{05}+\yx{08})/2},{(\yy{05}+\yy{08})/2});
      \y{07} = ({(\yx{09}+\yx{12})/2},{(\yy{09}+\yy{12})/2});
      \y{10} = ({(\yx{02}+\yx{14})/2},{(\yy{02}+\yy{14})/2});
      \y{11} = ({(\yx{03}+\yx{15})/2},{(\yy{03}+\yy{15})/2});
      \y{17} = ({(2*\yx{02}+\yx{14})/3},{(2*\yy{02}+\yy{14})/3});
      \y{19} = ({(2*\yx{14}+\yx{02})/3},{(2*\yy{14}+\yy{02})/3});
      \y{18} = ({(2*\yx{03}+\yx{15})/3},{(2*\yy{03}+\yy{15})/3});
      \y{20} = ({(2*\yx{15}+\yx{03})/3},{(2*\yy{15}+\yy{03})/3});
      \y{21} = ({(\yx{10}+\yx{11})/2},{(\yy{10}+\yy{11})/2});
      \y{22} = ({(\yx{09}+\yx{14})/2},{(\yy{09}+\yy{14})/2});
      \x{23} = ({(\xx{13}+\yx{14})/2},{(\xy{13}+\yy{14})/2});
      %
      coordinate \z{17}, \z{10}, \z{19}, \z{07};
      coordinate \z{20}, \z{11}, \z{18}, \z{06};
      coordinate \z{21};
      \z{06} = ({(\xx{06}+\yx{06})/2},{(\xy{06}+\yy{06})/2});
      \z{07} = ({(\xx{07}+\yx{07})/2},{(\xy{07}+\yy{07})/2});
      \z{10} = ({(\xx{10}+\yx{10})/2},{(\xy{10}+\yy{10})/2});
      \z{11} = ({(\xx{11}+\yx{11})/2},{(\xy{11}+\yy{11})/2});
      \z{17} = ({(\xx{17}+\yx{17})/2},{(\xy{17}+\yy{17})/2});
      \z{18} = ({(\xx{18}+\yx{18})/2},{(\xy{18}+\yy{18})/2});
      \z{19} = ({(\xx{19}+\yx{19})/2},{(\xy{19}+\yy{19})/2});
      \z{20} = ({(\xx{20}+\yx{20})/2},{(\xy{20}+\yy{20})/2});
      \z{21} = ({(\xx{21}+\yx{21})/2},{(\xy{21}+\yy{21})/2});
    }
    \begin{tikzpicture}[scale=0.8]
      \draw[dash dot] (\x{01}) -- (\x{04}) -- (\x{16});
      \draw[] (\x{16}) -- (\x{13}) -- (\x{01});
      \draw[dash dot] (\x{02}) -- (\x{14});
      \draw[dash dot] (\x{03}) -- (\x{15});
      \draw[dash dot] (\x{05}) -- (\x{08});
      \draw[dash dot] (\x{09}) -- (\x{12});
      \draw[dash dot] (\x{06}) -- (\x{07});
      \draw[dash dot] (\x{10}) -- (\x{11});
      \draw[] (\y{01}) -- (\y{04}) -- (\y{16}) -- (\y{13}) -- (\y{01});
      \draw[] (\y{02}) -- (\y{14});
      \draw[] (\y{03}) -- (\y{15});
      \draw[] (\y{05}) -- (\y{08});
      \draw[] (\y{09}) -- (\y{12});
      \draw[] (\y{06}) -- (\y{07});
      \draw[] (\y{10}) -- (\y{11});
      \draw[] (\x{01}) -- (\y{01});
      \draw[dash dot] (\x{02}) -- (\y{02});
      \draw[dash dot] (\x{03}) -- (\y{03});
      \draw[dash dot] (\x{04}) -- (\y{04});
      \draw[] (\x{05}) -- (\y{05});
      \draw[dash dot] (\x{06}) -- (\y{06});
      \draw[dash dot] (\x{07}) -- (\y{07});
      \draw[dash dot] (\x{08}) -- (\y{08});
      \draw[] (\x{09}) -- (\y{09});
      \draw[dash dot] (\x{10}) -- (\y{10});
      \draw[dash dot] (\x{11}) -- (\y{11});
      \draw[dash dot] (\x{12}) -- (\y{12});
      \draw[] (\x{13}) -- (\y{13});
      \draw[] (\x{14}) -- (\y{14});
      \draw[] (\x{15}) -- (\y{15});
      \draw[] (\x{16}) -- (\y{16});
      \draw[dash dot] (\x{17}) -- (\y{17});
      \draw[dash dot] (\x{18}) -- (\y{18});
      \draw[dash dot] (\x{19}) -- (\y{19});
      \draw[dash dot] (\x{20}) -- (\y{20});
      \draw[dash dot] (\x{21}) -- (\y{21});
      \draw[<->] (\yx{16}+3,\xy{16}) -- (\yx{16}+3,\yy{16}) node[midway,right]{$I_n$};
      \node[below] at (\x{16}) {$t_{n}$};
      \node[above] at (\y{16}) {$t_{n+1}$};
      \node[below left] at (\x{19}) {$\mathcal{K}$};
      \filldraw[fill=red!30,fill opacity=0.2,draw opacity=0] (\x{09}) --
      (\x{13}) -- (\y{13}) -- (\y{09}) -- (\x{09});
      \filldraw[fill=red!30,fill opacity=0.2,draw opacity=0] (\x{13}) --
      (\x{14}) -- (\y{14}) -- (\y{13}) -- (\x{13});
      \filldraw[fill=red!30,fill opacity=0.2,draw opacity=0] (\x{14}) --
      (\x{19}) -- (\y{19}) -- (\y{14}) -- (\x{14});
      \filldraw[fill=red!30,fill opacity=0.2,draw opacity=0] (\x{19}) --
      (\x{09}) -- (\y{09}) -- (\y{19}) -- (\x{19});
      \draw[very thin,->] (\yx{22}-60,\yy{22})
      node[left]{$K^*$} -- (\yx{22}-5,\yy{22});
      \draw[very thin,->] (\xx{22}-60,\xy{22}-5)
      node[left]{$K_*$} -- (\xx{22}-12,\xy{22}-5) ;
      \draw[dash dot] (\z{17}) -- (\z{10}) --
			(\z{19}) -- (\z{07}) --
			(\z{20}) -- (\z{11}) --
			(\z{18}) -- (\z{06}) -- (\z{17});
		\draw[dash dot] (\z{10}) -- (\z{11});
		\draw[dash dot] (\z{06}) -- (\z{07});
    \end{tikzpicture}
  }
  \caption{Illustration of a moving spatial domain
    $\Omega(t) \subset \mathbb{R}^d$ for $t \in I_n$ resulting in the
    space-time slab $\mathcal{E}^n \subset \mathbb{R}^{(d+1)}$ (with
    $d=2$). Local time-stepping within a space-time slab is featured
    in \cref{fig:localtimestep}. Here $K^*=K^{n+1}$ and $K_*=K^n$.}
  \label{fig:st_dom_ex}
\end{figure}

To define the finite element spaces, we require the mapping
$\Phi_{\mathcal{K}}$ between a fixed reference element
$\widehat{\mathcal{K}}=\del[0]{-1,1}^{d+1}$ and space-time element
$\mathcal{K}\in\mathcal{T}_h$. Following \citet{Georgoulis:thesis} and
\citet{Sudirham:2006}, this mapping
$\Phi_{\mathcal{K}}(\widehat{\mathcal{K}}) = \mathcal{K}$ is
decomposed into two parts. First,
$G_{\mathcal{K}}(\widehat{\mathcal{K}}) = \widetilde{\mathcal{K}}$
denotes the affine mapping defined by
$G_{\mathcal{K}}(\widehat{x})=A_{\mathcal{K}}\widehat{x}+b$, where
$A_{\mathcal{K}}=\mathrm{diag}\;\del[0]{{\delta t_{\mathcal{K}}}/2,
  h_K/2,\dots, h_K/2}$ and $b\in\mathbb{R}^{d+1}$ a constant
translation vector such that the brick
$\widetilde{\mathcal{K}}:=(0,h_K)^d\times(0,\delta t_{\mathcal{K}})$,
see \cref{fig:st_element_construction}. In the following, $h_K$ is
used to denote the spatial mesh size and $\delta t_{\mathcal{K}}$ the
time-step. We then define
$\Phi_{\mathcal{K}}:=\phi_{\mathcal{K}}\circ G_{\mathcal{K}}$, where
$\phi_{\mathcal{K}}$ is a diffeomorphism such that
$\phi_{\mathcal{K}}(\widetilde{\mathcal{K}}) = \mathcal{K}$ (see
\cref{fig:st_element_construction}). Note that $G_{\mathcal{K}}$ sets
the size of the element $\mathcal{K}$ while $\phi_{\mathcal{K}}$ sets
its shape. Following \citet{Georgoulis:thesis} and
\citet{Sudirham:2006}, we assume that $\phi_{\mathcal{K}}$ is close to
the identity, i.e., we will assume that $\phi_{\mathcal{K}}$
satisfies:
\begin{equation}
  \label{eq:diffeom_regular}
  c^{-1}
  \leq
  \envert[0]{\det J_{\phi_{\mathcal{K}}}}
  \leq
  c
  ,
  \quad
  \norm[0]{
    \del[0]{J_{\phi_{\mathcal{K}}}}_{ij}
  }_{L^\infty(\widetilde{\mathcal{K}})}
  \leq
  c
  \quad
  0\leq i,j\leq d,
  \quad
  \forall\mathcal{K}\in\mathcal{T}_h,
\end{equation}
where $c$ is a generic constant independent of $h_K$,
$\delta t_{\mathcal{K}}$, $\varepsilon$, and $T$, where
$J_{\phi_{\mathcal{K}}} \in \mathbb{R}^{(n+1)\times (n+1)}$ is the
Jacobian of the diffeomorphism $\phi_{\mathcal{K}}$, and where the
index $0$ denotes the time coordinate direction. Since $t$ only
depends on $\widetilde{t}$,
\begin{equation*}
  \norm[0]{
    \partial_{\widetilde{x}_k} t
  }_{L^\infty(\widetilde{\mathcal{K}})}
  =
  \norm[0]{
    \del{J_{\phi_\mathcal{K}}}_{0k}
  }_{L^\infty(\widetilde{\mathcal{K}})}
  =0,
  \quad
  1\leq k\leq d,
  \quad
  \forall\mathcal{K}\in\mathcal{T}_h.
\end{equation*}
For the inverse of $J_{\phi_\mathcal{K}}$, let
$\det J_{\phi_\mathcal{K}\setminus{mn}}$ denote the $\del[0]{m,n}$
minor of $J_{\phi_\mathcal{K}}$. We will assume that:
\begin{equation}
  \label{eq:diffeom_regular_inv}
  c^{-1}
  \leq
  \envert[0]{\det J^{-1}_{\phi_\mathcal{K}}}
  \leq
  c
  ,
  \quad
  \norm[0]{
    \det
    J_{\phi_\mathcal{K}\setminus{mn}}
  }_{L^\infty(\widetilde{\mathcal{K}})}
  \leq
  c
  ,
  \quad
  \forall\mathcal{K}\in\mathcal{T}_h.
\end{equation}
Let $F_{\mathcal{Q}}^j$ be a $\mathcal{Q}$-face where
$\widetilde{x}_j$ is fixed in its affine domain. The parametrization
of $F_{\mathcal{Q}}^j$, obtained from the restriction of
$\phi_{\mathcal{K}}$ to the boundary of $\widetilde{\mathcal{K}}$
where $\widetilde{x}_j$ is fixed, is denoted by
$\phi_{F_{\mathcal{Q}}}$. Then, \citep[see][Theorem 21.3 and
Definition on page 189]{Munkres:book},
\begin{equation}
  \label{eq:k_surface}
  \int_{F_{\mathcal{Q}}^j} f(x)\dif s
  =
  \int_{\widetilde{F}_{\mathcal{Q}}^j}
  f\del{\phi_{F_{\mathcal{Q}}}(\widetilde{x})}
  \del[1]{
    \det\del[1]{
      \del[0]{
        J_{\phi_\mathcal{K}}^{j}
      }^\intercal
      J_{\phi_\mathcal{K}}^{j}
    }
  }^{1/2}
  \dif \widetilde{s},
\end{equation}
where $J_{\phi_\mathcal{K}}^{j} \in \mathbb{R}^{(n+1)\times n}$ is
obtained by removing the $j^{\text{th}}$ column vector from
$J_{\phi_{\mathcal{K}}}$. We will assume that
\begin{equation}
  \label{eq:diffeom_regular_d}
  c^{-1}
  \leq
  \del[1]{
    \det\del[0]{
      \del[0]{
        J_{\phi_\mathcal{K}}^{i}
      }^\intercal
      J_{\phi_\mathcal{K}}^{i}
    }
  }^{1/2}
  \leq
  c
  ,
  \quad
  0\leq i\leq d.
\end{equation}

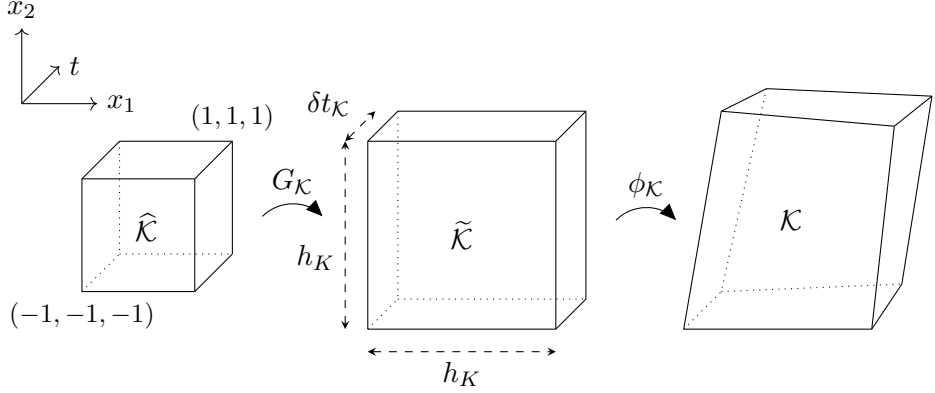
\begin{figure}[tbp]
  \centering
  \tikzmath{
    \x0 = 0.2; \y0 = -1.5;
    \x1 = \x0 + 1; \y1 =\y0;
    \x2 = \x0 + 0.5; \y2 =\y0 + 0.5;
    \x3 = \x0; \y3 =\y0 + 1;
    \a0 = 1; \b0 = -4;
    \c0 = 1.5; \c1 = 0.5; 
    \a1 = \a0 + \c0; \b1 = \b0;
    \a2 = \a0 + \c0; \b2 = \b0 + \c0;
    \a3 = \a0; \b3 = \b0 + \c0;
    \a4 = \a0 + \c1; \b4 = \b0 + \c1;
    \a5 = \a4 + \c0; \b5 = \b4;
    \a6 = \a4 + \c0; \b6 = \b4 + \c0;
    \a7 = \a4; \b7 = \b4 + \c0;
    \a8 = (3*\a0 + 4*\a1)/7; \b8 = (3*\b0 + 4*\b3)/7;
    \s0 = 4.8; \t0 = -4.5;
    \r0 = 2.5; \r1 = 2.5; \r2 = 0.4; 
    \s1 = \s0 + \r0; \t1 = \t0;
    \s2 = \s0 + \r0; \t2 = \t0 + \r1;
    \s3 = \s0; \t3 = \t0 + \r1;
    \s4 = \s0 + \r2; \t4 = \t0 + \r2;
    \s5 = \s4 + \r0; \t5 = \t4;
    \s6 = \s4 + \r0; \t6 = \t4 + \r1;
    \s7 = \s4; \t7 = \t4 + \r1;
    \s8 = (\s0 + \s1)/2; \t8 = (\t0 + \t3)/2;
    \m0 = 9; \n0 = -4.5;
    \l1 = 2.5; \h1 = 0;
    \l2 = 0.3; \h2 = 2.7;
    \l3 = -2.3; \h3 = 0.2;
    \l4 = 0.5; \h4 = 0.5;
    \l5 = 2.4; \h5 = 0.1;
    \l6 = 0.4; \h6 = 2.5;
    \l7 = -2.2; \h7 = 0.1;
    \m1 = \m0 + \l1; \n1 = \n0 + \h1;
    \m2 = \m1 + \l2; \n2 = \n1 + \h2;
    \m3 = \m2 + \l3; \n3 = \n2 + \h3;
    \m4 = \m0 + \l4; \n4 = \n0 + \h4;
    \m5 = \m4 + \l5; \n5 = \n4 + \h5;
    \m6 = \m5 + \l6; \n6 = \n5 + \h6;
    \m7 = \m6 + \l7; \n7 = \n6 + \h7;
    \m8 = (3*\m0 + 4*\m1)/7; \n8 = (\n0 + \n3)/2;
    \p1 = \a5 + 0.4; \q1 = (\b1 + \b6)/2;
    \p2 = \s0 - 0.6; \q2 = (\b1 + \b6)/2;
    \p3 = \s5 + 0.4; \q3 = (\t1 + \t6)/2;
    \p4 = \m0 - 0.1; \q4 = (\t1 + \t6)/2;
  }
  \begin{tikzpicture}
    \draw[->] (\x0,\y0) -- (\x1,\y1) node[right] {$x_1$};
    \draw[->] (\x0,\y0) -- (\x2,\y2) node[right] {$t$};
    \draw[->] (\x0,\y0) -- (\x3,\y3) node[above] {$x_2$};
    \draw[] (\a0,\b0) -- (\a1,\b1);
    \draw[] (\a1,\b1) -- (\a2,\b2);
    \draw[] (\a2,\b2) -- (\a3,\b3);
    \draw[] (\a3,\b3) -- (\a0,\b0) node[below] {\small$\del{-1,-1,-1}$};
    \draw[dotted] (\a4,\b4) -- (\a5,\b5);
    \draw[] (\a5,\b5) -- (\a6,\b6) node[above] {\small$\del{1,1,1}$};;
    \draw[] (\a6,\b6) -- (\a7,\b7);
    \draw[dotted] (\a7,\b7) -- (\a4,\b4);
    \draw[] (\a1,\b1) -- (\a5,\b5);
    \draw[] (\a2,\b2) -- (\a6,\b6);
    \draw[] (\a3,\b3) -- (\a7,\b7);
    \draw[dotted] (\a0,\b0) -- (\a4,\b4);
    \node at (\a8, \b8) {$\widehat{\mathcal{K}}$};
    \draw[] (\s0,\t0) -- (\s1,\t1);
    \draw[] (\s1,\t1) -- (\s2,\t2);
    \draw[] (\s2,\t2) -- (\s3,\t3);
    \draw[] (\s3,\t3) -- (\s0,\t0);
    \draw[dotted] (\s4,\t4) -- (\s5,\t5);
    \draw[] (\s5,\t5) -- (\s6,\t6);
    \draw[] (\s6,\t6) -- (\s7,\t7);
    \draw[dotted] (\s7,\t7) -- (\s4,\t4);
    \draw[] (\s1,\t1) -- (\s5,\t5);
    \draw[] (\s2,\t2) -- (\s6,\t6);
    \draw[] (\s3,\t3) -- (\s7,\t7);
    \draw[dotted] (\s0,\t0) -- (\s4,\t4);
    \node at (\s8, \t8) {$\widetilde{\mathcal{K}}$};
    \draw[stealth-stealth,dashed] (\s0, \t0-0.3) -- node[below] {$h_K$} (\s1, \t1-0.3);
    \draw[stealth-stealth,dashed] (\s0-0.3, \t0) -- node[below left] {$h_K$} (\s3-0.3, \t3);
    \draw[stealth-stealth,dashed] (\s3-0.3, \t3) -- node[above left] {${\delta t_{\mathcal{K}}}$} (\s7-0.3, \t7);
    \draw[] (\m0,\n0) -- (\m1,\n1);
    \draw[] (\m1,\n1) -- (\m2,\n2);
    \draw[] (\m2,\n2) -- (\m3,\n3);
    \draw[] (\m3,\n3) -- (\m0,\n0);
    \draw[dotted] (\m4,\n4) -- (\m5,\n5);
    \draw[] (\m5,\n5) -- (\m6,\n6);
    \draw[] (\m6,\n6) -- (\m7,\n7);
    \draw[dotted] (\m7,\n7) -- (\m4,\n4);
    \draw[] (\m1,\n1) -- (\m5,\n5);
    \draw[] (\m2,\n2) -- (\m6,\n6);
    \draw[] (\m3,\n3) -- (\m7,\n7);
    \draw[dotted] (\m0,\n0) -- (\m4,\n4);
    \node at (\m8, \n8) {${\mathcal{K}}$};
    \draw[>=triangle 45,->] (\p1, \q1) to [out=45, in=135]
    node[above] {${G}_{\mathcal{K}}$} (\p2, \q2);
    \draw[>=triangle 45,->] (\p3, \q3) to [out=45, in=135]
    node[above] {$\phi_\mathcal{K}$} (\p4, \q4);
  \end{tikzpicture}
  \caption{ Construction of the space-time element $\mathcal{K}$
    through an affine mapping
    ${G}_{\mathcal{K}}:\widehat{\mathcal{K}}\rightarrow\widetilde{\mathcal{K}}$
    and a diffeomorphism
    $\phi_\mathcal{K}:\widetilde{\mathcal{K}}\rightarrow\mathcal{K}$
    \citep{Sudirham:2006}. Note that the front and back faces of
    $\mathcal{K}$ have constant $t$-coordinate and hence are parallel
    to each other.}
  \label{fig:st_element_construction}
\end{figure}

We will assume that $\delta t_{\mathcal{K}} \leq h_K$. A general bound
$\delta t_{\mathcal{K}}\leq\mathcal{O}\del[0]{h_K}$ poses no
difficulty to our analysis but is avoided for the ease of
presentation. Our spatial mesh on $\Omega(t)$, for $0 \le t \le T$, is
shape-regular and we allow 1-irregularly refined space-time elements,
i.e., we allow at most one level of refinement difference between
neighboring space-time elements. It will also be useful to introduce a
second time-step, $\Delta t_{\mathcal{K}}$, which is the time-step
size of the space-time slab to which $\mathcal{K}$ belongs. When
performing uniform time-stepping, the ratio
$\Delta t_{\mathcal{K}}/\delta t_{\mathcal{K}} \equiv 1$ for all
$\mathcal{K} \in \mathcal{T}_h$. When using local time-stepping, we
will assume that the ratio between the time-step of a space-time slab
and the minimum local time-step size in that slab is bounded:
$\Delta t_{\mathcal{K}} / \delta t_{\mathcal{K}} \leq c$ for all
$\mathcal{K}\in\mathcal{T}_h$.

\subsection{Spaces and useful inequalities}
\label{ss:approxspaceswf}

Let $\partial^\alpha_xv$, with $\alpha$ a multi-index, be the weak
derivative of $v$ and let
$H^s({U}):=\{v\in L^2({U}) : \partial^\alpha_x v\in L^2({U}) \text{
  for } \envert[0]{\alpha}\leq s \}$, where $s$ is a nonnegative integer
and ${U}\subset\mathbb{R}^r$ is an open domain with
$x:=\del[0]{x_1,\dots,x_r}$ denoting the coordinates of
$\mathbb{R}^r$. The norm of $H^s({U})$ is defined by
$\norm[0]{v}^2_{H^s({U})} := \sum_{\envert[0]{\alpha}\leq
  s}\norm[0]{\partial^\alpha_x v}_{{U}}^2$, where
$\norm[0]{\cdot}_{{U}}$ is the usual $L^2$-norm on ${U}$.

We also require anisotropic Sobolev spaces. Following
\citet{Sudirham:2006} we only consider anisotropy between spatial and
temporal variables with no anisotropy between the spatial
variables. As such, let $s_s$ and $s_t$ denote the spatial and
temporal Sobolev indices, respectively. For
$\alpha_t,\alpha_{s_i}\geq 0$, $1\leq i\leq d$, the anisotropic
Sobolev space of order $\del[0]{s_t,s_s}$ is defined on an open domain
${U}\subset\mathbb{R}^{d+1}$ by \citep[see][]{Georgoulis:thesis}:
\begin{equation*}
  H^{\del[0]{s_t,s_s}}({U})
  :=
  \cbr[0]{
    v\in L^2({U})
    :
    \partial_t^{\alpha_t}
    \partial_x^{\alpha_s}
    v
    \in L^2({U})
    \;
    \text{for}
    \;
    \alpha_t\leq s_t
    ,
    \envert[0]{\alpha_s}\leq s_s
  },
\end{equation*}
where $\alpha_s = \del[0]{\alpha_{s_1},\dots,\alpha_{s_d}}$ and
$x:=\del[0]{x_1,\dots,x_d}$ denotes the spatial coordinates. The
anisotropic Sobolev norm reads
$\norm[0]{v}^2_{H^{\del{s_t,s_s}}({U})} := \sum_{\alpha_t\leq
  s_t,\,\envert[0]{\alpha_s}\leq s_s}
\norm[0]{\partial_t^{\alpha_t}\partial_x^{\alpha_s}v}^2_{U}$.

For the HDG method, we require the following finite element spaces
\begin{align*}
  V_h^{\del{p_t,p_s}}
  &:=
    \cbr[0]{
    v_h
    \in
    L^2\del[0]{\mathcal{E}}:
    v_h|_{\mathcal{K}}
    \circ
    \phi_{\mathcal{K}}
    \circ
    {G}_{\mathcal{K}}
    \in
    Q_{\del[0]{p_t,p_s}}\del[0]{\widehat{\mathcal{K}}}
    \quad
    \forall\mathcal{K}\in\mathcal{T}_h
    },
  \\
  M_h^{\del{p_t,p_s}}
  &:=
    \{
    \mu_h
    \in
    L^2\del[0]{\Gamma}:
    \mu_h|_F
    \circ
    \phi_{\mathcal{K}}
    \circ
    {G}_{\mathcal{K}}
    \in
    Q_{\del{p_t,p_s}}\del[0]{\widehat{F}}
    \quad
    \forall F\in\mathcal{F}_h,
    \mu_h=0 \text{ on } \partial\mathcal{E}_D\},
\end{align*}
where $Q^{\del{p_t,p_s}}(U)$  denotes the set of all tensor
product polynomials of degree $p_t$ in the temporal direction and
$p_s$ in each spatial direction on a domain $U$. Furthermore, we
define $\boldsymbol{V}_h=V_h^{\del{p_t,p_s}}\times M_h^{\del{p_t,p_s}}$.
We will denote the pairs $(v, \mu) \in \boldsymbol{V}_h$ and $(u,
\lambda) \in \boldsymbol{V}_h$ as $\boldsymbol{v} = (v, \mu)$ and
$\boldsymbol{u} = (u, \lambda)$. Furthermore, we denote by
$\sbr{\boldsymbol{v}}$ the HDG jump $\del{v-\mu}$.

Adapting \cite[Corollaries 3.49, 3.54]{Georgoulis:thesis} to the
space-time context, specifically taking into account the spatial mesh
size $h_K$ and time step $\delta t_{\mathcal{K}}$ of a space-time
element $\mathcal{K} \in \mathcal{T}_h$, we have the following
anisotropic inverse and trace inequalities, which hold for all
$v_h \in V_h$:
\begin{subequations}
  \label{eq:eg_inv}
  \begin{align}
    \norm[0]{\partial_tv_h}_\mathcal{K}
    &\leq
      c
      \del[0]{{\delta t_{\mathcal{K}}}^{-1}+h_K^{-1}}
      \norm[0]{v_h}_\mathcal{K},
      \label{eq:eg_inv_1}
    \\
    \norm[0]{\overline{\nabla}v_h}_{\mathcal{K}}
    &\leq
      c
      h_K^{-1}
      \norm[0]{v_h}_\mathcal{K},
      \label{eq:eg_inv_2}
    \\
    \norm[0]{v_h}_{\mathcal{Q}_\mathcal{K}}
    &\leq
      c_{\star}
      h_K^{-1/2}
      \norm[0]{v_h}_\mathcal{K},
      \label{eq:eg_inv_3}
    \\
    \norm[0]{v_h}_{\partial\mathcal{K}}
    &\leq
      c
      \del[0]{{\delta t_{\mathcal{K}}}^{-1/2} + h_K^{-1/2}}
      \norm[0]{v_h}_\mathcal{K},
      \label{eq:eg_inv_4}
  \end{align}
\end{subequations}
where $c_{\star}$ is a constant independent of $h_K$,
$\delta t_{\mathcal{K}}$, $\varepsilon$, and $T$. (We distinguish $c_{\star}$
from $c$ to prove \cref{lem:inf_sup_v_norm_weighted}.) The following
lemma introduces an additional inequality.

\begin{lemma}
  \label{lem:eg_inv}
  Let $\mathcal{K} \in \mathcal{T}_h$ be a space-time element.  For
  all $\mu_h\in M_h$,
  \begin{equation}
    \label{eq:eg_inv_low_d_1_F}
    \norm[0]{\partial_t\mu_h}_{F_{\mathcal{Q}}}
    \leq
    c
    \del[0]{{\delta t_{\mathcal{K}}}^{-1}+h_K^{-1}}
    \norm[0]{\mu_h}_{F_{\mathcal{Q}}}.
  \end{equation}
\end{lemma}
\begin{proof}
  The $d$-dimensional hypersurface $F_{\mathcal{Q}}$ is embedded in
  $\mathbb{R}^{d+1}$ and in general it may be curved. Therefore, we
  cannot use \cref{eq:eg_inv_1} directly to conclude
  \cref{eq:eg_inv_low_d_1_F}. Instead, we first map $F_{\mathcal{Q}}$
  to the affine domain. For this, let
  $\phi_{F_{\mathcal{Q}}}(\widetilde{F}_{\mathcal{Q}}) =
  F_{\mathcal{Q}}$, i.e., the transformation of a face from the affine
  domain to the physical domain. We then observe that one of the
  spatial coordinates, which is denoted by $\widetilde{x}_j$ without
  loss of generality, of $\widetilde{F}_\mathcal{Q}$ is fixed. This
  means we can view $\widetilde{F}_\mathcal{Q}$ in the
  $\mathbb{R}^{d}$ domain with coordinates
  $\del[0]{\widetilde{t}, \widetilde{x}_1, \dots, \widetilde{x}_{j-1},
    \widetilde{x}_{j+1}, \dots, \widetilde{x}_d}$ and apply the
  $d$-dimensional versions of \cref{eq:eg_inv_1} and
  \cref{eq:eg_inv_2} to $\widetilde{F}_\mathcal{Q}$:
  \begin{equation}
    \label{eq:some_eq_31}
    \norm[0]{
      \partial_{\widetilde{t}}\widetilde{\mu}_h
    }_{\widetilde{F}_\mathcal{Q}}
    \leq
    c
    \del[0]{{\delta t_{\mathcal{K}}}^{-1}+h_K^{-1}}
    \norm[0]{\widetilde{\mu}_h}_{\widetilde{F}_\mathcal{Q}},
    \quad
    \norm[0]{
      \widetilde{\overline{\nabla}}\widetilde{\mu}_h
    }_{\widetilde{F}_\mathcal{Q}}
    \leq
    c
    h_K^{-1}
    \norm[0]{\widetilde{\mu}_h}_{\widetilde{F}_\mathcal{Q}}.
  \end{equation}
  Using the mapping $\phi_{F_\mathcal{Q}}$,
  \cref{eq:k_surface,eq:diffeom_regular_d},
  \begin{equation*}
    \begin{split}
      \norm[0]{
        \partial_{t}{\mu}_h
      }_{F_\mathcal{Q}}^2
      &=
      \int_{\widetilde{F}_\mathcal{Q}}
      \sbr[1]{
        \del[1]{\partial_t\del[0]{\widetilde{\mu}_h\circ\phi_{F_\mathcal{Q}}^{-1}}}
        \circ\phi_{F_\mathcal{Q}}
      }^2
      \sbr[1]{
        \det\del[1]{
          \del[0]{J_{\phi_\mathcal{K}}^{j}}^\intercal
          J_{\phi_\mathcal{K}}^{j}
        }
      }^{1/2}
      \dif\widetilde{s}
      \\
      &
      \leq
      c
      \int_{\widetilde{F}_\mathcal{Q}}
      \sbr[1]{
        \del[1]{\partial_t\del[0]{\widetilde{\mu}_h\circ\phi_{F_\mathcal{Q}}^{-1}}}
        \circ\phi_{F_\mathcal{Q}}
      }^2
      \dif\widetilde{s}.
    \end{split}
  \end{equation*}
  By the chain rule,
  \begin{equation*}
    \partial_t\del[0]{\widetilde{\mu}_h\circ\phi_{F_\mathcal{Q}}^{-1}}
    =
    \del[1]{
      \del[0]{\partial_{\widetilde{t}}\widetilde{\mu}_h}
      \circ \phi_{F_\mathcal{Q}}^{-1}
    }
    \frac{\partial\widetilde{t}}{\partial{t}}
    +
    \sum_{1\leq i\leq d, i\neq j}
    \del[1]{
      \del[0]{\partial_{\widetilde{x}_i}\widetilde{\mu}_h}
      \circ \phi_{F_\mathcal{Q}}^{-1}
    }
    \frac{\partial\widetilde{x}_i}{\partial{t}}.
  \end{equation*}
  We note that $\frac{\partial\widetilde{x}_i}{\partial t}$ is the
  $\del{i,0}$-element of $J_{\phi_\mathcal{K}}^{-1}$ which equals
  $\del[0]{-1}^i\det J_{\phi_{\mathcal{K}\setminus{i0}}} / \det
  J_{\phi_\mathcal{K}}$. Similarly,
  $\frac{\partial\widetilde{t}}{\partial t}$ corresponds to the
  $\del{0,0}$-element of $J_{\phi_\mathcal{K}}^{-1}$ which equals
  $\det J_{\phi_{\mathcal{K}\setminus{00}}} / \det
  J_{\phi_\mathcal{K}}$. Now using \cref{eq:diffeom_regular},
  \cref{eq:diffeom_regular_inv}, \cref{eq:some_eq_31}, definition
  \cref{eq:k_surface}, and \cref{eq:diffeom_regular_d} we find that:
  \begin{equation*}
    \begin{split}
      \norm[0]{
        \partial_{t}{\mu}_h
      }_{F_\mathcal{Q}}^2
      &\leq
      c
      \del{
        \int_{\widetilde{F}_\mathcal{Q}}
        \del[0]{
          \partial_{\widetilde{t}}
          \widetilde{\mu}_h
        }^2
        \dif\widetilde{s}
        +
        \int_{\widetilde{F}_\mathcal{Q}}
        \del[0]{
          \widetilde{ \overline{\nabla}}
          \widetilde{\mu}_h
        }^2
        \dif\widetilde{s}
      }
      \\
      &\leq
      c
      \del[0]{
        {\delta t_{\mathcal{K}}}^{-2}
        +
        h_K^{-2}
      }
      \int_{\widetilde{F}_\mathcal{Q}}
      \widetilde{\mu}_h^2
      \sbr[1]{
        \det\del[1]{
          \del[0]{J_{\phi_\mathcal{K}}^{j}}^\intercal
          J_{\phi_\mathcal{K}}^{j}
        }
      }^{1/2}
      \dif\widetilde{s}
      \le
      c
      \del[0]{
        {\delta t_{\mathcal{K}}}^{-2}
        +
        h_K^{-2}
      }
      \norm[0]{ \mu_h }_{F_\mathcal{Q}}^2,
    \end{split}
  \end{equation*}
  which is \cref{eq:eg_inv_low_d_1_F}.
\end{proof}

To end this section we introduce $\boldsymbol{V} := V \times M$, where
$V := \cbr[0]{v \in H^1(\mathcal{E})\,|\, v|_{\partial\mathcal{E}_D} =
  0} \cap H^2(\mathcal{E})$ and $M$ its trace space, and define the
extended function space $\boldsymbol{V}(h) := V(h) \times M(h)$ where
$V(h) := V_h^{(p_t,p_s)}+V$ and $M(h) := M_h^{(p_t,p_s)}+M$. We will
require the following three norms on $\boldsymbol{V}(h)$:
\begin{subequations}
  \label{eq:norm-defs}
  \begin{align}
    \tnorm{\boldsymbol{v}}_v^2
    :=
    &
      \sum_{\mathcal{K}\in\mathcal{T}_h}
      \norm[0]{v}_\mathcal{K}^2
      +
      \sum_{\mathcal{K}\in\mathcal{T}_h}
      \norm[0]{ \envert[0]{ \beta_s - \tfrac{1}{2} \beta\cdot{n} }^{1/2} \sbr{\boldsymbol{v}}
      }_{\partial\mathcal{K}}^2
      +
      \sum_{F\in\partial\mathcal{E}_N}
      \norm[0]{
      \envert[0]{
      \tfrac{1}{2}
      \beta\cdot{n}
      }^{1/2}
      \mu
      }_F^2
      \label{eq:v_norm}
    \\
    &
      +
      \sum_{\mathcal{K}\in\mathcal{T}_h}
      \varepsilon\norm[0]{\overline{\nabla}v}_\mathcal{K}^2
      +
      \sum_{\mathcal{K}\in\mathcal{T}_h}
      \varepsilon h_K^{-1}
      \norm[0]{\sbr[0]{\boldsymbol{v}}}_{\mathcal{Q}_\mathcal{K}}^2,
      \nonumber
    \\
    \tnorm{\boldsymbol{v}}_s^2
    :=
    &
      \tnorm{\boldsymbol{v}}_v^2
      +
      \sum_{\mathcal{K}\in\mathcal{T}_h}
      \tau_\varepsilon
      \norm[0]{\partial_t v}^2_\mathcal{K},
      \label{eq:s_norm}
    \\
    \tnorm{\boldsymbol{v}}_{ss}^2
    :=
    &
      \tnorm{\boldsymbol{v}}_s^2
      +
      \norm{{v}}_{sd}^2
      :=
      \tnorm{\boldsymbol{v}}_s^2
      +
      \sum_{\mathcal{K}\in\mathcal{T}_h}
      \tfrac{\delta t_{\mathcal{K}}h_K^2}
      {\delta t_{\mathcal{K}}+h_K}
      \norm[0]{\Pi_h\del[0]{\beta\cdot\nabla v}}^2_\mathcal{K},
      \label{eq:sd_norm}
  \end{align}
\end{subequations}
where $\Pi_h$ denotes the $L^2$-projection onto $V_h^{(p_t,p_s)}$ and
the parameter $\tau_{\varepsilon}$ in the definition of
$\tnorm{\boldsymbol{v}}_s$ depends on the size of the space-time
element compared to the diffusion parameter $\varepsilon$:
\begin{equation*}
  \tau_{\varepsilon} :=
  \Delta t_{\mathcal{K}}\tilde{\varepsilon}
  :=
  \begin{cases}
    \Delta t_{\mathcal{K}}
    &\text{if } \mathcal{K} \in \mathcal{T}_h^d
    :=
    \cbr[1]{
      \mathcal{K}\in\mathcal{T}_h
      |
      {\delta t_{\mathcal{K}}}\leq h_K\leq\varepsilon
    },
    \\
    \Delta t_{\mathcal{K}}\varepsilon^{1/2}
    &\text{if } \mathcal{K} \in \mathcal{T}_h^x
    :=
    \cbr[1]{
      \mathcal{K}\in\mathcal{T}_h
      |
      {\delta t_{\mathcal{K}}}\leq \varepsilon<h_K
    },
    \\
    \Delta t_{\mathcal{K}}\varepsilon
    &\text{if } \mathcal{K} \in \mathcal{T}_h^c
    :=
    \cbr[1]{
      \mathcal{K}\in\mathcal{T}_h
      |
      \varepsilon<{\delta t_{\mathcal{K}}}\leq h_K
    }.
  \end{cases}
\end{equation*}
Finally, $\beta_s := \sup_{(x,t)\in F}|\beta\cdot n|$, for
$F \subset \partial \mathcal{K}$. It is useful to remark that
\begin{equation}
  \label{eq:betasinfmax}
  \inf_{(x,t) \in F}
  (\beta_s  - \tfrac{1}{2}\beta\cdot n)
  \ge
  \tfrac{1}{2}
  \max_{(x,t)\in F}|\beta\cdot n|
  \qquad
  \forall F \in \partial\mathcal{K},\ \forall \mathcal{K} \in \mathcal{T}_h.
\end{equation}

\subsection{The discretization}
\label{ss:discretization}

Let $u,v \in \sbr[0]{L^2(U)}^r$ for $1 \le r \le d+1$. We will write
$(u, v)_U = \int_U u\cdot v \dif x$ if $U \subset \mathbb{R}^{d+1}$
and $\langle u, v \rangle_U = \int_U u\cdot v \dif s$ if
$U \subset \mathbb{R}^{d}$. Furthermore, we define
$(u, v)_{\mathcal{T}_h} := \sum_{\mathcal{K} \in \mathcal{T}_h}(u,
v)_{\mathcal{K}}$,
$\langle u, v \rangle_{\partial \mathcal{T}_h} := \sum_{\mathcal{K}
  \in \mathcal{T}_h} \langle u, v \rangle_{\partial \mathcal{K}}$,
$\langle u, v \rangle_{\mathcal{Q}_h} := \sum_{\mathcal{K} \in
  \mathcal{T}_h} \langle u, v \rangle_{\mathcal{Q}_{\mathcal{K}}}$,
and
$\langle u, v \rangle_{\partial\mathcal{E}_N} := \sum_{F \in
  \mathcal{F}_h^b \cap \partial\mathcal{E}_N} \langle u, v
\rangle_{F}$.

The space-time HDG method for \cref{eq:st_adr} is given by: Find
$\boldsymbol{u}_h\in \boldsymbol{V}_h$ such that
\begin{equation}
  \label{eq:st_hdg_adr_compact}
  a_h\del[0]{
    \boldsymbol{u}_h
    ,
    \boldsymbol{v}_h
  }
  =
  \del[0]{f,v_h}_{\mathcal{T}_h}
  +
  \langle
  g,
  \mu_h
  \rangle_{\partial\mathcal{E}_N}
  \quad
  \forall\boldsymbol{v}_h\in \boldsymbol{V}_h,
\end{equation}
with
$a_h(\boldsymbol{u}_h, \boldsymbol{v}_h) := a_{h,d}(\boldsymbol{u}_h,
\boldsymbol{v}_h) + a_{h,c}(\boldsymbol{u}_h, \boldsymbol{v}_h)$ and
where
\begin{align*}
  a_{h,d}\del[0]{
  \boldsymbol{u} , \boldsymbol{v}
  }
  &
    :=
    \del[0]{\varepsilon\overline{\nabla}u,\overline{\nabla}v}_{\mathcal{T}_h}
    +
    \langle
    \varepsilon\alpha h_{K}^{-1}
    \sbr[0]{\boldsymbol{u}}
    ,
    \sbr[0]{\boldsymbol{v}}
    \rangle_{\mathcal{Q}_h}
    -
    \langle
    \varepsilon\sbr[0]{\boldsymbol{u}}
    ,
    \overline{\nabla}_{\bar{{n}}}v
    \rangle_{\mathcal{Q}_h}
    -
    \langle
    \varepsilon\overline{\nabla}_{\bar{{n}}}u,
    \sbr[0]{\boldsymbol{v}}
    \rangle_{\mathcal{Q}_h},
  \\
  a_{h,c}\del[0]{
  \boldsymbol{u} , \boldsymbol{v}
  }
  &
    :=
    -
    \del[0]{{\beta}u,\nabla v}_{\mathcal{T}_h}
    +
    \langle
    \zeta^+ \beta \cdot n \lambda, \mu
    \rangle_{\partial\mathcal{E}_N}
    +
    \langle
    \del[0]{{\beta}\cdot{n}}
    \lambda
    +
    \beta_s\sbr{\boldsymbol{u}}
    ,
    \sbr{\boldsymbol{v}}
    \rangle_{\partial\mathcal{T}_h}.
\end{align*}
Here $\alpha>0$ is a penalty parameter and $\zeta^+$ denotes the
outflow boundary indicator on a facet. The following boundedness
result, which follows from the Cauchy--Schwarz inequality and
\cref{eq:eg_inv_3}, will be useful in the proof of stability in
\cref{s:stability}:
\begin{equation}
  \label{eq:bnd_v}
  \envert[0]{a_{h,d}(\boldsymbol{u}_h,\boldsymbol{v}_h)}
  \leq
  c
  \tnorm{\boldsymbol{u}_h}_v
  \tnorm{\boldsymbol{v}_h}_v
  \qquad
  \forall \boldsymbol{u}_h,\boldsymbol{v}_h\in \boldsymbol{V}_h.
\end{equation}

\section{Stability}
\label{s:stability}

The main goal of this section is to prove \cref{thm:supginfsup}
which states stability of the space-time HDG method for the
advection-diffusion equation with respect to a norm that measures
the streamline derivative, i.e., $\tnorm{\cdot}_{ss}$. We will
prove that this result is robust with respect to $\varepsilon$. A
similar result for the stationary problem is shown in
\citet[Theorem 4.6]{Ayuso:2009}. For this and following sections,
$c_T$ denotes a constant independent of $h_K$, $\delta
t_{\mathcal{K}}$, and $\varepsilon$, but linear in $T$.

\begin{theorem}
  \label{thm:supginfsup}
  There exists ${\delta t}_0$, independent of $\varepsilon$ and
  $T$, such that when $\delta t_{\mathcal{K}} \le \min(h_K,
  \delta t_0)$ on all $\mathcal{K}\in\mathcal{T}_h$, and for all
  $\boldsymbol{w}_h\in\boldsymbol{V}_h$,
  \begin{equation}
    \label{eq:supginfsup}
    c_T^{-1}
    \tnorm{\boldsymbol{w}_h}_{ss}
    \leq
    \sup_{\boldsymbol{v}_h\in \boldsymbol{V}_h}
    \frac{
      a_h(\boldsymbol{w}_h,\boldsymbol{v}_h)
    }{
      \tnorm{\boldsymbol{v}_h}_s
    }.
  \end{equation}
\end{theorem}

The following two inf-sup conditions with respect to, respectively,
$\tnorm{\cdot}_v$ and $\tnorm{\cdot}_s$, and which hold under the same
conditions as \cref{thm:supginfsup}, are used to prove
\cref{thm:supginfsup}:
\begin{subequations}
  \begin{alignat}{2}
    \label{eq:new_inf_sup_v_norm}
    c_T^{-1}
    \tnorm{\boldsymbol{w}_h}_v
    \leq &
    \sup_{\boldsymbol{v}_h\in \boldsymbol{V}_h}
    \frac{
      a_h(\boldsymbol{w}_h,\boldsymbol{v}_h)
    }{
      \tnorm{\boldsymbol{v}_h}_v
    }
    &&\qquad \forall \boldsymbol{w}_h\in \boldsymbol{V}_h,
    \\
    \label{eq:new_inf_sup_s_norm}
    c_T^{-1}
    \tnorm{\boldsymbol{w}_h}_s
    \leq &
    \sup_{\boldsymbol{v}_h\in \boldsymbol{V}_h}
    \frac{
      a_h(\boldsymbol{w}_h,\boldsymbol{v}_h)
    }{
      \tnorm{\boldsymbol{v}_h}_s
    }
    &&\qquad \forall \boldsymbol{w}_h\in \boldsymbol{V}_h.
  \end{alignat}
\end{subequations}
We prove \cref{eq:new_inf_sup_v_norm,eq:new_inf_sup_s_norm} in
\cref{ss:vnorminfsupcondition,ss:snorminfsupcondition},
respectively. We then prove \cref{thm:supginfsup} in
\cref{ss:infsupsupgnorm}. To prove these results we introduce, for
$T\ge 1$, the weighting function
\begin{equation}
  \label{eq:weight_func}
  \varphi = eT\exp(-t/T) + \chi,
\end{equation}
where the positive constant $\chi$ will be determined later. For
$0<T<1$ we propose $\varphi(t) = e\exp(-t) + \chi$. In our analysis,
however, we will only consider $T\ge 1$; the analysis for $T<1$
follows identical steps as the $T\ge 1$ case, resulting in inf-sup
conditions
\cref{thm:supginfsup,eq:new_inf_sup_v_norm,eq:new_inf_sup_s_norm}
independent of $T$. Denoting the cell mean of $\bar{\beta}$ by
$\bar{\beta}_0$, we will also use that \citep[see][]{Burman:2007},
\begin{equation}
  \label{eq:betaminbeta0}
  \norm[0]{\bar{\beta}-\bar{\beta}_0}_{L^{\infty}(\mathcal{K})}\leq
  ch_K|\bar{\beta}|_{W^{1,{\infty}}(\mathcal{K})} \quad \forall \mathcal{K}\in\mathcal{T}_h.
\end{equation}

\subsection{The inf-sup condition with respect to
  $|\mkern-1.5mu|\mkern-1.5mu|\cdot|\mkern-1.5mu|\mkern-1.5mu|_v$}
\label{ss:vnorminfsupcondition}

To prove \cref{eq:new_inf_sup_v_norm} we first require the following
lemmas.

\begin{lemma}
  \label{lem:inf_sup_v_norm_weighted}
  Let $\varphi$ be defined as in \cref{eq:weight_func} with $\chi$
  chosen such that $\chi >
  (e-\sqrt{2})T/\del[0]{\sqrt{2}-1}$. Furthermore, choose the penalty
  parameter $\alpha$ in \cref{eq:st_hdg_adr_compact} such that
  $\alpha>1+4c_{\star}^2$, with $c_{\star}$ the constant in
  \cref{eq:eg_inv_3}. Then for all
  $\boldsymbol{w}_h:=\del{w_h,\varkappa_h}\in \boldsymbol{V}_h$:
  \begin{equation*}
    \begin{split}
      a_{h}(\boldsymbol{w}_h,\varphi\boldsymbol{w}_h)
      \geq
      &
      \tfrac{1}{2}({T+\chi})
      \del[2]{
        \sum_{\mathcal{K}\in\mathcal{T}_h}
        \varepsilon
        \norm[0]{\overline{\nabla}w_h}_\mathcal{K}^2
        +
        \sum_{\mathcal{K}\in\mathcal{T}_h}
        \varepsilon h_K^{-1}
        \norm[0]{\sbr{\boldsymbol{w}_h}}_{\mathcal{Q}_\mathcal{K}}^2
      }
      +
      {\tfrac{1}{2}}
      \sum_{\mathcal{K}\in\mathcal{T}_h}
      \norm[0]{w_h}_\mathcal{K}^2
      \\
      &
      +
      {\del[0]{T+\chi}}
      \del[2]{
        \norm[0]{
          \envert[0]{
            \tfrac{1}{2}\beta\cdot{n}
          }^{1/2}
          \varkappa_h
        }_{\partial\mathcal{E}_N}^2
        +
        \sum_{\mathcal{K}\in\mathcal{T}_h}
        \norm[0]{
          \envert[0]{
            \beta_s
            -
            \tfrac{1}{2}
            \beta\cdot{n}
          }^{1/2}
          \sbr{\boldsymbol{w}_h}
        }_{\partial\mathcal{K}}^2
      }.
    \end{split}
  \end{equation*}
\end{lemma}
\begin{proof}
  On an element $\mathcal{K} \in \mathcal{T}_h$ we have
  $-w_h\beta\cdot\nabla\del[0]{\varphi
    w_h}=-\tfrac{1}{2}\nabla\cdot\del[0]{\varphi\beta
    w_h^2}-\tfrac{1}{2}w_h^2\beta\cdot\nabla\varphi$. Using Gauss's
  theorem,
  $\sbr[0]{\varphi\boldsymbol{w}_h}=\varphi\sbr[0]{\boldsymbol{w}_h}$,
  and that
  $\zeta^+\beta \cdot n = (\beta\cdot n + |\beta \cdot n|)/2$, we note
  that
  \begin{equation*}
    \begin{split}
      a_{h,c}(\boldsymbol{w}_h, \varphi\boldsymbol{w}_h)
      =&
      -
      \del[0]{
        \tfrac{1}{2}
        w_h^2
        ,
        {\beta\cdot\nabla\varphi}
      }_{\mathcal{T}_h}
      +
      \langle
      \tfrac{1}{2}
      \varphi
      \varkappa_h^2
      ,
      {\beta\cdot{n}+\envert[0]{\beta\cdot{n}}}
      \rangle_{\partial\mathcal{E}_N}
      \\
      &
      -
      \langle
      \tfrac{1}{2}
      \varphi w_h^2
      ,
      {\beta\cdot{n}}
      \rangle_{\partial\mathcal{T}_h}
      +
      \langle
      \varphi
      \sbr[0]{\boldsymbol{w}_h}^2
      ,
      \sup\envert[0]{\beta\cdot{n}}
      \rangle_{\partial\mathcal{T}_h}
      +
      \langle
      \varphi
      \varkappa_h
      \sbr[0]{\boldsymbol{w}_h}
      ,
      {\beta\cdot{n}}
      \rangle_{\partial\mathcal{T}_h}.
    \end{split}
  \end{equation*}
  Since
  $-\tfrac{1}{2}{w}_h^2+\varkappa_h\sbr[0]{\boldsymbol{w}_h}=
  -\tfrac{1}{2}\sbr[0]{\boldsymbol{w}_h}^2-\tfrac{1}{2}\varkappa_h^2$,
  $\varkappa_h$ is single-valued on element boundaries,
  $\varkappa_h=0$ on $\partial\mathcal{E}_D$, we have by definition of
  $\varphi$ and using that $-\beta \cdot \nabla \varphi \ge 1$:
  \begin{multline}
    \label{eq:some_eq_37}
    a_{h,c}(\boldsymbol{w}_h,\varphi\boldsymbol{w}_h)
    \geq
    \del[0]{T+\chi}
    \norm[0]{
      \envert[0]{\tfrac{1}{2}\beta\cdot{n}}^{1/2}
      \varkappa_h
    }_{\partial\mathcal{E}_N}^2
    +\tfrac{1}{2}
    \sum_{\mathcal{K}\in\mathcal{T}_h}
    \norm[0]{w_h}_\mathcal{K}^2
    \\
    +
    \del[0]{T+\chi}
    \sum_{\mathcal{K}\in\mathcal{T}_h}
    \norm[0]{
      \envert[0]{
        \beta_s
        -
        \tfrac{1}{2}
        \beta\cdot{n}
      }^{1/2}
      \sbr[0]{\boldsymbol{w}_h}
    }_{\partial\mathcal{K}}^2.
  \end{multline}
  Next, noting that $\overline{\nabla}\varphi =0$, and using the
  Cauchy--Schwarz inequality and \cref{eq:eg_inv_3},
  \begin{equation*}
    \begin{split}
      a_{h,d} (\boldsymbol{w}_h,\varphi\boldsymbol{w}_h)
      \geq
      &
      \del[0]{T+\chi}
      \sum_{\mathcal{K}\in\mathcal{T}_h}
      \varepsilon\norm[0]{\overline{\nabla} w_h}_\mathcal{K}^2
      +
      \del[0]{T+\chi}
      \alpha
      \sum_{\mathcal{K}\in\mathcal{T}_h}
      \varepsilon h_K^{-1}
      \norm[0]{\sbr[0]{\boldsymbol{w}_h}}_{\mathcal{Q}_\mathcal{K}}^2
      \\
      &
      -
      \sum_{\mathcal{K}\in\mathcal{T}_h}
      2\varepsilon^{1/2}
      c_{\star}
      \del[0]{eT+\chi}
      \norm[0]{\overline{\nabla}w_h}_\mathcal{K}
      \varepsilon^{1/2}h_K^{-1/2}
      \norm[0]{\sbr[0]{\boldsymbol{w}_h}}_{\mathcal{Q}_\mathcal{K}}.
    \end{split}
  \end{equation*}
  Using H\"older's inequality for sums and the inequality
  $ax^2-2bxy+dy^2 \geq (ad-b^2) (x^2+y^2) / (a+d)$, which holds for
  positive real numbers $a,b,d$ and $ad > b^2$
  \citep[see][]{Pietro:book} allows us to obtain
  \begin{equation*}
    a_{h,d}(\boldsymbol{w}_h,\varphi\boldsymbol{w}_h)
    \geq
    \del[0]{T+\chi}
    \tfrac{
      \alpha
      -
      \del[1]{
        \tfrac{
          c_{\star}
          \del[0]{eT+\chi}
        }{
          T+\chi
        }
      }^2
    }{1+\alpha}
    \del[0]{
      \sum_{\mathcal{K}\in\mathcal{T}_h}
      \varepsilon\norm[0]{\overline{\nabla} w_h}_\mathcal{K}^2
      +
      \sum_{\mathcal{K}\in\mathcal{T}_h}
      \varepsilon h_K^{-1}
      \norm[0]{\sbr[0]{\boldsymbol{w}_h}}_{\mathcal{Q}_\mathcal{K}}^2
    }.
  \end{equation*}
  Since $\chi$ and $\alpha$ are chosen such that
  $\chi > (e-\sqrt{2})T/\del[0]{\sqrt{2}-1}$, so that
  $T+\chi > (eT+\chi)/\sqrt{2}$, and
  $\alpha > 1 + 4 c_{\star}^2$, it follows that
  \begin{equation}
    \label{eq:some_eq_42}
    a_{h,d}(\boldsymbol{w}_h,\varphi\boldsymbol{w}_h)
    \geq
    \tfrac{1}{2}(T+\chi)
    \del[0]{
      \sum_{\mathcal{K}\in\mathcal{T}_h}
      \varepsilon
      \norm[0]{\overline{\nabla}w_h}_\mathcal{K}^2
      +
      \sum_{\mathcal{K}\in\mathcal{T}_h}
      \varepsilon h_K^{-1}
      \norm[0]{\sbr[0]{\boldsymbol{w}_h}}_{\mathcal{Q}_\mathcal{K}}^2
    }.
  \end{equation}
  The result follows after combining
  \cref{eq:some_eq_37,eq:some_eq_42}.
\end{proof}

Let $\Pi_h^{\mathcal{F}}$ be the $L^2$-projection onto
$M_h^{(p_t,p_s)}$. In \cref{s:projforstab} we show that for
$u \in H^1(\mathcal{K})$,
\begin{subequations}\label{eq:projforstab}
  \begin{align}
    \norm[0]{\overline{\nabla}\del[0]{u-\Pi_hu}}_{\mathcal{K}}
    &\leq
      c
      \norm[0]{\overline{\nabla}u}_{\mathcal{K}},
      \label{eq:proj_spatial_grad}
    \\
    \norm[0]{\partial_t\del[0]{u-\Pi_hu}}_{\mathcal{K}}
    &\leq
      c
      \del[0]{
      \norm[0]{\partial_tu}_{\mathcal{K}}
      +
      \norm[0]{\overline{\nabla}u}_{\mathcal{K}}
      },
      \label{eq:proj_time_derivative}
    \\
    \norm[0]{\Pi_hu-\Pi_h^{\mathcal{F}}u}_{\mathcal{Q}_\mathcal{K}}
    &\leq
      c
      h_K^{1/2}
      \norm[0]{\overline{\nabla}u}_{\mathcal{K}}.
      \label{eq:proj_diff_elem_facet_Q}
  \end{align}
\end{subequations}

The following lemma extends the $L^2$-projection estimates of
\cite[Lemma 4.2]{Ayuso:2009} to space-time elements, taking into
account the spatial mesh size $h_K$ and time step
$\delta t_{\mathcal{K}}$.

\begin{lemma}
  \label{lem:proj_with_wt}
  Let $\varphi$ be the function defined in \cref{eq:weight_func}. For
  any $w_h \in V_h^{(p_t,p_s)}$ the following estimates hold:
  \begin{subequations}
    \label{eq:proj_with_wt}
    \begin{align}
      \norm[0]{\del{I-\Pi_h}\del{\varphi w_h}}_\mathcal{K}
      &\leq
        c
        \delta t_{\mathcal{K}}
        \norm[0]{w_h}_\mathcal{K},
        \label{eq:proj_with_wt_1}
      \\
      \norm[0]{
      \overline{\nabla}
      \del[0]{
      \del[0]{I-\Pi_h}
      \del[0]{\varphi w_h}
      }
      }_\mathcal{K}
      &\leq
        c
        {\delta t_{\mathcal{K}} h_K^{-1}}
        \norm[0]{w_h}_\mathcal{K},
        \label{eq:proj_with_wt_2}
      \\
      \norm[0]{
      \overline{\nabla}
      \del[0]{
      \del[0]{I-\Pi_h}
      \del[0]{\varphi w_h}
      }
      }_{\mathcal{Q}_\mathcal{K}}
      &\leq
        c
        \delta t_{\mathcal{K}}
        h_K^{-3/2}
        \norm[0]{w_h}_\mathcal{K},
        \label{eq:proj_with_wt_5}
      \\
      \norm[0]{
      \del[0]{I-\Pi_h}
      \del[0]{\varphi w_h}
      }_{\mathcal{Q}_\mathcal{K}}
      &\leq
        c
        {\delta t_{\mathcal{K}}h_K^{-1/2}}
        \norm[0]{w_h}_\mathcal{K},
        \label{eq:proj_with_wt_3}
      \\
      \norm[0]{
      \del[0]{I-\Pi_h}
      \del[0]{\varphi w_h}
      }_{\mathcal{R}_\mathcal{K}}
      &\leq
        c
        {\delta t_{\mathcal{K}}}^{1/2}
        \norm[0]{w_h}_\mathcal{K}.
        \label{eq:proj_with_wt_4}
    \end{align}
  \end{subequations}
\end{lemma}
\begin{proof}
  See \cref{s:proof_prof_with_wt}.
\end{proof}

\begin{lemma}
  \label{lem:stab_fortin}
  Let
  $\boldsymbol{\Pi}_h\del[0]{\varphi\boldsymbol{w}_h} :=
  \del[0]{\Pi_h\del[0]{\varphi
      w_h},\Pi_h^{\mathcal{F}}\del[0]{\varphi\varkappa_h}}$ for all
  $\boldsymbol{w}_h:=\del{w_h,\varkappa_h}\in \boldsymbol{V}_h$. The
  following holds:
  \begin{equation*}
    \tnorm{\boldsymbol{\Pi}_h\del{\varphi\boldsymbol{w}_h}}_v
    \leq
    c_T
    \tnorm{\boldsymbol{w}_h}_v.
  \end{equation*}
\end{lemma}
\begin{proof}
  See \cref{s:stab_fortin}.
\end{proof}

\begin{lemma}
  \label{lem:inf_sup_v_norm_wtd_proj}
  For any
  $\boldsymbol{w}:=\del{w,\varkappa}\in L^2(\mathcal{E})\times
  L^2(\Sigma)$, let
  $\boldsymbol{\delta}\boldsymbol{w} :=
  \del[0]{{w}-\Pi_h{w},\varkappa-\Pi_h^{\mathcal{F}}\varkappa}$. The
  following holds for all $\boldsymbol{w}_h\in \boldsymbol{V}_h$:
  \begin{equation*}
    \begin{split}
      a_h(\boldsymbol{w}_h,  \boldsymbol{\delta}\del{\varphi\boldsymbol{w}_h})
      \leq
      &
      c_T
      \del[1]{
        \sum_{\mathcal{K}\in\mathcal{T}_h}
        \varepsilon
        \norm[0]{\overline{\nabla}w_h}_\mathcal{K}^2
       +
        \sum_{\mathcal{K}\in\mathcal{T}_h}
        \varepsilon h_K^{-1}
        \norm[0]{\sbr{\boldsymbol{w}_h}}_{\mathcal{Q}_\mathcal{K}}^2
      }
      +
      \sum_{\mathcal{K}\in\mathcal{T}_h}
      {\del[0]{
          1/8
          +
          {\delta t_{\mathcal{K}}}
        }}
      \norm[0]{w_h}_\mathcal{K}^2
      \\
      &
      +
      c_T
      \del[1]{
        \sum_{F\in\partial\mathcal{E}_N}
        \norm[0]{
          \envert[0]{
            \tfrac{1}{2}\beta\cdot{n}
          }^{1/2}
          \varkappa_h
        }_F^2
        +
        \sum_{\mathcal{K}\in\mathcal{T}_h}
        \norm[0]{
          \envert[0]{
            \beta_s
            -
            \tfrac{1}{2}
            \beta\cdot{n}
          }^{1/2}
          \sbr{\boldsymbol{w}_h}
        }_{\partial\mathcal{K}}^2
      }.
    \end{split}
  \end{equation*}
\end{lemma}
\begin{proof}
  Let $z\in H^1(\mathcal{T}_h)$ and $\varpi\in L^2(\mathcal{F}_h)$
  such that $\varpi|_{\partial\mathcal{E}_D}=0$. Let
  $\boldsymbol{z}:=\del{z,\varpi}$. Integrating
  $\del{\beta w_h,\nabla z}_{\mathcal{T}_h}$ by parts and using that
  $\langle \del{\beta\cdot{n}}\varkappa_h , \varpi
  \rangle_{\partial\mathcal{T}_h} = \langle
  \del{\beta\cdot{n}}\varkappa_h , \varpi
  \rangle_{\partial\mathcal{E}_N} $, because $\varkappa_h$ and
  $\varpi$ are single-valued on $\Gamma$ and zero on
  $\partial\mathcal{E}_D$, we have:
  \begin{equation}
    \label{eq:some_eq_112}
    \begin{split}
      a_{h,c}(\boldsymbol{w}_h,\boldsymbol{z})
        =&
        \del{\beta\cdot\nabla w_h,z}_{\mathcal{T}_h}
        -
        \langle
        \tfrac{1}{2}
        \del{\beta\cdot{n}}
        \sbr{\boldsymbol{w}_h},
        z
        \rangle_{\partial\mathcal{T}_h}
        \\
        &+
        \langle
        \del{
          \beta_s
          -
          \tfrac{1}{2}
          \del{\beta\cdot{n}}
        }
        \sbr{\boldsymbol{w}_h}
        ,
        z
        \rangle_{\partial\mathcal{T}_h}
        -
        \langle
        \beta_s
        \sbr{\boldsymbol{w}_h}
        ,
        \varpi
        \rangle_{\partial\mathcal{T}_h}
        +
        \langle
        \tfrac{1}{2}
        \del{
          \envert[0]{{\beta}\cdot{n}}
          -
          \beta\cdot{n}
        }\varkappa_h,
        \varpi
        \rangle_{\partial\mathcal{E}_N}.
      \end{split}
    \end{equation}
    At this point, note that
    $\boldsymbol{\delta}\del{\varphi\boldsymbol{w}_h}
    =\boldsymbol{\delta}\del{eT\exp(-{t/T})\boldsymbol{w}_h}$ because
    $\boldsymbol{\delta}\del{\chi\boldsymbol{w}_h}=0$. Furthermore,
    let $\beta_0 = (1,\bar{\beta}_0)$. By definition of $\Pi_h$, the
    following vanishes:
    $\del[0]{ \beta_0\cdot\nabla w_h, \del[0]{I-\Pi_h}
      \del[0]{eT\exp(-t/T)}w_h}_{\mathcal{T}_h} =0$. From
    \cref{eq:some_eq_112}, with
    $\boldsymbol{z}=\boldsymbol{\delta}\del[0]{eT\exp(-t/T)\boldsymbol{w}_h}$,
    we now find that:
    \begin{equation*}
      \begin{split}
      a_{h,c}\big(\boldsymbol{w}_h, \boldsymbol{\delta}\del[0]{eT\exp(-t/T)}\boldsymbol{w}_h\big)
      =
      &
      \del[0]{
        \del[0]{\beta-\beta_0}\cdot\nabla w_h,
        \del[0]{I-\Pi_h}\del[0]{eT\exp(-t/T)w_h}
      }_{\mathcal{T}_h}
      \\
      &
      -
      \langle
      \tfrac{1}{2}
      \del[0]{\beta\cdot{n}}\sbr{\boldsymbol{w}_h},
      \del[0]{I-\Pi_h}\del[0]{eT\exp(-t/T)w_h}
      \rangle_{\partial\mathcal{T}_h}
      \\
      &
      +
      \langle
      \del[0]{
        \beta_s
        -
        \tfrac{1}{2}
        \del[0]{\beta\cdot{n}}
      }
      \sbr{\boldsymbol{w}_h}
      ,
      \del[0]{I-\Pi_h}\del[0]{eT\exp(-t/T) w_h}
      \rangle_{\partial\mathcal{T}_h}
      \\
      &
      +
      \langle
      \tfrac{1}{2}
      \del[0]{
        \envert[0]{{\beta}\cdot{n}}
        -
        \beta\cdot{n}
      }\varkappa_h,
      \del[0]{I-\Pi_h^{\mathcal{F}}}
      \del[0]{eT\exp(-t/T)\varkappa_h}
      \rangle_{\partial\mathcal{E}_N}
      \\
      =:&
      M_1+M_2+M_3+M_4,
    \end{split}
  \end{equation*}
  where, by definitions of $\Pi_h^\mathcal{F}$ and $\beta_s$,
  $\langle \beta_s\sbr{\boldsymbol{w}_h},
  \del[0]{I-\Pi_h^{\mathcal{F}}} \del[0]{eT\exp(-t/T)\varkappa_h}
  \rangle_{ \partial\mathcal{T}_h } =0$. We will now bound each of the
  terms $M_i$, $i=1,\hdots,4$ separately.

  We observe that
  $ \del[0]{\beta-\beta_0} \cdot\nabla w_h=
  \del[0]{\bar{\beta}-\bar{\beta}_0}\cdot\overline{\nabla} w_h$
  because the first components of $\beta$ and $\beta_0$ are 1. We then
  bound $M_1$ using the Cauchy--Schwarz inequality,
  \cref{eq:betaminbeta0}, \cref{eq:proj_with_wt_1}, and
  \cref{eq:eg_inv_2}:
  \begin{equation}
    \label{eq:M_term_2}
    M_1
    \leq
    \sum_{\mathcal{K}\in\mathcal{T}_h}
    c
    h_K
    h_K^{-1}
    \norm[0]{w_h}_\mathcal{K}
    \delta t_{\mathcal{K}}
    \norm[0]{w_h}_\mathcal{K}
    =
    c
    \sum_{\mathcal{K}\in\mathcal{T}_h}
    \delta t_{\mathcal{K}}
    \norm[0]{w_h}_\mathcal{K}^2.
  \end{equation}
  We proceed with bounding $M_2$ and $M_3$. Using the Cauchy--Schwarz
  inequality, \cref{eq:betasinfmax}, \cref{eq:proj_with_wt_3}, and
  \cref{eq:proj_with_wt_4}, we find that
  \begin{multline*}
    M_2+M_3
    \le
    c
    \sum_{\mathcal{K}\in\mathcal{T}_h}
    \delta t_{\mathcal{K}}h_K^{-1/2}
    \norm[0]{
      \envert[0]{
        \beta_s
        -
        \tfrac{1}{2}
        \beta\cdot{n}
      }^{1/2}
      \sbr{\boldsymbol{w}_h}
    }_{\mathcal{Q}_\mathcal{K}}
    \norm[0]{w_h}_\mathcal{K}
    \\
    +
    c
    \sum_{\mathcal{K}\in\mathcal{T}_h}
    {\delta t_{\mathcal{K}}}^{1/2}
    \norm[0]{
      \envert[0]{
        \beta_s
        -
        \tfrac{1}{2}
        \beta\cdot{n}
      }^{1/2}
      \sbr{\boldsymbol{w}_h}
    }_{\mathcal{R}_\mathcal{K}}
    \norm[0]{w_h}_\mathcal{K}.
  \end{multline*}
  Since $\delta t_{\mathcal{K}}\leq h_K$,
  $\delta t_{\mathcal{K}}h_K^{-1/2}$ can be bounded by $1$. Therefore,
  applying Young's inequality,
  \begin{equation}
    \label{eq:M_term_3_4Y}
    \begin{split}
      M_2+M_3
      \leq&
      c
      \sum_{\mathcal{K}\in\mathcal{T}_h}
      \norm[0]{
        \envert[0]{
          \beta_s
          -
          \tfrac{1}{2}
          \del[0]{\beta\cdot{n}}
        }^{1/2}
        \sbr{\boldsymbol{w}_h}
      }_{\partial\mathcal{K}}
      \norm[0]{w_h}_\mathcal{K}
      \\
      \le&
      \tfrac{1}{2}\delta\sum_{\mathcal{K}\in\mathcal{T}_h}
      \norm[0]{w_h}_{\mathcal{K}}^2
      +
      c
      \delta^{-1}
      \sum_{\mathcal{K}\in\mathcal{T}_h}
      \norm[0]{
        \envert[0]{
          \beta_s-\tfrac{1}{2}\del[0]{\beta\cdot n}
        }^{1/2}\sbr[0]{\boldsymbol{w}_h}
      }_{\partial\mathcal{K}}^2.
    \end{split}
  \end{equation}
  For $M_4$, we first apply the Cauchy--Schwarz inequality and the
  triangle inequality:
  \begin{equation}
    \label{eq:some_eq_113}
    \begin{split}
      M_4
      &
      \leq
      c
      \sum_{F\in\partial\mathcal{E}_N}
      \norm[0]{
        \envert[0]{\tfrac{1}{2}\beta\cdot{n}}^{1/2}\varkappa_h
      }_F
      \del[1]{
        \norm[0]{
          \envert[0]{\tfrac{1}{2}\beta\cdot{n}}^{1/2}
          \varkappa_h
        }_F
        +
        \norm[0]{
          \envert[0]{\tfrac{1}{2}\beta\cdot{n}}^{1/2}
          \Pi_h^\mathcal{F}\del[0]{eT\exp(-t/T)\varkappa_h}
        }_F
      }.
    \end{split}
  \end{equation}
  The second term in parentheses on the right-hand side of
  \cref{eq:some_eq_113} is bounded following identical steps in
  showing \cref{eq:singleboundFinEN}. Applying also Young's
  inequality, and denoting by $\mathcal{K}_F$ the space-time element
  of which $F$ is a facet,
  \begin{equation}
    \label{eq:M_term_5_Y}
    \begin{split}
      M_4
      \leq
      &
      c
      \sum_{F\in\partial\mathcal{E}_N}
      \norm[0]{
        \envert[0]{\tfrac{1}{2}\beta\cdot{n}}^{1/2}\varkappa_h}_F^2
      \\
      &
      +
      c
      \sum_{F\in\partial\mathcal{E}_N}
      \norm[0]{
        \envert[0]{\tfrac{1}{2}\beta\cdot{n}}^{1/2}\varkappa_h
      }_F
      \del[1]{
        c_T
        \norm[0]{
          \envert[0]{\tfrac{1}{2}\beta\cdot{n}}^{1/2}\varkappa_h
        }_F
        +
        c_T
        \norm[0]{\del[0]{\beta_s-\tfrac{1}{2}\beta\cdot n}^{1/2}\sbr{\boldsymbol{w}_h}}_{\partial\mathcal{K}}
        +
        \norm[0]{w_h}_{\mathcal{K}_F}
      }
      \\
      =
      &
      \del[0]{c+c_T}
      \sum_{F\in\partial\mathcal{E}_N}
      \norm[0]{
        \envert[0]{\tfrac{1}{2}\beta\cdot{n}}^{1/2}\varkappa_h
      }_F^2
      \\
      &
      +
      c_T
      \sum_{F\in\partial\mathcal{E}_N}
      \norm[0]{
        \envert[0]{\tfrac{1}{2}\beta\cdot{n}}^{1/2}\varkappa_h
      }_F
      \norm[0]{\del[0]{\beta_s-\tfrac{1}{2}\beta\cdot n}^{1/2}\sbr{\boldsymbol{w}_h}}_{\partial\mathcal{K}}
      +
      c
      \sum_{F\in\partial\mathcal{E}_N}
      \norm[0]{
        \envert[0]{\tfrac{1}{2}\beta\cdot{n}}^{1/2}\varkappa_h
      }_F
      \norm[0]{w_h}_{\mathcal{K}_F}
      \\
      \le
      &
      \tfrac{\delta }{2}\sum_{\mathcal{K}\in\mathcal{T}_h}
      \norm[0]{w_h}_{\mathcal{K}}^2
      +
      \del[0]{c_T+c\delta^{-1}}
      \sum_{F\in\partial\mathcal{E}_N}
      \norm[0]{\envert[0]{\tfrac{1}{2}\beta\cdot
          n}^{1/2}\varkappa_h}_F^2
      +
      c_T
      \sum_{F\in\partial\mathcal{E}_N}
      \norm[0]{
        \del[0]{\beta_s-\tfrac{1}{2}\beta\cdot n }^{1/2}
        \sbr{\boldsymbol{w}_h}
      }_{\partial\mathcal{K}}^2.
    \end{split}
  \end{equation}
  We proceed with the diffusive term $a_{h,d}$. With test function
  $\boldsymbol{z} =
  \boldsymbol{\delta}\del[0]{eT\exp(-t/T)\boldsymbol{w}_h}$,
  \begin{equation*}
    \begin{split}
      &
      a_{h,d}\del[1]{\boldsymbol{w}_h, \boldsymbol{\delta}\del[0]{eT\exp(-t/T)}\boldsymbol{w}_h}
      \\
      =
      &
      \del[1]{
        \varepsilon\overline{\nabla}w_h
        ,
        \overline{\nabla}\del[0]{\del[0]{I-\Pi_h}\del[0]{eT\exp(-t/T)}w_h}
      }_{\mathcal{T}_h}
      -
      \langle
      \varepsilon\alpha h_K^{-1}\sbr{\boldsymbol{w}_h},
      \del[0]{\Pi_h-\Pi_h^\mathcal{F}}
      \del[0]{eT\exp(-t/T)w_h}
      \rangle_{\mathcal{Q}_h}
      \\
      &
      +
      \langle
      \varepsilon\alpha h_K^{-1}\sbr{\boldsymbol{w}_h},
      \del[0]{I-\Pi_h^{\mathcal{F}}}
      \del[0]{eT\exp(-t/T)\sbr[0]{\boldsymbol{w}_h}}
      \rangle_{\mathcal{Q}_h}
      -
      \langle
      \varepsilon\sbr{\boldsymbol{w}_h},
      \overline{\nabla}_{\bar{{n}}}
      \del[0]{
        \del[0]{I-\Pi_h}\del[0]{eT\exp(-t/T)w_h}
      }
      \rangle_{\mathcal{Q}_h}
      \\
      &
      +
      \langle
      \varepsilon\overline{\nabla}_{\bar{{n}}}{w_h},
      \del[0]{\Pi_h-\Pi_h^\mathcal{F}}\del[0]{eT\exp(-t/T)w_h}
      \rangle_{\mathcal{Q}_h}
      -
      \langle
      \varepsilon\overline{\nabla}_{\bar{{n}}}{w_h},
      \del[0]{I-\Pi_h^\mathcal{F}}\del[0]{eT\exp(-t/T)\sbr[0]{\boldsymbol{w}_h}}
      \rangle_{\mathcal{Q}_h}
      \\
      =:
      &
      M_5+M_6+M_7+M_8+M_9+M_{10}.
    \end{split}
  \end{equation*}
  To bound $M_5$ we use the Cauchy--Schwarz inequality,
  \cref{eq:proj_with_wt_2}, the assumption that
  ${\delta t_{\mathcal{K}}}\leq h_K$, and Young's inequality:
  \begin{equation}
    \label{eq:M_term_6_Y}
    M_5
    \leq
    c
    \sum_{\mathcal{K}\in\mathcal{T}_h}
    \varepsilon\norm[0]{\overline{\nabla}w_h}_\mathcal{K}
    \delta t_{\mathcal{K}} h_K^{-1}
    \norm[0]{ w_h }_\mathcal{K}
    \le
    \tfrac{1}{2}\delta
    \sum_{\mathcal{K}\in\mathcal{T}_h}
    \norm[0]{w_h}_{\mathcal{K}}^2
    +
    c\varepsilon{\delta}^{-1}
    \sum_{\mathcal{K}\in\mathcal{T}_h}
    \varepsilon
    \norm[0]{\overline{\nabla}w_h}_{\mathcal{K}}^2.
  \end{equation}
  To bound $M_6$ we apply the Cauchy--Schwarz inequality,
  \cref{eq:proj_diff_elem_facet_Q}, and Young's inequality:
  \begin{equation}
    \label{eq:M_term_7_Y}
    \begin{split}
      M_6
      \leq&
      c
      \sum_{\mathcal{K}\in\mathcal{T}_h}
      \varepsilon h_K^{-1}
      \norm[0]{\sbr[0]{\boldsymbol{w}_h}}_{\mathcal{Q}_{\mathcal{K}}}
      h_K^{1/2}
      \norm[0]{
        \overline{\nabla}
        \del[0]{eT\exp(-t/T)w_h}
      }_{\mathcal{K}}
      \\
      \le&
      c_T
      \sum_{\mathcal{K}\in\mathcal{T}_h}
      \varepsilon h_K^{-1}
      \norm[0]{\sbr[0]{\boldsymbol{w}_h}}_{\mathcal{Q}_{\mathcal{K}}}^2
      +
      c_T
      \sum_{\mathcal{K}\in\mathcal{T}_h}
      \varepsilon \norm[0]{\overline{\nabla}w_h}_{\mathcal{K}}^2.
    \end{split}
  \end{equation}
  $M_7$ can be bounded using the boundedness of $\Pi_h^{\mathcal{F}}$:
  \begin{equation}
    \label{eq:M_term_7_1}
    M_7
    \leq
    c
    \sum_{\mathcal{K}\in\mathcal{T}_h}
    \varepsilon h_K^{-1}
    \norm[0]{\sbr[0]{\boldsymbol{w}_h}}_{\mathcal{Q}_{\mathcal{K}}}
    \norm[0]{
      \del[0]{I-\Pi_h^{\mathcal{F}}}
      \del[0]{eT\exp(-t/T)\sbr[0]{\boldsymbol{w}_h}}
    }_{\mathcal{Q}_{\mathcal{K}}}
    \leq
    c_T
    \sum_{\mathcal{K}\in\mathcal{T}_h}
    \varepsilon h_K^{-1}
    \norm[0]{\sbr[0]{\boldsymbol{w}_h}}_{\mathcal{Q}_{\mathcal{K}}}^2.
  \end{equation}
  Terms $M_9$ and $M_{10}$ are bounded in a similar way as $M_6$ and
  $M_7$, and using \cref{eq:eg_inv_3}:
  \begin{equation}
    \label{eq:M_term_9_Y}
    M_9+M_{10}
    \le
    c_T
    \sum_{\mathcal{K}\in\mathcal{T}_h}
    \varepsilon h_K^{-1}
    \norm[0]{\sbr[0]{\boldsymbol{w}_h}}_{\mathcal{Q}_{\mathcal{K}}}^2
    +
    c_T
    \sum_{\mathcal{K}\in\mathcal{T}_h}
    \varepsilon
    \norm[0]{\overline{\nabla}w_h}_{\mathcal{K}}^2.
  \end{equation}
  Finally, we bound $M_8$ using the Cauchy--Schwarz inequality,
  \cref{eq:proj_with_wt_5}, the assumption that
  ${\delta t_{\mathcal{K}}}\leq h_K$, and Young's inequality:
  \begin{equation}
    \label{eq:M_term_8_Y}
    \begin{split}
      M_8
      &
      \leq
      \sum_{\mathcal{K}\in\mathcal{T}_h}
      \varepsilon
      \norm[0]{\sbr{\boldsymbol{w}_h}}_{\mathcal{Q}_\mathcal{K}}
      \norm[0]{
        \overline{\nabla}
        \del[0]{(I-\Pi_h)\del[0]{eT\exp(-t/T)w_h}}
      }_{\mathcal{Q}_\mathcal{K}}
      \\
      &
      \leq
      c
      \sum_{\mathcal{K}\in\mathcal{T}_h}
      \varepsilon
      \norm[0]{\sbr{\boldsymbol{w}_h}}_{\mathcal{Q}_\mathcal{K}}
      \delta t_{\mathcal{K}} h_K^{-3/2}
      \norm[0]{ w_h }_{\mathcal{K}}
      \\
      &\le
      \tfrac{1}{2}\delta
      \sum_{\mathcal{K}\in\mathcal{T}_h}
      \norm[0]{w_h}_{\mathcal{K}}^2
      +
      c\varepsilon\delta^{-1}
      \sum_{\mathcal{K}\in\mathcal{T}_h}
      \varepsilon h_K^{-1}
      \norm[0]{\sbr[0]{\boldsymbol{w}_h}}_{\mathcal{Q}_{\mathcal{K}}}^2.
    \end{split}
  \end{equation}
  Collecting
  \cref{eq:M_term_2,eq:M_term_3_4Y,eq:M_term_5_Y,eq:M_term_6_Y,eq:M_term_7_Y,eq:M_term_7_1,eq:M_term_8_Y,eq:M_term_9_Y}
  we find that
  \begin{equation*}
    \begin{split}
      a_h(\boldsymbol{w}_h,\boldsymbol{\delta}\del[0]{\varphi\boldsymbol{w}_h})
      \leq
      &
      c
      \sum_{\mathcal{K}\in\mathcal{T}_h}
      \delta t_{\mathcal{K}}
      \norm[0]{w_h}_\mathcal{K}^2
      +
      2\delta\sum_{\mathcal{K}\in\mathcal{T}_h}\norm[0]{w_h}_{\mathcal{K}}^2
      \\
      &
      +
      \del[0]{c_T+c\delta^{-1}}
      \del[2]{
        \sum_{F\in\partial\mathcal{E}_N}
        \norm[0]{ \envert[0]{ \tfrac{1}{2}\beta\cdot{n} }^{1/2} \varkappa_h }_F^2
        +
        \sum_{\mathcal{K}\in\mathcal{T}_h}
        \norm[0]{
          \envert[0]{
            \beta_s-\tfrac{1}{2}\beta\cdot{n}
          }^{1/2}
          \sbr{\boldsymbol{w}_h}
        }_{\partial\mathcal{K}}^2
      }
      \\
      &
      +
      \del[0]{c_T+c\varepsilon\delta^{-1}}
      \del[2]{
        \sum_{\mathcal{K}\in\mathcal{T}_h}
        \varepsilon
        \norm[0]{\overline{\nabla}w_h}_{\mathcal{K}}^2
        +
        \sum_{\mathcal{K}\in\mathcal{T}_h}
        \varepsilon h_K^{-1}
        \norm[0]{\sbr{\boldsymbol{w}_h}}_{\mathcal{Q}_\mathcal{K}}^2
        }.
    \end{split}
 \end{equation*}
  The result follows by choosing $\delta=1/16$.
\end{proof}

We are now ready to prove \cref{eq:new_inf_sup_v_norm}.

\begin{proof}[Proof of \cref{eq:new_inf_sup_v_norm}]
  Choose $\delta t_0=1/8$. When
  $\delta t_{\mathcal{K}} \le \delta t_0$ for all
  $\mathcal{K}\in\mathcal{T}_h$ we find, by combining
  \cref{lem:inf_sup_v_norm_weighted,lem:inf_sup_v_norm_wtd_proj},
  \begin{equation*}
    \begin{split}
      a_h(\boldsymbol{w}_h,\boldsymbol{\Pi}_h\del[0]{\varphi\boldsymbol{w}_h})
      \ge&
      \del[0]{
        \tfrac{1}{4}({T+\chi})
        -
        c_T
      }
      \del[0]{
        \sum_{\mathcal{K}\in\mathcal{T}_h}
        \varepsilon
        \norm[0]{\overline{\nabla}w_h}_\mathcal{K}^2
        +
        \sum_{\mathcal{K}\in\mathcal{T}_h}
        \varepsilon h_K^{-1}
        \norm[0]{\sbr[0]{\boldsymbol{w}_h}}_{\mathcal{Q}_\mathcal{K}}^2
      }
      +
      \tfrac{1}{4}
      \sum_{\mathcal{K}\in\mathcal{T}_h}
      \norm[0]{w_h}_\mathcal{K}^2
      \\
      &
      +
      \del[0]{
        T+\chi-
        c_T
      }
      \del[0]{
        \sum_{F\in\partial\mathcal{E}_N}
        \norm[0]{
          \envert[0]{\tfrac{1}{2}\beta\cdot{n}}^{1/2}
          \varkappa_h
        }_F^2
        +
        \sum_{\mathcal{K}\in\mathcal{T}_h}
        \norm[0]{
          \envert[0]{
            \beta_s
            -
            \tfrac{1}{2}
            \beta\cdot{n}
          }^{1/2}
          \sbr[0]{\boldsymbol{w}_h}
        }_{\partial\mathcal{K}}^2
      }.
    \end{split}
  \end{equation*}
  Choosing $\chi$ to satisfy $\chi\geq 4 c_T$ in addition to the
  conditions of \cref{lem:inf_sup_v_norm_weighted}, we obtain
  \begin{equation}
    \label{eq:sv_0}
    a_h\del[0]{\boldsymbol{w}_h,\boldsymbol{\Pi}_h\del[0]{\varphi\boldsymbol{w}_h}}
    \geq
    \tfrac{1}{4}
    \tnorm{\boldsymbol{w}_h}_v^2
    \geq
    c_T^{-1}
    \tnorm{\boldsymbol{w}_h}_v
    \tnorm{\boldsymbol{\Pi}_h(\varphi\boldsymbol{w}_h)}_v,
  \end{equation}
  where the second inequality is due to \cref{lem:stab_fortin}. We
  therefore conclude \cref{eq:new_inf_sup_v_norm}.
\end{proof}

\subsection{The inf-sup condition with respect to
  $|\mkern-1.5mu|\mkern-1.5mu|\cdot|\mkern-1.5mu|\mkern-1.5mu|_s$}
\label{ss:snorminfsupcondition}

To prove \cref{eq:new_inf_sup_s_norm}, we first construct the test
function $\boldsymbol{y}_h:=\del{y_h,\vartheta_h}$ as a function of
$\boldsymbol{w}_h=\del{w_h,\varkappa_h}\in \boldsymbol{V}_h$. The
elemental test function $y_h$ is defined as:
\begin{subequations}
  \label{eq:test_func_s_inf_sup_localtimestep}
  \begin{alignat}{2}
    y_h
    &:=
    \tau_\varepsilon\partial_tw_h.
    &
    \label{eq:test_func_s_inf_sup_1_localtimestep}
    \\
    \intertext{To define the facet test function $\vartheta_h$ we
      consider four different sets of facets. First we consider facets
      $F$ in
      $\partial\mathcal{K}_1\cap\partial\mathcal{K}_2\cap\mathcal{Q}_h^i$
      and such that there is no difference in the refinement level in
      the time direction between $\mathcal{K}_1$ and
      $\mathcal{K}_2$. This means that
      $\delta t_{\mathcal{K}_1}=\delta t_{\mathcal{K}_2}:=\delta
      t_{\mathcal{K}}$ and, since $\mathcal{K}_1$ and $\mathcal{K}_2$
      must come from the same space-time slab,
      $\Delta t_{\mathcal{K}_1}=\Delta t_{\mathcal{K}_2}:=\Delta
      t_{\mathcal{K}}$. We then define: }
    \vartheta_h &:=
    \begin{cases}
      {\Delta t_{\mathcal{K}}}\partial_t\varkappa_h,
      &
		\delta t_{\mathcal{K}}\leq h_{K_1}\leq\varepsilon,
      \delta t_{\mathcal{K}}\leq h_{K_2}\leq\varepsilon,
      \\
      {\Delta t_{\mathcal{K}}}\varepsilon^{1/2}\partial_t\varkappa_h,
      &
      \delta t_{\mathcal{K}}\leq \varepsilon<h_{K_1},
      \delta t_{\mathcal{K}}\leq \varepsilon<h_{K_2},
      \\
      0,
      &
      \text{otherwise}.
    \end{cases}
    &
    \label{eq:test_func_s_inf_sup_2_localtimestep}
    \\
    \intertext{We next consider facets $F$ in
      $\partial\mathcal{K}_1\cap\partial\mathcal{K}_2\cap\mathcal{Q}_h^i$
      and such that there is one level of refinement difference
      between $\mathcal{K}_1$ and $\mathcal{K}_2$ in the time
      direction.  Without loss of generality, we assume that
      $2\delta t_{\mathcal{K}_1}= \delta
      t_{\mathcal{K}_2}$. Furthermore, since $\mathcal{K}_1$ and
      $\mathcal{K}_2$ must come from the same space-time slab,
      $\Delta t_{\mathcal{K}_1}=\Delta t_{\mathcal{K}_2}:=\Delta
      t_{\mathcal{K}}$. We then define: }
    \vartheta_h &:=
    \begin{cases}
      {\Delta t_{\mathcal{K}}}\partial_t\varkappa_h,
      &
		\delta t_{\mathcal{K}_1}\leq h_{K_1}\leq\varepsilon,
      \delta t_{\mathcal{K}_2}\leq h_{K_2}\leq\varepsilon,
      \\
      {\Delta t_{\mathcal{K}}}\varepsilon^{1/2}\partial_t\varkappa_h,
      &
      \delta t_{\mathcal{K}_1}\leq\varepsilon<h_{K_1},
      \delta t_{\mathcal{K}_2}\leq\varepsilon<h_{K_2},
      \\
      0,
      &
      \text{otherwise}.
    \end{cases}
    &
    \label{eq:test_func_s_inf_sup_3_localtimestep}
    \\
    \intertext{For facets $F$ in
      $\partial\mathcal{K}\cap\mathcal{Q}_h^b$, we define: }
    \vartheta_h
    &:=
    \begin{cases}
      {\Delta t_{\mathcal{K}}}\partial_t\varkappa_h,
      &
		\delta t_{\mathcal{K}}\le h_{K}\leq\varepsilon,
      \\
      {\Delta t_{\mathcal{K}}}\varepsilon^{1/2}\partial_t\varkappa_h,
      &
      \delta t_{\mathcal{K}}\le \varepsilon < h_{K},
      \\
      0,
      &
      \text{otherwise}.
    \end{cases}
    &
    \label{eq:test_func_s_inf_sup_4_localtimestep}
    \\
    \intertext{Finally, for facets $F$ in $\mathcal{R}_h$, we define: }
    \vartheta_h &:= 0.  &
    \label{eq:test_func_s_inf_sup_5_localtimestep}
  \end{alignat}
\end{subequations}
We observe from definition \cref{eq:test_func_s_inf_sup_localtimestep}
that $\vartheta_h\equiv 0$ on $\partial\mathcal{T}_h^c$, which denotes
the set of element boundaries of space-time elements in
$\mathcal{T}_h^c$. Furthermore, for any space-time element
$\mathcal{K}\in\mathcal{T}_h^{dx}:=\mathcal{T}_h^d\cup\mathcal{T}_h^x$,
we introduce $\mathcal{Q}_\mathcal{K}^0$ to denote those
$\mathcal{Q}$-faces on which $\vartheta_h$ is prescribed in
\cref{eq:test_func_s_inf_sup_2_localtimestep,eq:test_func_s_inf_sup_3_localtimestep}
to be zero. We will define
$\mathcal{Q}_h^0 :=
\cup_{\mathcal{K}\in\mathcal{T}_h}\mathcal{Q}_{\mathcal{K}}^0$. Consider
now $\mathcal{K}\in\mathcal{T}_h^{dx,0}$, which denotes the set of
space-time elements in $\mathcal{T}_h^{dx}$ for which
$\mathcal{Q}_{\mathcal{K}}^0 \ne \emptyset$. Then, there exists a
$\mathcal{K}'$ such that
$\partial\mathcal{K}'\cap\partial\mathcal{K}\neq\emptyset$ and that
either $h_K\leq\varepsilon\leq h_{K'}$ (or
$h_{K'}\leq\varepsilon\leq h_{K}$), or
$\delta t_{\mathcal{K}}\leq\varepsilon\leq \delta t_{\mathcal{K}'}$
(or
$\delta t_{\mathcal{K}'}\leq\varepsilon\leq \delta t_{\mathcal{K}}$).
For the former case, since spatial elements are shape-regular and the
difference of refinement levels in the spatial direction between two
adjacent space-time elements is at most one, we have
$c^{-1}h_{K'}\leq h_K\leq ch_{K'}$. If the latter case holds, since
$\delta t_{\mathcal{K}}=\tfrac{1}{2} \delta t_{\mathcal{K}'}$, it
holds that $\delta t_{\mathcal{K}}\sim\varepsilon$. Therefore,
\begin{equation}
  \label{eq:q0_const_localtimestep}
  c^{-1}
  h_K
  \leq
  \varepsilon
  \leq ch_K
  \quad
  \text{ or }
  \quad
  c^{-1}
  \delta t_{\mathcal{K}}
  \leq
  \varepsilon
  \leq c \delta t_{\mathcal{K}}
  \qquad
  \forall
  \mathcal{K}\in\mathcal{T}_h^{dx,0}.
\end{equation}

\Cref{lem:zh_bounded_localtimestep,lem:s_inf_sup_bounded_below_localtimestep}
will be used to prove \cref{eq:new_inf_sup_s_norm}. The proofs of
these lemmas will repeatedly use the following set of inequalities:
For all $\mathcal{K}\in\mathcal{T}_h$,
\begin{equation}
  \label{eq:setofusefulineq}
  h_K^{-1}\leq \delta t_{\mathcal{K}}^{-1},
  \quad
  \Delta t_{\mathcal{K}}\leq
  c
  \delta t_{\mathcal{K}},
  \quad
  \tau_\varepsilon\leq\Delta t_{\mathcal{K}},
  \quad
  \varepsilon \le 1.
\end{equation}

\begin{lemma}
  \label{lem:zh_bounded_localtimestep}
  Assume that ${\delta t_{\mathcal{K}}}\leq h_K$ for all space-time
  elements $\mathcal{K}\in\mathcal{T}_h$. Let
  $\boldsymbol{w}_h=\del{w_h,\varkappa_h}\in \boldsymbol{V}_h$ and let
  $\boldsymbol{y}_h$ be defined by
  \cref{eq:test_func_s_inf_sup_localtimestep}. The following holds:
  \begin{equation}
    \label{eq:zh_bounded_localtimestep}
    \tnorm{\boldsymbol{y}_h}_s
    \leq
    c
    \tnorm{\boldsymbol{w}_h}_s.
  \end{equation}
\end{lemma}
\begin{proof}
  We start with the volume terms of $\tnorm{\cdot}_s$. Using
  \cref{eq:eg_inv_1} and \cref{eq:setofusefulineq}, we have:
  \begin{equation}
    \label{eq:zh_bounded_1_localtimestep}
    \sum_{\mathcal{K}\in\mathcal{T}_h}
    \norm[0]{y_h}_\mathcal{K}^2
    \le
    \sum_{\mathcal{K}\in\mathcal{T}_h}
    \tau_{\varepsilon}\norm[0]{\partial_tw_h}_\mathcal{K}^2.
  \end{equation}
  For the diffusive volume term, using commutativity of
  $\overline{\nabla}$ and $\partial_t$, \cref{eq:eg_inv_1} and
  \cref{eq:setofusefulineq}:
  \begin{equation}
    \label{eq:zh_bounded_2_localtimestep}
    \sum_{\mathcal{K}\in\mathcal{T}_h}
    \varepsilon\norm[0]{\overline{\nabla}y_h}_\mathcal{K}^2
    =
    \sum_{\mathcal{K}\in\mathcal{T}_h}
    \tau_{\varepsilon}^2\varepsilon\norm[0]{\partial_t\del[0]{\overline{\nabla}w_h}}_\mathcal{K}^2
    \leq
    c
    \sum_{\mathcal{K}\in\mathcal{T}_h}
    \varepsilon\norm[0]{{\overline{\nabla}w_h}}_\mathcal{K}^2.
  \end{equation}
  The time-derivative volume term is treated similarly, using
  \cref{eq:setofusefulineq}:
  \begin{equation}
    \label{eq:zh_bounded_3_localtimestep}
    \sum_{\mathcal{K}\in\mathcal{T}_h}
    \tau_\varepsilon
    \norm[0]{\partial_ty_h}_\mathcal{K}^2
    \leq
    c
    \sum_{\mathcal{K}\in\mathcal{T}_h}
    \tau_\varepsilon
    \norm[0]{{\partial_tw_h}}_\mathcal{K}^2.
  \end{equation}
  For the diffusive facet term in the definition of $\tnorm{\cdot}_s$,
  we use \cref{lem:eg_inv}, \cref{eq:eg_inv_3}, that $\vartheta_h$
  vanishes on $\partial\mathcal{T}_h^c$, that $\vartheta_h$ vanishes
  on $\mathcal{Q}_\mathcal{K}^0$ when
  $\mathcal{K}\in\mathcal{T}_h^{dx}$, that
  $\varepsilon\le\delta t_{\mathcal{K}}$ and
  $\varepsilon h_K^{-2}\tau_\varepsilon\leq1$ on $\mathcal{T}_h^c$,
  and that $\varepsilon h_K^{-2}\tau_\varepsilon\leq c$ on
  $\mathcal{T}_h^{dx,0}$ due to \cref{eq:q0_const_localtimestep}:
  \begin{equation}
    \label{eq:zh_bounded_4_localtimestep}
    \begin{split}
      &
      \sum_{\mathcal{K}\in\mathcal{T}_h}
      \varepsilon h_K^{-1}
      \norm[0]{\sbr{\boldsymbol{y}_h}}_{\mathcal{Q}_\mathcal{K}}^2
      \\
      =
      &
      \sum_{\mathcal{K}\in\mathcal{T}_h^{dx}}
      \varepsilon h_K^{-1}
      \tau_\varepsilon^2
      \norm[0]{
        \sbr{\partial_t\boldsymbol{w}_h}
      }_{\mathcal{Q}_\mathcal{K}\setminus\mathcal{Q}_\mathcal{K}^0}^2
      +
      \sum_{\mathcal{K}\in\mathcal{T}_h^{dx}}
      \varepsilon h_K^{-1}
      \tau_\varepsilon^2
      \norm[0]{
        {\partial_t{w}_h}
      }_{\mathcal{Q}_\mathcal{K}^0}^2
      +
      \sum_{\mathcal{K}\in\mathcal{T}_h^c}
      \varepsilon h_K^{-1}
      \tau_\varepsilon^2
      \norm[0]{ \partial_t{w}_h }_{\mathcal{Q}_\mathcal{K}}^2
      \\
      \leq
      &
      c
      \sum_{\mathcal{K}\in\mathcal{T}_h}
      \varepsilon h_K^{-1}
      \norm[0]{\sbr{\boldsymbol{w}_h}}_{\mathcal{Q}_\mathcal{K}}^2
      +
      c
      \sum_{\mathcal{K}\in\mathcal{T}_h}
      \tau_{\varepsilon}
      \norm[0]{ \partial_t{w}_h }_{\mathcal{K}}^2.
    \end{split}
  \end{equation}
  For the advective facet term, using \cref{lem:eg_inv},
  \cref{eq:eg_inv_4}, \cref{eq:setofusefulineq}, that
  $\tilde{\varepsilon}^2\leq\varepsilon h_K^{-1}$ on
  $\mathcal{K}\in\mathcal{T}_h^{dx}$ since $h_K\leq\varepsilon$ on
  $\mathcal{T}_h^d$ and $\tilde{\varepsilon}=\varepsilon^{1/2}$ on $\mathcal{T}_h^x$:
  \begin{equation}
    \label{eq:zh_bounded_5_localtimestep}
    \begin{split}
      &
      \sum_{\mathcal{K}\in\mathcal{T}_h}
      \norm[0]{
        \envert[0]{\beta_s-\tfrac{1}{2}\beta\cdot{n}}^{1/2}
        \sbr{\boldsymbol{y}_h}
      }_{\partial\mathcal{K}}^2
      \\
      \le
      &
      c
      \del[2]{
        \sum_{\mathcal{K}\in\mathcal{T}_h^{dx}}
        \tau_\varepsilon^2
        \norm[0]{
          \sbr{\partial_t\boldsymbol{w}_h}
        }_{\mathcal{Q}_\mathcal{K}\setminus\mathcal{Q}_\mathcal{K}^0}^2
        +
        \sum_{\mathcal{K}\in\mathcal{T}_h^{dx}}
        \tau_\varepsilon^2
        \norm[0]{
          \partial_t{w}_h
        }_{\mathcal{R}_\mathcal{K}\cup\mathcal{Q}_\mathcal{K}^0}^2
        +
        \sum_{\mathcal{K}\in\mathcal{T}_h^c}
        \tau_\varepsilon^2
        \norm[0]{
          {\partial_t{w}_h}
        }_{\partial\mathcal{K}}^2
      }
      \\
      \leq
      &
      c
      \del[2]{
        \sum_{\mathcal{K}\in\mathcal{T}_h^{dx}}
        \varepsilon h_K^{-1}
        \norm[0]{
          \sbr{\boldsymbol{w}_h}
        }_{\mathcal{Q}_\mathcal{K}\setminus\mathcal{Q}_\mathcal{K}^0}^2
      }
      +
      c
      \sum_{\mathcal{K}\in\mathcal{T}_h^c}
      \tau_\varepsilon\norm[0]{\partial_tw_h}_{\mathcal{K}}^2.
    \end{split}
  \end{equation}
  Finally, the Neumann boundary term is bounded using the triangle
  inequality, Young's inequality, \cref{lem:eg_inv},
  \cref{eq:eg_inv_3}, \cref{eq:setofusefulineq}, and that
  $h_K\leq \varepsilon$ for $\mathcal{K}\in\mathcal{T}_h^d$:
  \begin{equation}
    \label{eq:zh_bounded_6_localtimestep}
    \begin{split}
      \sum_{\mathcal{K}\in\mathcal{T}_h^{dx}}
      \norm[0]{
        \envert[0]{\tfrac{1}{2}\beta\cdot{n}}^{1/2}
        \vartheta_h
      }_{
        \mathcal{Q}_\mathcal{K}
        \cap
        \partial\mathcal{E}_N
      }^2
      \leq
      &
      c
      \sum_{\mathcal{K}\in{\mathcal{T}_h^{dx}}}
      \tau_\varepsilon^2
      \norm[0]{
        \sbr{\partial_t \boldsymbol{w}_h}
      }_{
        \mathcal{Q}_\mathcal{K}
        \cap
        \partial\mathcal{E}_N
      }^2
      +
      c
      \sum_{\mathcal{K}\in{\mathcal{T}_h^{dx}}}
      \tau_\varepsilon^2
      \norm[0]{
        \partial_tw_h
      }_{
        \mathcal{Q}_\mathcal{K}
        \cap
        \partial\mathcal{E}_N
      }^2
      \\
      \leq
      &
      c
      \sum_{\mathcal{K}\in{\mathcal{T}_h^{dx}}}
      \varepsilon h_K^{-1}
      \norm[0]{
        \sbr{\boldsymbol{w}_h}
      }_{
        \mathcal{Q}_\mathcal{K}
        \cap
        \partial\mathcal{E}_N
      }^2
      +
      c
      \sum_{\mathcal{K}\in{\mathcal{T}_h^{dx}}}
      \tau_\varepsilon
      \norm[0]{
        \partial_tw_h
      }_{ \mathcal{K} }^2.
    \end{split}
  \end{equation}
  Combining
  \crefrange{eq:zh_bounded_1_localtimestep}{eq:zh_bounded_6_localtimestep}
  yields \cref{eq:zh_bounded_localtimestep}.
\end{proof}

\begin{lemma}
  \label{lem:s_inf_sup_bounded_below_localtimestep}
  Assume that $\delta t_{\mathcal{K}}\leq h_K$ for all space-time
  elements $\mathcal{K} \in \mathcal{T}_h$. Let
  $\boldsymbol{w}_h=\del{w_h,\varkappa_h}\in \boldsymbol{V}_h$, let
  $\boldsymbol{y}_h$ be defined as in
  \cref{eq:test_func_s_inf_sup_localtimestep}, and let
  $\boldsymbol{\Pi}_h\del[0]{\varphi\boldsymbol{w}_h}$ be defined as
  in \cref{lem:stab_fortin}. There exists a positive constant $c$
  such that
  \begin{equation*}
    \tnorm{\boldsymbol{w}_h}_s^2
    \leq
    a_h(\boldsymbol{w}_h,2\del[0]{\boldsymbol{y}_h+
      c
      \boldsymbol{\Pi}_h\del[0]{\varphi\boldsymbol{w}_h}}).
  \end{equation*}
\end{lemma}
\begin{proof}
  Let us first note that $\vartheta_h$ vanishes on $\mathcal{R}_h$ and
  $\partial\mathcal{T}_h^c$. Therefore, defining
  $\mathcal{Q}_h^{dx} := \partial\mathcal{T}_h^{dx}\cap\mathcal{Q}_h$,
  we find after some algebraic manipulation that:
  \begin{equation*}
    \begin{split}
      a_{h,c}(\boldsymbol{w}_h,\boldsymbol{y}_h)
      &
      =
      \del[0]{\nabla\cdot\del[0]{\beta w_h},y_h}_{\mathcal{T}_h}
      +
      \langle
      \del[0]{
        \beta_s
        -
        \tfrac{1}{2}
        \beta\cdot{n}
      }
      \sbr{\boldsymbol{w}_h}, y_h
      \rangle_{\partial\mathcal{T}_h}
      -
      \langle
      \tfrac{1}{2}
      \beta\cdot{n}
      \sbr{\boldsymbol{w}_h}, y_h
      \rangle_{\partial\mathcal{T}_h}
      \\
      &
      \quad-
      \langle
      \tfrac{1}{2}
      \del[0]{
        \envert[0]{\beta\cdot{n}}
        -
        \beta\cdot{n}
      }
      \varkappa_h,\sbr{\boldsymbol{y}_h}
      \rangle_{\partial\mathcal{E}_N\cap\mathcal{Q}_h^{dx}}
      +
      \langle
      \tfrac{1}{2}
      \del[0]{
        \envert[0]{\beta\cdot{n}}
        -
        \beta\cdot{n}
      }
      \varkappa_h,y_h
      \rangle_{\partial\mathcal{E}_N\cap\mathcal{Q}_h^{dx}}
      \\
      &
      \quad+
      \langle
      \beta_s
      \sbr{\boldsymbol{w}_h}, \sbr{\boldsymbol{y}_h}
      \rangle_{\mathcal{Q}_h^{dx}\setminus\mathcal{Q}_h^0}
      -
      \langle
      \beta_s
      \sbr{\boldsymbol{w}_h}, y_h
      \rangle_{\mathcal{Q}_h^{dx}\setminus\mathcal{Q}_h^0}.
    \end{split}
  \end{equation*}
  Furthermore, since
  $ \del[0]{\nabla\cdot\del[0]{\beta w_h},y_h}_{\mathcal{T}_h} =
  \del[0]{\partial_tw_h,\tau_\varepsilon\partial_tw_h}_{\mathcal{T}_h} +
  \del[0]{\overline{\nabla}\cdot\del[0]{\bar{\beta}w_h},\tau_\varepsilon\partial_tw_h}_{\mathcal{T}_h}$,
  we find that
  \begin{equation}
    \label{eq:ss_0_localtimestep}
    \begin{split}
      \sum_{\mathcal{K}\in\mathcal{T}_h}
      \tau_\varepsilon
      \norm[0]{\partial_tw_h}_\mathcal{K}^2
      =
      &
      a_h(\boldsymbol{w}_h,\boldsymbol{y}_h)
      -
      a_{h,d}(\boldsymbol{w}_h,\boldsymbol{y}_h)
      -
      \del[0]{\overline{\nabla}\cdot\del[0]{\bar{\beta}w_h},\tau_\varepsilon\partial_tw_h}_{\mathcal{T}_h}
      \\
      &
      -
      \langle (\beta_s - \tfrac{1}{2}\beta\cdot n)
		\sbr{\boldsymbol{w}_h}, \tau_{\varepsilon}\partial_t w_h \rangle_{\partial\mathcal{T}_h}
      +
      \langle \tfrac{1}{2}\beta \cdot n \sbr{\boldsymbol{w}_h},
		\tau_{\varepsilon}\partial_t w_h \rangle_{\partial\mathcal{T}_h}
      \\
      &
      - \langle \beta_s \sbr{\boldsymbol{w}_h},
		\tau_{\varepsilon}\sbr{\partial_t \boldsymbol{w}_h}
		\rangle_{\mathcal{Q}_{h}^{dx}\backslash\mathcal{Q}_{h}^0}
      + \langle \beta_s \sbr{\boldsymbol{w}_h},
		\tau_{\varepsilon}\partial_t w_h
		\rangle_{\mathcal{Q}_{h}^{dx}\backslash\mathcal{Q}_{h}^0}
      \\
      &
      +
      \langle \tfrac{1}{2}(|\beta\cdot n|-\beta \cdot
		n)\varkappa_h, \tau_{\varepsilon}\sbr{\partial_t\boldsymbol{w}_h} \rangle_{\partial\mathcal{E}_N \cap \mathcal{Q}_h^{dx}}
      -
      \langle \tfrac{1}{2}(|\beta\cdot n|-\beta \cdot
		n)\varkappa_h, \tau_{\varepsilon} \partial_t w_h \rangle_{\partial\mathcal{E}_N \cap \mathcal{Q}_h^{dx}}
      \\
      =
      &
      a_h(\boldsymbol{w}_h,\boldsymbol{y}_h)
      -
      a_{h,d}(\boldsymbol{w}_h,\boldsymbol{y}_h)
      -
      \del[0]{\overline{\nabla}\cdot\del[0]{\bar{\beta}w_h},\tau_\varepsilon\partial_tw_h}_{\mathcal{T}_h}
      +T_1
      +T_2
      +T_3
      +T_4
      +T_5
      +T_6.
    \end{split}
  \end{equation}
  We will bound the last eight terms on the right hand side of the
  above equation. First, by \cref{eq:bnd_v},
  \cref{lem:zh_bounded_localtimestep},
  and Young's inequality, we have
  \begin{equation}
    \label{eq:ss_1_localtimestep}
    \begin{split}
      a_{h,d}(\boldsymbol{w}_h,\boldsymbol{y}_h)
      \leq
      &
      c
      \tnorm{\boldsymbol{w}_h}_v
      \tnorm{\boldsymbol{w}_h}_s
      \leq
      c
      \del[0]{
        \del[0]{
          \sum_{\mathcal{K}\in\mathcal{T}_h}
          \tau_\varepsilon
          \norm[0]{\partial_tw_h}_\mathcal{K}^2
        }^{1/2}
        +
        \tnorm{\boldsymbol{w}_h}_v
      }
      \tnorm{\boldsymbol{w}_h}_v
      \\
      \leq
      &
      c\del[0]{1+\delta^{-1}}
      \tnorm{\boldsymbol{w}_h}_v^2
      +
      \tfrac{1}{2}\delta
      \sum_{\mathcal{K}\in\mathcal{T}_h}
      \tau_\varepsilon
      \norm[0]{\partial_tw_h}_\mathcal{K}^2.
    \end{split}
  \end{equation}
  Next, using the Cauchy--Schwarz and Young's inequalities,
  ${\delta t_{\mathcal{K}}} \leq \varepsilon$ for
  $\mathcal{K}\in\mathcal{T}_h^{dx}$, \cref{eq:eg_inv_2} and
  \cref{eq:setofusefulineq}:
  \begin{equation}
    \label{eq:ss_2_localtimestep}
    \begin{split}
      (\overline{\nabla}\cdot\del[0]{\bar{\beta}w_h},
		\tau_\varepsilon\partial_tw_h)_{\mathcal{T}_h}
      \leq
      &
      \frac{\delta}{2}
      \sum_{\mathcal{K}\in\mathcal{T}_h}
      \tau_\varepsilon
      \norm[0]{\partial_tw_h}_\mathcal{K}^2
      +
      c\delta^{-1}
      \sum_{\mathcal{K}\in\mathcal{T}_h}
      \tau_\varepsilon
      \norm[0]{\overline{\nabla}w_h}_\mathcal{K}^2
      \\
      \leq
      &
      \frac{\delta}{2}
      \sum_{\mathcal{K}\in\mathcal{T}_h}
      \tau_\varepsilon
      \norm[0]{\partial_tw_h}_\mathcal{K}^2
      +
      c\delta^{-1}
      \sum_{\mathcal{K}\in\mathcal{T}_h}
      \varepsilon
      \norm[0]{\overline{\nabla}w_h}_{\mathcal{K}}^2
      +
      c\delta^{-1}
      \sum_{\mathcal{K}\in\mathcal{T}_h}
      \norm[0]{w_h}_\mathcal{K}^2.
    \end{split}
  \end{equation}
  $T_1$ and $T_2$ can be bounded using the Cauchy--Schwarz inequality,
  $\tfrac{1}{2}\envert[0]{\beta\cdot{n}} \leq \envert[0]{ \beta_s -
    \tfrac{1}{2}\beta\cdot{n} }$ for all $F\in\partial\mathcal{T}_h$,
  \cref{eq:eg_inv_4}, \cref{eq:setofusefulineq}, and Young's
  inequality:
  \begin{equation}
    \label{eq:ss_3_localtimestep}
    T_1+T_2
    \leq
    \tfrac{1}{2}\delta
    \sum_{\mathcal{K}\in\mathcal{T}_h}
    \tau_\varepsilon
    \norm[0]{\partial_tw_h}_\mathcal{K}^2
    +
    c\delta^{-1}
    \sum_{\mathcal{K}\in\mathcal{T}_h}
    \norm[0]{
      \envert[0]{
        \beta_s
        -
        \tfrac{1}{2}
        \beta\cdot{n}
      }^{1/2}
      \sbr{\boldsymbol{w}_h}
    }_{\partial\mathcal{K}}^2.
  \end{equation}
  Similarly $T_3$ and $T_4$ are bounded using the Cauchy--Schwarz
  inequality, \cref{lem:eg_inv}, \cref{eq:eg_inv_3},
  \cref{eq:betasinfmax}, ${\delta t_{\mathcal{K}}}\leq
  h_K\leq\varepsilon$ for
  $\mathcal{K}\in\mathcal{T}_h^d$, \cref{eq:setofusefulineq}, and
  Young's inequality. Note that we also make use of
  $\tilde{\varepsilon}\leq \varepsilon^{1/2}h_K^{-1/2}$ on $\mathcal{T}_h^{dx}$ since
  on $\mathcal{T}_h^d$, $\tilde{\varepsilon}=1$ and
  $h_K\le\varepsilon$ while on
  $\mathcal{T}_h^x$, $\tilde{\varepsilon}=\varepsilon^{1/2}$ and $h_K\le 1$:
  \begin{equation}
    \label{eq:ss_4_localtimestep}
    \begin{split}
      T_3+T_4
      \leq
      &
      c
      \sum_{\mathcal{K}\in\mathcal{T}_h^{dx}}
      \tau_\varepsilon
      \norm[0]{
        \envert[0]{
          \beta_s
          -
          \tfrac{1}{2}
          \beta\cdot{n}
        }^{1/2}
        \sbr{\boldsymbol{w}_h}
      }_{\mathcal{Q}_\mathcal{K}\setminus\mathcal{Q}_\mathcal{K}^0}
      \del[0]{
        \del[0]{\delta t_{\mathcal{K}}^{-1}+h_K^{-1}}
        \norm[0]{
          \sbr{\boldsymbol{w}_h}
        }_{\mathcal{Q}_\mathcal{K}\setminus\mathcal{Q}_\mathcal{K}^0}
        +
        h_K^{-1/2}
        \norm[0]{
          \partial_t{w}_h
        }_{\mathcal{K}}
      }
      \\
      \leq
      &
      c
      \sum_{\mathcal{K}\in\mathcal{T}_h^{dx}}
      \norm[0]{
        \envert[0]{
          \beta_s
          -
          \tfrac{1}{2}
          \beta\cdot{n}
        }^{1/2}
        \sbr{\boldsymbol{w}_h}
      }_{\mathcal{Q}_\mathcal{K}\setminus\mathcal{Q}_\mathcal{K}^0}
      \del[0]{
        \tilde{\varepsilon}
        \norm[0]{
          \sbr{\boldsymbol{w}_h}
        }_{\mathcal{Q}_\mathcal{K}\setminus\mathcal{Q}_\mathcal{K}^0}
        +
        \Delta t_{\mathcal{K}}^{1/2}
        \tilde{\varepsilon}
        \norm[0]{
          \partial_t{w}_h
        }_{\mathcal{K}}
      }
      \\
      \leq
      &
      c
      \sum_{\mathcal{K}\in\mathcal{T}_h^{dx}}
      \norm[0]{
        \envert[0]{
          \beta_s
          -
          \tfrac{1}{2}
          \beta\cdot{n}
        }^{1/2}
        \sbr{\boldsymbol{w}_h}
      }_{\mathcal{Q}_\mathcal{K}\setminus\mathcal{Q}_\mathcal{K}^0}
      \del[0]{
        \varepsilon^{1/2}h_K^{-1/2}
        \norm[0]{
          \sbr{\boldsymbol{w}_h}
        }_{\mathcal{Q}_\mathcal{K}\setminus\mathcal{Q}_\mathcal{K}^0}
        +
        \tau_\varepsilon^{1/2}
        \norm[0]{
          \partial_t{w}_h
        }_{\mathcal{K}}
      }
      \\
      \leq
      &
      \tfrac{1}{2}\delta
      \sum_{\mathcal{K}\in\mathcal{T}_h}
      \tau_\varepsilon
      \norm[0]{\partial_tw_h}_\mathcal{K}^2
      +
      c
      \sum_{\mathcal{K}\in\mathcal{T}_h}
      \varepsilon h_K^{-1}
      \norm[0]{\sbr[0]{\boldsymbol{w}_h}}_{\mathcal{Q}_{\mathcal{K}}}^2
      +
      c
      \del[0]{1+\delta^{-1}}
      \sum_{\mathcal{K}\in\mathcal{T}_h}
      \norm[0]{
        \envert[0]{
          \beta_s
          -
          \tfrac{1}{2}
          \beta\cdot{n}
        }^{1/2}
        \sbr{\boldsymbol{w}_h}
      }_{\partial\mathcal{K}}^2.
    \end{split}
  \end{equation}
  Similarly, to bound $T_5$ and $T_6$, we use the Cauchy--Schwarz
  inequality, \cref{lem:eg_inv}, \cref{eq:eg_inv_3}, that
  ${\delta t_{\mathcal{K}}}\leq h_K\leq\varepsilon$ for
  $\mathcal{K}\in\mathcal{T}_h^d$, \cref{eq:setofusefulineq},
  $\tilde{\varepsilon}\leq \varepsilon^{1/2}h_K^{-1/2}$ on $\mathcal{T}_h^{dx}$ and
  Young's inequality:
  \begin{equation}
    \label{eq:ss_5_localtimestep}
    T_5+T_6
    \leq
    \tfrac{1}{2}\delta
    \sum_{\mathcal{K}\in\mathcal{T}_h}
    \tau_\varepsilon
    \norm[0]{\partial_tw_h}_\mathcal{K}^2
    +
    c
    \sum_{\mathcal{K}\in\mathcal{T}_h}
    \varepsilon h_K^{-1}
    \norm[0]{\sbr[0]{\boldsymbol{w}_h}}_{\mathcal{Q}_{\mathcal{K}}}^2
    +
    c
    \del[0]{1 + \delta^{-1}}
    \sum_{\mathcal{K}\in\mathcal{T}_h}
    \norm[0]{
      \envert[0]{\tfrac{1}{2}\beta\cdot{n}}^{1/2}
      \varkappa_h
    }_{\mathcal{Q}_\mathcal{K}\cap\partial\mathcal{E}_N}^2.
  \end{equation}
  Combining
  \cref{eq:ss_0_localtimestep,eq:ss_1_localtimestep,eq:ss_2_localtimestep,eq:ss_3_localtimestep,eq:ss_4_localtimestep,eq:ss_5_localtimestep}
  and choosing $\delta=1/5$ we obtain
  \begin{equation}
    \label{eq:ss_6_localtimestep}
    \sum_{\mathcal{K}\in\mathcal{T}_h}
    \tau_\varepsilon
    \norm[0]{\partial_tw_h}_\mathcal{K}^2
    \leq
    a_h(\boldsymbol{w}_h,2\boldsymbol{y}_h)
    +
    c
    \tnorm{\boldsymbol{w}_h}_v^2.
  \end{equation}
  Adding $\tnorm{\boldsymbol{w}_h}_v^2$ to both sides of
  \cref{eq:ss_6_localtimestep}, the first bound in \cref{eq:sv_0}
  yields the result.
\end{proof}

We end this section by proving \cref{eq:new_inf_sup_s_norm}.

\begin{proof}[Proof of \cref{eq:new_inf_sup_s_norm}]
  By \cref{eq:proj_time_derivative} and using that
  $\tau_{\varepsilon} \le c\varepsilon$, because on $\mathcal{T}_h^{dx}$,
  $\tau_\varepsilon\leq \Delta t\leq c \delta t\leq c \varepsilon$ and on
  $\mathcal{T}_h^{c}$,
  $\tau_{\varepsilon}=\Delta t_{\mathcal{K}}\varepsilon\le
  \varepsilon (\le c \varepsilon)$, we find
  \begin{equation*}
    \begin{split}
      \sum_{\mathcal{K}\in\mathcal{T}_h}
      \tau_\varepsilon
      \norm[0]{\partial_t\del{\Pi_h\del{\varphi w_h}}}_{\mathcal{K}}^2
      \leq
      &
      c
      \del[2]{
        \sum_{\mathcal{K}\in\mathcal{T}_h}
        \tau_\varepsilon
        \norm[0]{\partial_t\del{\varphi w_h}}_{\mathcal{K}}^2
        +
        \sum_{\mathcal{K}\in\mathcal{T}_h}
        \tau_\varepsilon
        \norm[0]{\overline{\nabla}\del{\varphi w_h}}_{\mathcal{K}}^2
      }
      \\
      \leq
      &
      c
      \del[3]{
        c_T^2
        \del[2]{
          \sum_{\mathcal{K}\in\mathcal{T}_h}
          \tau_\varepsilon
          \norm[0]{\partial_t{w_h}}_{\mathcal{K}}^2
          +
          c
          \sum_{\mathcal{K}\in\mathcal{T}_h}
          \varepsilon
          \norm[0]{\overline{\nabla}{w_h}}_{\mathcal{K}}^2
        }
        +
        \sum_{\mathcal{K}\in\mathcal{T}_h}
        \norm[0]{{w_h}}_{\mathcal{K}}^2
      }
      \\
      \leq
      &
      c_T^2
      \tnorm{\boldsymbol{w}_h}_s^2.
    \end{split}
  \end{equation*}
  Therefore, using \cref{lem:stab_fortin}, we conclude that
  \begin{equation}
    \label{eq:stab_fortin_s_localtimestep}
    \tnorm{\boldsymbol{\Pi}_h\del{\varphi\boldsymbol{w}_h}}_s
    \leq
    c_T
    \tnorm{\boldsymbol{w}_h}_s.
  \end{equation}
  \Cref{eq:new_inf_sup_s_norm} can now be shown to hold after
  combining \cref{eq:stab_fortin_s_localtimestep} with
  \cref{lem:zh_bounded_localtimestep,lem:s_inf_sup_bounded_below_localtimestep}.
\end{proof}

\subsection{The inf-sup condition with respect to
  $|\mkern-1.5mu|\mkern-1.5mu|\cdot|\mkern-1.5mu|\mkern-1.5mu|_{ss}$}
\label{ss:infsupsupgnorm}

\begin{proof}[Proof of \cref{thm:supginfsup}]
  We construct the test function
  $\boldsymbol{\kappa}_h:=\del[0]{\kappa_h,\varsigma_h}$ such that for
  $\mathcal{K}\in\mathcal{T}_h$,
  $\kappa_h|_{\mathcal{K}}:= \tfrac{\delta
    t_{\mathcal{K}}h_K^2}{\delta t_{\mathcal{K}}+h_K}
  \Pi_h\del[0]{\beta\cdot\nabla w_h }$ while $\varsigma_h$ vanishes on
  all faces of $\mathcal{F}_h$. We first show that there exists a
  positive constant $c_1$, independent of $h_K$,
  $\delta t_{\mathcal{K}}$, $\varepsilon$, and $T$ such that the following
  holds:
  \begin{equation}
    \label{eq:supginfsup1}
    \tnorm{\boldsymbol{\kappa}_h}_{s}
    \leq
    c_1
    \norm{w_h}_{sd}.
  \end{equation}
  We bound each term of $\tnorm{\cdot}_s$, starting with the volume
  terms. Noting that
  $\tfrac{\delta t_{\mathcal{K}}h_K^2}{\delta t_{\mathcal{K}}+h_K} \le
  1$ and using the definition of $\norm{\cdot}_{sd}$ in
  \cref{eq:sd_norm}, we have:
  \begin{equation}
    \label{eq:kappahKwhsd-1}
    \sum_{\mathcal{K}\in\mathcal{T}_h}
    \norm{\kappa_h}_{\mathcal{K}}^2
    \leq
    \norm{w_h}_{sd}^2.
  \end{equation}
  The diffusive volume term is bounded using $\varepsilon\leq 1$,
  $\tfrac{\delta t_{\mathcal{K}}h_K^2}{\delta
    t_{\mathcal{K}}+h_K}h_K^{-2}\leq 1$ and \cref{eq:eg_inv_2}:
  \begin{equation}
    \label{eq:kappahKwhsd-2}
    \sum_{\mathcal{K}\in\mathcal{T}_h}
    \varepsilon\norm[0]{\overline{\nabla}\kappa_h}_{\mathcal{K}}^2
    \leq
    c
    \sum_{\mathcal{K}\in\mathcal{T}_h}
    \del{\tfrac{\delta t_{\mathcal{K}}h_K^2}{\delta
        t_{\mathcal{K}}+h_K}}^2
    h_K^{-2}
    \norm{{\Pi_h\del[0]{\beta\cdot\nabla
          w_h}}}_{\mathcal{K}}^2
    \leq
    c
    \norm{w_h}_{sd}^2.
  \end{equation}
  For the time derivative volume term, we need
  \cref{eq:setofusefulineq} and \cref{eq:eg_inv_1}:
  \begin{equation}
    \label{eq:kappahKwhsd-3}
    \sum_{\mathcal{K}\in\mathcal{T}_h}
    \tau_\varepsilon
    \norm{\partial_t\kappa_h}_{\mathcal{K}}^2
    \leq
    c
    \sum_{\mathcal{K}\in\mathcal{T}_h}
    \delta t_{\mathcal{K}}
    \del[0]{\delta t_{\mathcal{K}}^{-1}+h_K^{-1}}^2
    \del{\tfrac{\delta t_{\mathcal{K}}h_K^2}{\delta
        t_{\mathcal{K}}+h_K}}^2
    \norm{\Pi_h\del[0]{\beta\cdot\nabla w_h}}_{\mathcal{K}}^2
    \leq
    c
    \norm{w_h}_{sd}^2.
  \end{equation}
  Next we turn to the facet terms. To bound the diffusive facet term,
  we apply \cref{eq:eg_inv_3}:
  \begin{equation}
    \label{eq:kappahKwhsd-4}
    \begin{split}
      \sum_{\mathcal{K}\in\mathcal{T}_h}
      \varepsilon h_K^{-1}
      \norm[0]{\sbr{\boldsymbol{\kappa}_h}}_{\mathcal{Q}_{\mathcal{K}}}^2
      &
      =
      \sum_{\mathcal{K}\in\mathcal{T}_h}
      \varepsilon h_K^{-1}
      \del{\tfrac{\delta t_{\mathcal{K}}h_K^2}{\delta
          t_{\mathcal{K}}+h_K}}^2
      \norm[0]{\Pi_h\del{\beta\cdot\nabla w_h}}_{\mathcal{Q}_{\mathcal{K}}}^2
      \\
      &\leq
      c
      \sum_{\mathcal{K}\in\mathcal{T}_h}
      {\tfrac{\delta t_{\mathcal{K}}h_K^2}{\delta
          t_{\mathcal{K}}+h_K}}
      \norm[0]{\Pi_h\del{\beta\cdot\nabla w_h}}_{\mathcal{K}}^2
      \le
      c
      \norm[0]{w_h}_{sd}^2.
    \end{split}
  \end{equation}
  We use \cref{eq:eg_inv_4} and that
  $\del[0]{\delta t_{\mathcal{K}}^{-1/2}+h_K^{-1/2}}^2\tfrac{\delta
    t_{\mathcal{K}}h_K}{\delta t_{\mathcal{K}}+h_K} \le 2$ to bound
  the advective facet term:
  \begin{equation}
    \label{eq:kappahKwhsd-5}
    \begin{split}
      \sum_{\mathcal{K}\in\mathcal{T}_h}
      \norm[0]{
        \envert[0]{\beta_s-\tfrac{1}{2}\beta\cdot n}^{1/2}
        \sbr[0]{\boldsymbol{\kappa}_h}
      }_{\partial\mathcal{K}}^2
      &\leq
      c
      \sum_{\mathcal{K}\in\mathcal{T}_h}
      \del{\tfrac{\delta t_{\mathcal{K}}h_K^2}{\delta
          t_{\mathcal{K}}+h_K}}^2
      \norm[0]{
        \Pi_h\del[0]{\beta\cdot\nabla w_h}
      }_{\partial\mathcal{K}}^2
      \\
      &\leq
      c
      \sum_{\mathcal{K}\in\mathcal{T}_h}
      {\tfrac{\delta t_{\mathcal{K}}h_K^2}{\delta t_{\mathcal{K}}+h_K}}
      \norm[0]{
        \Pi_h\del[0]{\beta\cdot\nabla w_h}
      }_{\mathcal{K}}^2
      \le
      c
      \norm{w_h}_{sd}^2.
    \end{split}
  \end{equation}
  The Neumann boundary term vanishes since $\varsigma_h\equiv 0$. We
  can therefore conclude \cref{eq:supginfsup1} from
  \cref{eq:kappahKwhsd-1,eq:kappahKwhsd-2,eq:kappahKwhsd-3,eq:kappahKwhsd-4,eq:kappahKwhsd-5}.

  We next show that there exists a positive constant $c_2$,
  independent of $h_K$, $\delta t_{\mathcal{K}}$, $\varepsilon$, and $T$ such
  that
  \begin{equation}
    \label{eq:supginfsup2}
    \norm{w_h}_{sd}^2-
    c_2
    \tnorm{\boldsymbol{w}_h}_s
    \norm{w_h}_{sd}
    \leq
    a_{h}(\boldsymbol{w}_h,\boldsymbol{\kappa}_h).
  \end{equation}
  We first write the advective part of the bilinear form as:
  \begin{equation}
    \label{eq:supginfsup2_1}
    a_{h,c}(\boldsymbol{w}_h,\boldsymbol{\kappa}_h)
    =
    \del{\nabla\cdot\del{\beta w_h},\kappa_h}_{\mathcal{T}_h}
    +
    \langle
    \del{\beta_s-\beta\cdot n}
    \sbr{\boldsymbol{w}_h}
    ,\kappa_h
    \rangle_{\partial\mathcal{T}_h}
    =: T_1 + T_2.
  \end{equation}
  We bound $T_1$ using the definition of the projection operator
  $\Pi_h$:
  \begin{equation}
    \label{eq:supginfsup2_2}
    \begin{split}
      T_1
      &=
      \del[1]{
        \del{I-\Pi_h}\del[0]{\beta\cdot\nabla w_h}
        ,
        \tfrac{\delta t_{\mathcal{K}}h_K^2}{\delta t_{\mathcal{K}}+h_K}
        \Pi_h\del{\beta\cdot\nabla w_h}
      }_{\mathcal{T}_h}
      +
      \del[1]{
        \Pi_h\del[0]{\beta\cdot\nabla w_h}
        ,
        \tfrac{\delta t_{\mathcal{K}}h_K^2}{\delta t_{\mathcal{K}}+h_K}
        \Pi_h\del{\beta\cdot\nabla w_h}
      }_{\mathcal{T}_h}
      \\
      &=
      \norm{w_h}_{sd}^2.
    \end{split}
  \end{equation}
  For $T_2$, we note that
  $\envert[0]{\beta_s-\beta\cdot n}\leq
  2\envert[0]{\beta_s-\tfrac{1}{2}\beta\cdot n}$ for any
  $F\in\mathcal{T}_h$. Then, also using \cref{eq:eg_inv_4} and
  H\"older's inequality for sums,
  \begin{equation}
    \label{eq:supginfsup2_3}
    T_2
    \leq
    c
    \sum_{\mathcal{K}\in\mathcal{T}_h}
    \norm[1]{
      \envert[1]{\beta_s-\tfrac{1}{2}\beta\cdot n}^{1/2}
      \sbr{\boldsymbol{w}_h}
    }_{\partial\mathcal{K}}
    \tfrac{\delta t_{\mathcal{K}}h_K^2}{\delta t_{\mathcal{K}}+h_K}
    \norm{
      \Pi_h\del{\beta\cdot\nabla w_h}
    }_{\partial\mathcal{K}}
    \leq
    c
    \tnorm{ \boldsymbol{w}_h }_{s}
    \norm{ w_h }_{sd}.
  \end{equation}
  For the diffusive part of the bilinear form, we write:
  \begin{equation}
    \label{eq:supginfsup2_4}
    \begin{split}
      a_{h,d}(\boldsymbol{w}_h,\boldsymbol{\kappa}_h)
      &=
      \del[0]{
        \varepsilon\overline{\nabla}w_h,
        \overline{\nabla}\kappa_h
      }_{\mathcal{T}_h}
      -
      \langle
      \varepsilon\sbr{\boldsymbol{w}_h},
      \overline{\nabla}_{\bar{n}}
      \kappa_h
      \rangle_{\mathcal{Q}_h}
      -
      \langle
      \varepsilon{\kappa_h},
      \overline{\nabla}_{\bar{n}}
      w_h
      \rangle_{\mathcal{Q}_h}
      +
      \langle
      \alpha\varepsilon h_K^{-1}
      \sbr{\boldsymbol{w}_h}
      ,
      {\kappa_h}
      \rangle_{\mathcal{Q}_h}
      \\
      &=:
      I_1+I_2+I_3+I_4.
    \end{split}
  \end{equation}
  For $I_1$, we apply the Cauchy--Schwarz inequality,
  \cref{eq:eg_inv_2}, and H\"older's inequality for sums:
  \begin{equation}
    \label{eq:supginfsup2_5}
    \begin{split}
      I_1
      &\leq
      \sum_{\mathcal{K}\in\mathcal{T}_h}
      \varepsilon
      \norm[0]{\overline{\nabla}w_h}_{\mathcal{K}}
      \tfrac{\delta t_{\mathcal{K}}h_K^2}{\delta t_{\mathcal{K}}+h_K}
      \norm[0]{
        \overline{\nabla}
        \del[0]{
          \Pi_h\del[0]{\beta\cdot\nabla w_h}
        }
      }_{\mathcal{K}}
      \\
      &\leq
      c
      \sum_{\mathcal{K}\in\mathcal{T}_h}
      \varepsilon^{1/2}
      \norm[0]{\overline{\nabla}w_h}_{\mathcal{K}}
      \del[1]{\tfrac{\delta t_{\mathcal{K}}h_K^2}{\delta
          t_{\mathcal{K}}+h_K}}^{1/2}
      \norm[0]{
        \Pi_h\del[0]{\beta\cdot\nabla w_h}
      }_{\mathcal{K}}
      \leq
      c
      \tnorm{\boldsymbol{w}_h}_s\norm{w_h}_{sd}.
    \end{split}
  \end{equation}
  For $I_2$, we apply the Cauchy--Schwarz inequality,
  \cref{eq:eg_inv_2,eq:eg_inv_3}, and H\"older's inequality for
  sums:
  \begin{equation}
    \label{eq:supginfsup2_6}
    \begin{split}
      I_2
      &\leq
      \sum_{\mathcal{K}\in\mathcal{T}_h}
      \varepsilon\norm[0]{\sbr{\boldsymbol{w}_h}}_{\mathcal{Q}_{\mathcal{K}}}
      \tfrac{\delta t_{\mathcal{K}}h_K^2}{\delta t_{\mathcal{K}}+h_K}
      \norm[0]{
        \overline{\nabla}
        \del[0]{
          \Pi_h\del[0]{\beta\cdot\nabla w_h}
        }
      }_{\mathcal{Q}_{\mathcal{K}}}
      \\
      &\leq
      c
      \sum_{\mathcal{K}\in\mathcal{T}_h}
      \varepsilon^{1/2}h_K^{-1/2}\norm[0]{\sbr{\boldsymbol{w}_h}}_{\mathcal{Q}_{\mathcal{K}}}
      \del[1]{\tfrac{\delta t_{\mathcal{K}}h_K^2}{\delta
          t_{\mathcal{K}}+h_K}}^{1/2}
      \norm[0]{
        \Pi_h\del[0]{\beta\cdot\nabla w_h}
      }_{\mathcal{K}}
      \leq
      c
      \tnorm{\boldsymbol{w}_h}_s\norm{w_h}_{sd}.
    \end{split}
  \end{equation}
  Similarly for $I_3$ and $I_4$, we apply the Cauchy--Schwarz inequality,
  \cref{eq:eg_inv_3}, and H\"older's inequality for sums:
  \begin{equation}
    \label{eq:supginfsup2_8}
    I_3 + I_4 \le c
      \tnorm{\boldsymbol{w}_h}_s\norm{w_h}_{sd}.
  \end{equation}
  Combining
  \cref{eq:supginfsup2_1,eq:supginfsup2_2,eq:supginfsup2_3,eq:supginfsup2_4,eq:supginfsup2_5,eq:supginfsup2_6,eq:supginfsup2_8},
  we conclude \cref{eq:supginfsup2}.

  Combining \cref{eq:supginfsup2} and \cref{eq:supginfsup1} then
  yields:
  \begin{equation*}
    c_1^{-1}
    \del{
      \norm[0]{w_h}_{sd}
      -
      c_2
      \tnorm{\boldsymbol{w}_h}_{s}
    }
    \leq
    \frac{a_h({\boldsymbol{w}_h,\boldsymbol{\kappa}_h})}{
      c_1
      \norm{w_h}_{sd}
    }
    \leq
    \frac{a_h({\boldsymbol{w}_h,\boldsymbol{\kappa}_h})}{\tnorm{\boldsymbol{\kappa}_h}_s}
    \leq
    \sup_{\boldsymbol{v}_h\in \boldsymbol{V}_h}
    \frac{a_h({\boldsymbol{w}_h,\boldsymbol{v}_h})}{\tnorm{\boldsymbol{v}_h}_s}.
  \end{equation*}
  By combining the above with \cref{eq:new_inf_sup_s_norm},
  \begin{equation*}
    \del[1]{1+(c_{1}^{-1}c_{2}+1)c_T}
    \sup_{\boldsymbol{v}_h\in \boldsymbol{V}_h}
    \frac{a_h({\boldsymbol{w}_h,\boldsymbol{v}_h})}{\tnorm{\boldsymbol{v}_h}_s}
    \geq
    c_{1}^{-1}
    \norm{w_h}_{sd}
    +\tnorm{\boldsymbol{w}_h}_s
    \geq
    c
    \tnorm{\boldsymbol{w}_h}_{ss},
  \end{equation*}
  proving \cref{eq:supginfsup}.
\end{proof}

\section{Error analysis}
\label{s:conv_anal}

The following projection estimates for $\Pi_h$ and
$\Pi_h^{\mathcal{F}}$ were shown to hold for any
$u|_{\mathcal{K}} \in H^{(p_t+1,p_s+1)}(\mathcal{K})$,
$\mathcal{K} \in \mathcal{T}_h$, see \citet[Lemma 5.2]{Kirk:2019},
\citet[Lemma 6.1 and Remark 6.2]{Sudirham:2006}, and \citet[Lemmas
3.13 and 3.17]{Georgoulis:thesis}:
\begin{subequations}
  \label{eq:proj_est_conv}
  \begin{align}
    \norm[0]{u-\Pi_hu}_\mathcal{K}^2
    &
      \leq
      c
      \del[1]{h_K^{2p_s+2}+{\delta t_{\mathcal{K}}}^{2p_t+2}}
      \norm[0]{u}_{
      {H^{(p_t+1,p_s+1)}(\mathcal{K})}
      }^2,
      \label{eq:proj_est_conv_1}
    \\
    \norm[0]{\overline{\nabla}\del[0]{u-\Pi_hu}}_\mathcal{K}^2
    &
      \leq
      c
      \del[1]{h_K^{2p_s}+{\delta t_{\mathcal{K}}}^{2p_t+2}}
      \norm[0]{u}_{
      {H^{(p_t+1,p_s+1)}(\mathcal{K})}
      }^2,
      \label{eq:proj_est_conv_2}
    \\
    \norm[0]{\partial_t\del[0]{u-\Pi_hu}}_\mathcal{K}^2
    &
      \leq
      c
      \del[1]{h_K^{2p_s}+{\delta t_{\mathcal{K}}}^{2p_t}}
      \norm[0]{u}_{
      {H^{(p_t+1,p_s+1)}(\mathcal{K})}
      }^2,
      \label{eq:proj_est_conv_3}
    \\
    \norm[0]{
    \overline{\nabla}_{\bar{{n}}}\del[0]{u-\Pi_hu}
    }_{\mathcal{Q}_\mathcal{K}}^2
    &
      \leq
      c
      \del[1]{h_K^{2p_s-1}+h_K^{-1}{\delta t_{\mathcal{K}}}^{2p_t+2}}
      \norm[0]{u}_{
      {H^{(p_t+1,p_s+1)}(\mathcal{K})}
      }^2,
      \label{eq:proj_est_conv_4}
    \\
    \norm[0]{u-\Pi_hu}_{\partial\mathcal{K}}^2
    &
      \leq
      c
      \del[1]{h_K^{2p_s+1}+{\delta t_{\mathcal{K}}}^{2p_t+1}}
      \norm[0]{u}_{
      {H^{(p_t+1,p_s+1)}(\mathcal{K})}
      }^2,
      \label{eq:proj_est_conv_5}
    \\
    \norm[0]{u-\Pi^{\mathcal{F}}_hu}_{\partial\mathcal{K}}^2
    &
      \leq
      c
      \del[1]{h_K^{2p_s+1}+{\delta t_{\mathcal{K}}}^{2p_t+1}}
      \norm[0]{u}_{
      {H^{(p_t+1,p_s+1)}(\mathcal{K})}
      }^2.
      \label{eq:proj_est_conv_6}
  \end{align}
\end{subequations}
Let us define $h := \max_{\mathcal{K}\in\mathcal{T}_h}h_K$ and
$\delta t := \max_{\mathcal{K}\in\mathcal{T}_h}\delta
t_{\mathcal{K}}$. An immediate consequence of \cref{eq:proj_est_conv}
is the following estimate.

\begin{lemma}
  \label{lem:uminPihusbound}
  Let $u$, with $u|_{\mathcal{K}} \in H^{(p_t+1,p_s+1)}(\mathcal{K})$
  for all $\mathcal{K} \in \mathcal{T}_h$, and define
  $\boldsymbol{u} := (u, u|_{\Gamma})$. Let
  $\boldsymbol{\Pi}_h\boldsymbol{u} = (\Pi_hu,\Pi_h^{\mathcal{F}}u)$.
  Then,
  \begin{equation*}
    \tnorm{\boldsymbol{u}-\boldsymbol{\Pi}_h\boldsymbol{u}}_{ss}^2
    \le
    c\sbr{ h^{2p_s}(h + \varepsilon + \tilde{\varepsilon} \delta t)
      + \delta t^{2p_t}
      (\delta t + \varepsilon h^{-1}\delta t)},
  \end{equation*}
  where the constant $c$ depends on
  $\sum_{\mathcal{K}\in\mathcal{T}_h}\norm[0]{u}_{H^{(p_t+1,p_s+1)}(\mathcal{K})}$.
\end{lemma}
\begin{proof}
  By \cref{eq:proj_est_conv_1},
  \begin{equation}
    \label{eq:conv_anal_2}
    \norm[0]{
      u-\Pi_hu
    }_\mathcal{K}^2
    \leq
    c
    \del[1]{
      h_K^{2p_s+2}
      +{\delta t_{\mathcal{K}}}^{2p_t+2}
    }
    \norm[0]{u}_{H^{(p_t+1,p_s+1)}(\mathcal{K})}^2.
  \end{equation}
  Next, \cref{eq:proj_est_conv_2} gives us:
  \begin{equation}
    \label{eq:conv_anal_4}
    \varepsilon
    \norm[0]{
      \overline{\nabla}\del{u-\Pi_hu}
    }_\mathcal{K}^2
    \leq
    c
    \varepsilon
    \del[1]{h_K^{2p_s}+{\delta t_{\mathcal{K}}}^{2p_t+2}}
    \norm[0]{u}_{H^{(p_t+1,p_s+1)}(\mathcal{K})}^2.
  \end{equation}
  For the advective facet terms, we use
  \cref{eq:proj_est_conv_5,eq:proj_est_conv_6} and the triangle
  inequality:
  \begin{multline}
    \label{eq:conv_anal_5}
    \norm[1]{
      \envert[0]{\beta_s-\tfrac{1}{2}\beta\cdot{n}}^{1/2}
      \del[1]{
        \del[0]{u-\Pi_hu}-\del[0]{\gamma(u)-\Pi_h^\mathcal{F}\gamma(u)}
      }
    }_{\partial\mathcal{K}}^2
    +
    \norm[0]{
      \envert[0]{
        \tfrac{1}{2}\beta\cdot{n}
      }^{1/2}
      \del[0]{\gamma(u)-\Pi_h^\mathcal{F}\gamma(u)}
    }_{\partial\mathcal{K}\cap\partial\mathcal{E}_N}^2
    \\
    \leq
    c
    \del[0]{h_K^{2p_s+1}+{\delta t_{\mathcal{K}}}^{2p_t+1}}
    \norm[0]{u}_{H^{(p_t+1,p_s+1)}(\mathcal{K})}^2.
  \end{multline}
  Similarly, for the diffusive facet term, we again apply the triangle
  inequality and \cref{eq:proj_est_conv_5,eq:proj_est_conv_6}:
  \begin{equation}
    \label{eq:conv_anal_6}
    \varepsilon h_K^{-1}
    \norm[0]{
      \del[0]{u-\Pi_hu}
      -
      \del[0]{\gamma(u)-\Pi_h^{\mathcal{F}}\gamma(u)}
    }_{\mathcal{Q}_\mathcal{K}}^2
    \leq
    c
    \varepsilon
    \del[0]{h_K^{2p_s}+h_K^{-1}{\delta t_{\mathcal{K}}}^{2p_t+1}}
    \norm[0]{u}_{H^{(p_t+1,p_s+1)}(\mathcal{K})}^2.
  \end{equation}
  For the streamline derivative term, we use
  \cref{eq:proj_est_conv_2,eq:proj_est_conv_3} and that
  $\tfrac{\delta t_{\mathcal{K}}h_K^2}{\delta t_{\mathcal{K}}+h_K}\leq
  \delta t_{\mathcal{K}}h_K$:
  \begin{equation}
    \label{eq:conv_anal_supg}
    \begin{split}
      \tfrac{\delta t_{\mathcal{K}}h_K^2}{\delta t_{\mathcal{K}}+h_K}
      \norm{\Pi_h\del[0]{\beta\cdot\nabla\del[0]{u-\Pi_hu}}}_{\mathcal{K}}^2
      \leq
      &
      c
      \delta t_{\mathcal{K}}h_K
      \del[1]{
        \norm[0]{
          \overline{\nabla}\del[0]{u-\Pi_hu}
        }_{\mathcal{K}}^2
        +
        \norm[0]{
          \partial_t\del[0]{u-\Pi_hu}
        }_{\mathcal{K}}^2
      }
      \\
      \leq
      &
      c
      \delta t_{\mathcal{K}}h_K
      \del[0]{h_K^{2p_s}+{\delta t_{\mathcal{K}}}^{2p_t}}
      \norm[0]{u}_{H^{(p_t+1,p_s+1)}(\mathcal{K})}^2.
    \end{split}
  \end{equation}
  Finally, for the time-derivative term, using
  \cref{eq:proj_est_conv_3},
  \begin{equation}
    \label{eq:conv_anal_3}
    \begin{split}
      \tau_{\varepsilon}
      \norm[0]{
        \partial_t\del{u-\Pi_hu}
      }_\mathcal{K}^2
      &\le
      \begin{cases}
        c
        \del[0]{h_K^{2p_s}{\delta t_{\mathcal{K}}}+{\delta t_{\mathcal{K}}}^{2p_t+1}}
        \norm[0]{u}_{H^{(p_t+1,p_s+1)}(\mathcal{K})}^2
        &
        \text{if } \mathcal{K}\in\mathcal{T}_h^d,
        \\
	c
        \varepsilon^{1/2}
        \del[0]{h_K^{2p_s}{\delta t_{\mathcal{K}}}+{\delta t_{\mathcal{K}}}^{2p_t+1}}
        \norm[0]{u}_{H^{(p_t+1,p_s+1)}(\mathcal{K})}^2
        &
        \text{if } \mathcal{K}\in\mathcal{T}_h^x,
        \\
	c
        \varepsilon
        \del[0]{h_K^{2p_s}{\delta t_{\mathcal{K}}}+{\delta t_{\mathcal{K}}}^{2p_t+1}}
        \norm[0]{u}_{H^{(p_t+1,p_s+1)}(\mathcal{K})}^2
        &
        \text{if } \mathcal{K}\in\mathcal{T}_h^c.
      \end{cases}
      \\
      &\le
      c
      \tilde{\varepsilon}\del[0]{h_K^{2p_s}{\delta t_{\mathcal{K}}}+{\delta t_{\mathcal{K}}}^{2p_t+1}}
      \norm[0]{u}_{H^{(p_t+1,p_s+1)}(\mathcal{K})}^2.
    \end{split}
  \end{equation}
  The result follows after combining
  \cref{eq:conv_anal_2,eq:conv_anal_4,eq:conv_anal_5,eq:conv_anal_supg,eq:conv_anal_6,eq:conv_anal_3}
  and summing over all $\mathcal{K} \in \mathcal{T}_h$.
\end{proof}

The following lemma will be used to prove the global error estimate of
\cref{thm:global_err_est}.

\begin{lemma}
  \label{lem:ahuminPihu}
  Let $u$, with $u|_{\mathcal{K}} \in H^{(p_t+1,p_s+1)}(\mathcal{K})$
  for all $\mathcal{K} \in \mathcal{T}_h$, solve \cref{eq:st_adr} and
  define $\boldsymbol{u} := (u, \lambda)$ with
  $\lambda = u|_{\Gamma}$. Let
  $\boldsymbol{\Pi}_h\boldsymbol{u} = (\Pi_hu,\Pi_h^{\mathcal{F}}u)$
  and let
  $\boldsymbol{u}_h=\del[0]{u_h,\lambda_h} \in \boldsymbol{V}_h$ be
  the solution to \cref{eq:st_hdg_adr_compact}. The following holds:
  \begin{multline*}
    |a_h(\boldsymbol{u}-\boldsymbol{\Pi}_h\boldsymbol{u}, \boldsymbol{v}_h)|
    \\
    \le
    \sbr[2]{
    c	\tnorm{\boldsymbol{u}-\boldsymbol{\Pi}_h\boldsymbol{u}}_{ss}
    +
    c	\norm[0]{\envert[0]{\beta_s- \tfrac{1}{2}\beta\cdot{n}}^{1/2}
      \del[0]{u-\Pi_hu}}_{\partial\mathcal{T}_h}
    +\del[1]{
      \sum_{\mathcal{K}\in\mathcal{T}_h}
      \varepsilon h_K
      \norm[0]{
        \overline{\nabla}_{\bar{{n}}}\del[0]{u-\Pi_hu}
      }_{\mathcal{Q}_{\mathcal{K}}}^2
    }^{1/2}
    }
    \tnorm{\boldsymbol{v}_h}_s.
  \end{multline*}
\end{lemma}
\begin{proof}
  We start with the advective part of $a_h(\cdot, \cdot)$. Writing
  $\zeta^+ \beta \cdot n =
  (\beta\cdot{n}+\envert[0]{\beta\cdot{n}})/2$ and using the triangle
  inequality,
  \begin{multline*}
    |a_{h,c}(\boldsymbol{u}-\boldsymbol{\Pi}_h\boldsymbol{u},\boldsymbol{v}_h)|
    \leq
    \envert[0]{ \del[0]{\beta \del[0]{u-\Pi_hu},\nabla v_h}_{\mathcal{T}_h} }
    +
    \envert[0]{
      \langle
      \tfrac{1}{2}
      \del{\beta\cdot{n}+\envert[0]{\beta\cdot{n}}}
      \del[0]{\lambda-\Pi_h^{\mathcal{F}}\lambda}
      ,
      \mu_h
      \rangle_{\partial\mathcal{E}_N}
    }
    \\
    +
    \envert[0]{
      \langle
      \del{\beta\cdot{n}}
      \del[0]{\lambda-\Pi_h^{\mathcal{F}}\lambda}
      +
      \beta_s \sbr{\boldsymbol{u}-\boldsymbol{\Pi}_h\boldsymbol{u}}
      ,\sbr{\boldsymbol{v}_h}
      \rangle_{\partial\mathcal{T}_h}
    }
    =:I_1 + I_2 + I_3.
  \end{multline*}
  To bound $I_1$, we follow the proof of \citet[Theorem
  5.1]{Ayuso:2009} by noting that if $\beta_0 = (1,\bar{\beta}_0)$
  then
  ${\del[0]{{\beta_0} \del[0]{u-\Pi_hu},\nabla v_h}_{\mathcal{T}_h}
  }=0$ and
  ${\del[0]{\del[0]{\beta-\beta_0} \del[0]{u-\Pi_hu},\nabla
      v_h}_{\mathcal{T}_h}
  }={\del[0]{\del[0]{\bar{\beta}-\bar{\beta}_0}
      \del[0]{u-\Pi_hu},\overline{\nabla} v_h}_{\mathcal{T}_h}
  }$. Then, using the Cauchy--Schwarz inequality,
  \cref{eq:betaminbeta0}, \cref{eq:eg_inv_2}, and H\"older's
  inequality for sums, we obtain
  \begin{equation*}
    I_1
    \leq
    \sum_{\mathcal{K}\in\mathcal{T}_h}
    c
    \norm[0]{u-\Pi_hu}_\mathcal{K}
    \norm[0]{v_h}_\mathcal{K}.
  \end{equation*}
  Using the Cauchy--Schwarz inequality, we bound $I_2$ as:
  \begin{equation*}
    I_2
    \leq
    c
    \norm[0]{
      \envert[0]{\tfrac{1}{2}\beta\cdot{n}}^{1/2}
      \del[0]{\lambda-\Pi_h^{\mathcal{F}}\lambda}
    }_{\partial\mathcal{E}_N}
    \norm[0]{
      \envert[0]{\tfrac{1}{2}\beta\cdot{n}}^{1/2}
      \mu_h
    }_{\partial\mathcal{E}_N}.
  \end{equation*}
  With
  $ \beta\cdot n \leq {\sup\envert[0]{\beta\cdot{n}}} \leq {2}
  \del[0]{ \sup\envert[0]{\beta\cdot{n}} - \tfrac{1}{2}\beta\cdot{n}
  }$, for all $F\in\partial\mathcal{T}_h$, and the Cauchy--Schwarz
  inequality, we bound $I_3$ as:
  \begin{equation*}
    \begin{split}
      I_3
      &
      \leq
      c
      \envert[0]{
        \langle
        \del[0]{
          \sup\envert[0]{\beta\cdot{n}}
          -
          \tfrac{1}{2}\beta\cdot{n}
        }
        \del[0]{
          \lambda-\Pi_h^{\mathcal{F}}\lambda
          +
          \sbr{\boldsymbol{u}-\boldsymbol{\Pi}_h\boldsymbol{u}}
        }
        ,\sbr{\boldsymbol{v}_h}
        \rangle_{\partial\mathcal{T}_h}
      }
      \\
      &
      \leq
      c
      \sum_{\mathcal{K}\in\mathcal{T}_h}
      \norm[0]{
        \envert[0]{
          \beta_s- \tfrac{1}{2}\beta\cdot{n}
        }^{1/2}
        \del[0]{u-\Pi_hu}
      }_{\partial\mathcal{K}}
      \norm[0]{
        \envert[0]{
          \beta_s -
          \tfrac{1}{2}\beta\cdot{n}
        }^{1/2}
        \sbr{\boldsymbol{v}_h}
      }_{\partial\mathcal{K}}.
    \end{split}
  \end{equation*}
  Collecting the bounds for $I_1$, $I_2$, and $I_3$, and using
  H\"older's inequality for sums,
  \begin{equation}
    \label{eq:ahcboundupiu}
    |a_{h,c}(\boldsymbol{u}-\boldsymbol{\Pi}_h\boldsymbol{u},\boldsymbol{v}_h)|
    \le
    \del[1]{
      c
      \tnorm{\boldsymbol{u}-\boldsymbol{\Pi}_h\boldsymbol{u}}_{ss}
      +
      c
      \norm[0]{
        \envert[0]{\beta_s- \tfrac{1}{2}\beta\cdot{n}}^{1/2}
        \del[0]{u-\Pi_hu}}_{\partial\mathcal{T}_h}
    }
    \tnorm{\boldsymbol{v}_h}_{s}.
  \end{equation}
  We now proceed with the diffusive part of $a_h(\cdot, \cdot)$. By
  the triangle inequality,
  \begin{equation*}
    \begin{split}
    \envert[0]{a_{h,d}(\boldsymbol{u}-\boldsymbol{\Pi}_h\boldsymbol{u},\boldsymbol{v}_h)}
    \leq &
    \envert[0]{
      \del[0]{
        \varepsilon\overline{\nabla}\del[0]{u-\Pi_hu}
        ,
        \overline{\nabla}v_h
      }_{\mathcal{T}_h}
    }
    +
    \envert[0]{
      \langle
      \varepsilon\alpha h_K^{-1}
      \sbr{\boldsymbol{u}-\boldsymbol{\Pi}_h\boldsymbol{u}},
      \sbr{\boldsymbol{v}_h}
      \rangle_{\mathcal{Q}_h}
    }
    \\
    &
    +
    \envert[0]{
      \langle
      \varepsilon\sbr{\boldsymbol{u}-\boldsymbol{\Pi}_h\boldsymbol{u}}
      ,
      \overline{\nabla}_{\bar{{n}}}v_h
      \rangle_{\mathcal{Q}_h}
    }
    +
    \envert[0]{
      \langle
      \varepsilon \overline{\nabla}_{\bar{{n}}}\del[0]{u-\Pi_hu}
      ,
      \sbr{\boldsymbol{v}_h}
      \rangle_{\mathcal{Q}_h}
    }.
    \end{split}
  \end{equation*}
  By applying the Cauchy--Schwarz inequality, the first two terms on
  the right-hand side can be bounded by
  $\del[0]{1+\alpha}\tnorm{\boldsymbol{u}-\boldsymbol{\Pi}_h\boldsymbol{u}}_{s}\tnorm{\boldsymbol{v}_h}_s$.
  For the last two terms, by the Cauchy--Schwarz inequality and
  \cref{eq:eg_inv_3},
  \begin{equation*}
    \begin{split}
      | \langle
      \varepsilon \sbr[0]{\boldsymbol{u} - \boldsymbol{\Pi}_h\boldsymbol{u}}
      ,
      \overline{\nabla}_{\bar{{n}}}v_h
      \rangle_{\mathcal{Q}_h} |
      +&
      \envert[0]{
        \langle
        \varepsilon \overline{\nabla}_{\bar{{n}}}\del[0]{u-\Pi_hu}
        ,
        \sbr{\boldsymbol{v}_h}
        \rangle_{\mathcal{Q}_h}
      }
      \\
      \leq
      &
      c
      \sum_{\mathcal{K}\in\mathcal{T}_h}
      \varepsilon^{1/2}h_K^{-1/2}
      \norm[0]{
        \sbr{\boldsymbol{u}-\boldsymbol{\Pi}_h\boldsymbol{u}}
      }_{\mathcal{Q}_\mathcal{K}}
      \varepsilon^{1/2}
      \norm[0]{
        \overline{\nabla}v_h
      }_{\mathcal{K}}
      \\
      &
      +
      \sum_{\mathcal{K}\in\mathcal{T}_h}
      \varepsilon^{1/2}h_K^{1/2}
      \norm[0]{
        \overline{\nabla}_{\bar{{n}}}\del[0]{u-\Pi_hu}
      }_{\mathcal{Q}_\mathcal{K}}
      \varepsilon^{1/2}h_K^{-1/2}
      \norm[0]{
        \sbr{\boldsymbol{v}_h}
      }_{\mathcal{Q}_\mathcal{K}}.
    \end{split}
  \end{equation*}
  Therefore, using H\"older's inequality for sums,
  \begin{equation}
    \label{eq:bnd_proj_diff_1}
    \envert[0]{a_{h,d}(\boldsymbol{u}-\boldsymbol{\Pi}_h\boldsymbol{u},\boldsymbol{v}_h)}
    \leq
    \del[1]{
      c
      \tnorm{\boldsymbol{u}-\boldsymbol{\Pi}_h\boldsymbol{u}}_{ss}
      +
      \del[1]
      {
        \sum_{\mathcal{K}\in\mathcal{T}_h}
        \varepsilon h_K
        \norm[0]{
          \overline{\nabla}_{\bar{{n}}}\del[0]{u-\Pi_hu}
        }_{\mathcal{Q}_{\mathcal{K}}}^2
      }^{1/2}
    }
    \tnorm{\boldsymbol{v}_h}_s.
  \end{equation}
  The result follows by combining \cref{eq:ahcboundupiu} and
  \cref{eq:bnd_proj_diff_1}.
\end{proof}

\begin{theorem}[Global error estimate]
  \label{thm:global_err_est}
  Let $\boldsymbol{u}$ and $\boldsymbol{u}_h$ be as in
  \cref{lem:ahuminPihu}. Then
  \begin{equation*}
    \tnorm{\boldsymbol{u}-\boldsymbol{u}_h}_{ss}^2
    \le
    c_T
    \sbr{
      h^{2p_s}(h + \varepsilon + \tilde{\varepsilon} \delta t)
      +
      \delta t^{2p_t+1}(1 + \varepsilon h^{-1})
    }.
  \end{equation*}
  where  $c_T$ depends on
  $\sum_{\mathcal{K}\in\mathcal{T}_h}\norm[0]{u}_{H^{(p_t+1,p_s+1)}(\mathcal{K})}$.
\end{theorem}
\begin{proof}
  We start by noting that Galerkin orthogonality was shown in
  \citet{Kirk:2019}:
  \begin{equation}
    \label{eq:galerkin_ort}
    a_h(\boldsymbol{u}-\boldsymbol{u}_h,\boldsymbol{v}_h)=0
    \quad
    \forall
    \boldsymbol{v}_h:=\del{v_h,\mu_h}\in \boldsymbol{V}_h.
  \end{equation}
  By a triangle inequality, \cref{thm:supginfsup}, and
  \cref{eq:galerkin_ort} we find:
  \begin{equation*}
    \tnorm{\boldsymbol{u}_h-\boldsymbol{u}}_{ss}
    \le
    \tnorm{\boldsymbol{u}-\boldsymbol{\Pi}_h\boldsymbol{u}}_{ss}
    +
    c_T
    \sup_{\boldsymbol{v}_h \in \boldsymbol{V}_h}
    \frac{a_h(\boldsymbol{u}-\boldsymbol{\Pi}_h\boldsymbol{u}, \boldsymbol{v}_h)}{
      \tnorm{\boldsymbol{v}_h}_{s}}.
  \end{equation*}
  Using \cref{lem:ahuminPihu},
  \begin{multline}
    \label{eq:isupu_final}
    \tnorm{\boldsymbol{u}_h-\boldsymbol{u}}_{ss}
    \le
    c_T
    \tnorm{\boldsymbol{u}-\boldsymbol{\Pi}_h\boldsymbol{u}}_{ss}
    +
    c_T
    \norm[0]{
      \envert[0]{
        \beta_s
        -
        \tfrac{1}{2}\beta\cdot{n}
      }^{1/2}
      \del[0]{u-\Pi_hu}
    }_{\partial\mathcal{T}_h}
    \\
    +
    c_T
    \del[1]
    {
      \sum_{\mathcal{K}\in\mathcal{T}_h}
      \varepsilon h_K
      \norm[0]{
        \overline{\nabla}_{\bar{{n}}}\del[0]{u-\Pi_hu}
      }_{\mathcal{Q}_{\mathcal{K}}}^2
    }^{1/2}.
  \end{multline}
  The second term on the right-hand side of \cref{eq:isupu_final} is
  bounded using \cref{eq:proj_est_conv_5}:
  \begin{equation}
    \label{eq:conv_anal_7}
    \sum_{\mathcal{K}\in\mathcal{T}_h}
    \norm[0]{
      \envert[0]{
        \beta_s
        -
        \tfrac{1}{2}\beta\cdot{n}
      }^{1/2}
      \del[0]{u-\Pi_hu}
    }_{\partial\mathcal{K}_h}^2
    \le
    c
    \del[0]{h^{2p_s+1}+{\delta t}^{2p_t+1}}
    \sum_{\mathcal{K}\in\mathcal{T}_h}
    \norm[0]{u}_{H^{(p_t+1,p_s+1)}(\mathcal{K})}^2.
  \end{equation}
  The last term on the right-hand side of \cref{eq:isupu_final} is
  bounded using \cref{eq:proj_est_conv_4}:
  \begin{equation}
    \label{eq:conv_anal_8}
    \sum_{\mathcal{K}\in\mathcal{T}_h}\varepsilon h_K
    \norm[0]{
      \overline{\nabla}_{\bar{{n}}}
      \del{u-\Pi_hu}
    }_{\mathcal{Q}_{\mathcal{K}}}^2
    \le
    c
    \varepsilon
    \del[0]{h^{2p_s}+{\delta t}^{2p_t+2}}
    \sum_{\mathcal{K}\in\mathcal{T}_h}
    \norm[0]{u}_{H^{(p_t+1,p_s+1)}(\mathcal{K})}^2.
  \end{equation}
  The result follows after combining
  \cref{eq:isupu_final,eq:conv_anal_7,eq:conv_anal_8,lem:uminPihusbound}.
\end{proof}

\begin{remark}
  \label{rem:leadingordertermsestimate}
  The error estimate of \cref{thm:global_err_est} shows that if
  $\varepsilon < \delta t = h$ then
  $\tnorm{\boldsymbol{u}-\boldsymbol{u}_h}_{ss} =
  \mathcal{O}(h^{p_s+1/2} + \delta t^{p_t+1/2})$, while if
  $\delta t = h < \varepsilon$ then
  $\tnorm{\boldsymbol{u}-\boldsymbol{u}_h}_{ss} = \mathcal{O}(h^{p_s}
  + \delta t^{p_t})$.
\end{remark}

\section{Numerical example}
\label{s:numerics}

We consider the solution of a two-dimensional rotating Gaussian pulse
on a deforming domain \citep{Rhebergen:2013} to demonstrate the
convergence properties of the space-time HDG method predicted by
\cref{thm:global_err_est}. In \cref{eq:st_adr} we set
$\beta = \del[0]{1,-4x_2,4x_1}^T$ and $f=0$. Defining
$\widetilde{x}_1 := x_1\cos(4t)+x_2\sin(4t)$ and
$\widetilde{x}_2 := -x_1\sin(4t)+x_2\cos(4t)$, the exact solution to
this problem is given by
\begin{equation*}
  u(t,x_1,x_2)
  =
  \tfrac{\sigma^2}{\sigma^2+2\varepsilon t}
  \exp\del[1]{
    -\tfrac{\del[0]{\widetilde{x}_1-x_{1c}}^2+\del[0]{\widetilde{x}_2-x_{2c}}^2}{2\sigma^2+4\varepsilon t}
  },
\end{equation*}
with initial and boundary conditions set appropriately. We choose
$\sigma = 0.1$ and $(x_{1c},x_{2c})=(-0.2,0.1)$. The deforming domain
$\Omega(t)$ is obtained by transforming a uniform mesh, with
coordinates $(x_1^u,x_2^u) \in (-0.5,0.5)^2$, to
\begin{equation}
  \label{eq:ex_deform_1}
  x_i
  =
  x_i^u+A(\tfrac{1}{2}-x_i^u)
  \sin(2\pi(\tfrac{1}{2}-x_i^*+t)),
  \quad
  i=1,2,
\end{equation}
where $\del[0]{x_1^*,x_2^*}=\del[0]{x_2,x_1}$ and $A=0.1$. We consider
this problem for $t \in [0,1]$.

The space-time HDG method \cref{eq:st_hdg_adr_compact} is implemented
using the finite element library deal.II \citep{dealII94} on
unstructured hexahedral space-time meshes with p4est
\citep{Bangerth:2011} to obtain distributed mesh information. We use
PETSc \citep{petsc-web-page,petsc-user-ref,petsc-efficient} to solve
the linear systems arising at each time step (GMRES preconditioned by
classical algebraic multigrid from BoomerAMG \citep{Henson:2002} with
an absolute solver tolerance of $10^{-12}$). To create our coarsest
mesh, we start with an initial mesh with elements of size
$h \approx \delta t = 10^{-1}$. Space-time elements in a ring
prescribed by $|((x_1^c)^2+(x_2^c)^2)^{1/2}-0.2|<0.1$, where
$(x_1^c,x_2^c)$ is the spatial coordinate of the center of a
space-time element, are then uniformly refined once and are of size
$h\approx \delta t = 0.05$. See \cref{fig:Gaussian-moving-mesh} for
plots of the solution and spatial mesh at different time levels. The
reason to consider two sets of elements is to verify that the analysis
of previous sections hold on 1-irregular space-time meshes. Finer
meshes are obtained by uniformly refining our coarsest mesh. In our
implementation we furthermore choose the penalty parameter
$\alpha = 8 p_s^2$ \citep[see, for example][]{Riviere:book}. We show
the rates of convergence for different polynomial degrees when the
error is measured in $\tnorm{\cdot}_{ss}$ in
\cref{tb:guasspulse_conv_m2,tb:guasspulse_conv_m8} for
$\varepsilon = 10^{-2}$ and $\varepsilon = 10^{-8}$, respectively.

\begin{figure}[tbp]
  \centering
  \begin{subfigure}{0.3\textwidth}
    \centering
    \includegraphics[width=\textwidth]{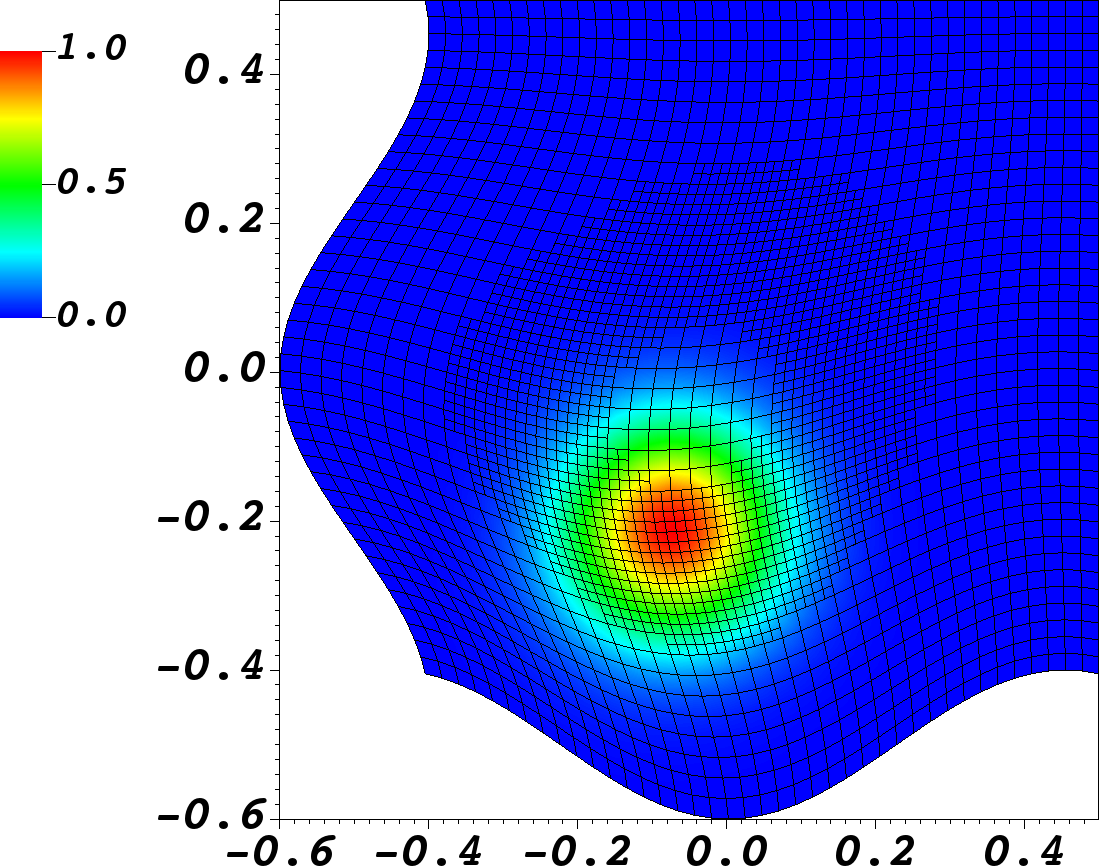}
  \end{subfigure}
  \;\;
  \begin{subfigure}{0.3\textwidth}
    \centering
    \includegraphics[width=\textwidth]{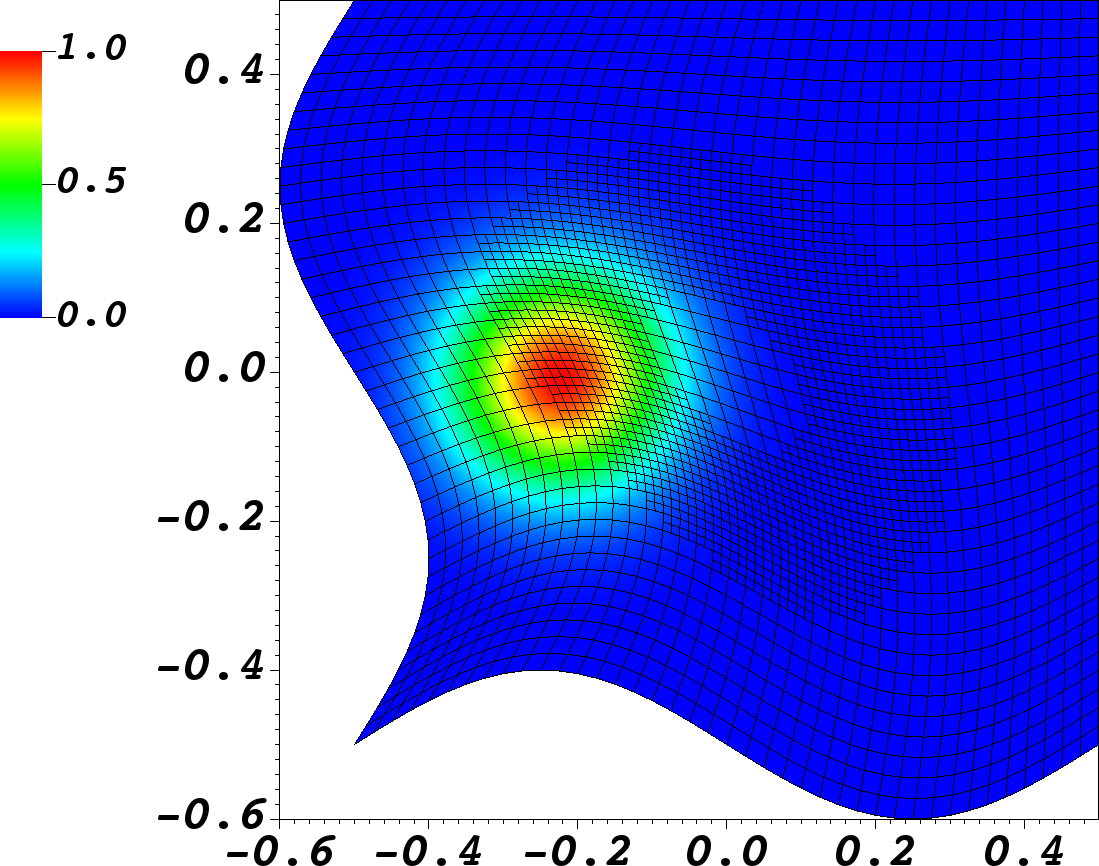}
  \end{subfigure}
  \;\;
  \begin{subfigure}{0.3\textwidth}
    \centering
    \includegraphics[width=\textwidth]{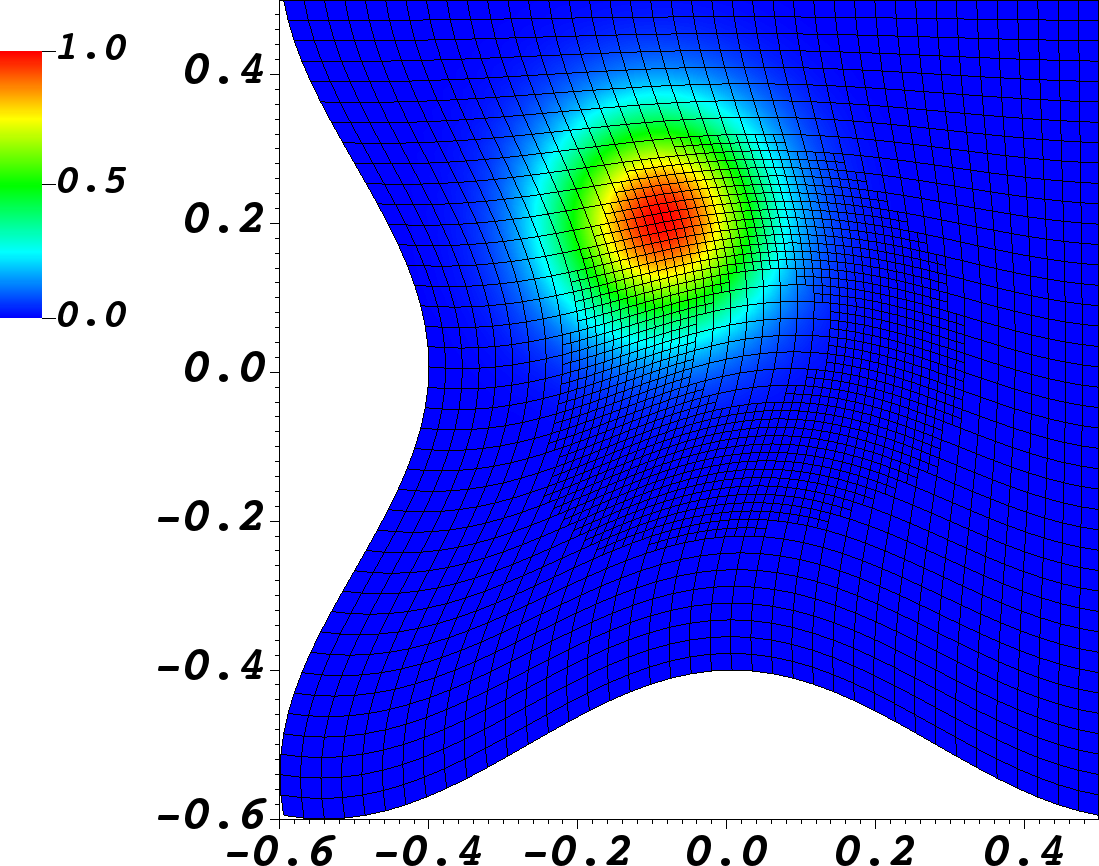}
  \end{subfigure}
  \caption{The spatial mesh and the ring of elements with an extra
    level of refinement deform over time. The solution shown is for
    $\varepsilon=10^{-8}$. Plots correspond to time levels $t = 0.2,0.5,0.8$
    from left to right.}
  \label{fig:Gaussian-moving-mesh}
\end{figure}

In the third row of \cref{tb:guasspulse_conv_m2} we have that
$h\approx\delta t =1.25\times 10^{-2}$ inside the refined ring while
elsewhere $h\approx\delta t =2.5\times 10^{-2}$. Therefore, for the
first three rows in \cref{tb:guasspulse_conv_m2},
$h\approx\delta t \ge\varepsilon=10^{-2}$ and we observe a rate of convergence
of approximately $p+1/2$. In the following three rows we observe a
drop in the rate of convergence to approximately $p$. This happens in
two stages since there are two sets of elements in our mesh, see
\cref{fig:Gaussian-moving-mesh}.  In the first stage (the fourth row
of \cref{tb:guasspulse_conv_m2}), elements in the refined ring are
such that $h\approx\delta t=6.25\times 10^{-3}<\varepsilon$, but elsewhere
$h\approx\delta t=1.25\times 10^{-2}>\varepsilon$. In the next stage (fifth
row of \cref{tb:guasspulse_conv_m2}), all elements satisfy
$h\approx\delta t<\varepsilon$, resulting in a rate of convergence of $p$
after the fifth row.  In \cref{tb:guasspulse_conv_m8} we observe that
$h \approx \delta t>\varepsilon=10^{-8}$ for all cycles and the error
converges at a rate of approximately $p+1/2$. These observations from
\cref{tb:guasspulse_conv_m2,tb:guasspulse_conv_m8} are in agreement
with \cref{rem:leadingordertermsestimate}.

\begin{table}[tbp]
  \caption{The solution errors measured in
    $|\mkern-1.5mu|\mkern-1.5mu|\cdot|\mkern-1.5mu|\mkern-1.5mu|_{ss}$
    and corresponding rates of convergence when using polynomial
    approximation $p=1,2,3$ for the case $\varepsilon=10^{-2}$.}
  \label{tb:guasspulse_conv_m2}
  \centering
  \begin{tabular}{cc|cc|cc|cc}
    \hline
     Cells per slab & Number of slabs & $p=1$ & Rates & $p=2$
     & Rates & $p=3$ & Rates\\
     \hline
     296 & 10 & 4.7e-2 & - & 7.8e-3 & - &1.3e-3&-\\
     1100& 20 & 1.8e-2 & 1.4 & 1.6e-3 & 2.4 &1.2e-4&3.6\\
     4372& 40 & 7.7e-3 & 1.3 & 3.2e-4 & 2.3 &1.7e-5&3.4\\
     17572& 80 & 3.7e-3 & 1.1 & 7.3e-5 & 2.1 &1.4e-6&3.2\\
     70540&160 & 2.0e-3 & 0.9 & 2.3e-5 & 1.7 & 2.4e-7&2.4\\
     282580&320 & 9.0e-4 & 1.1 & 4.9e-6 & 2.2 & 2.5e-8& 3.3\\
     \hline
  \end{tabular}
\end{table}

\begin{table}[tbp]
  \caption{The solution errors measured in
    $|\mkern-1.5mu|\mkern-1.5mu|\cdot|\mkern-1.5mu|\mkern-1.5mu|_{ss}$
    and corresponding rates of convergence when using polynomial
    approximation $p=1,2,3$ for the case $\varepsilon=10^{-8}$.}
  \label{tb:guasspulse_conv_m8}
  \centering
  \begin{tabular}{cc|cc|cc|cc}
    \hline
    Cells per slab & Number of slabs & $p=1$ & Rates & $p=2$
    & Rates & $p=3$ & Rates\\
    \hline
    289 & 10 & 1.1e-1 & - & 1.6e-2 & - &2.8e-3&-\\
    1086& 20 & 3.9e-2 & 1.5 & 2.8e-3 & 2.7 &2.3e-4&3.8\\
    4372& 40 & 1.1e-2 & 1.8 & 4.4e-4 & 2.7 &1.8e-5&3.7\\
    17572& 80 & 3.4e-3 & 1.7 & 7.1e-5 & 2.6 &1.4e-6&3.7\\
    70540&160 & 1.1e-3 & 1.6 & 1.2e-5 & 2.6 &1.1e-7&3.6\\
    282580&320 & 4.0e-4 & 1.5 & 2.1e-6& 2.5 &9.7e-9&3.6 \\
    \hline
  \end{tabular}
\end{table}

\section{Conclusions}
\label{s:conclusions}

In this paper we analyzed the space-time HDG method for the
time-dependent advection-diffusion equation on deforming domains. We
proved stability of the discretization in the advection-dominated
regime using a time-dependent weighted test function, and derived a
priori error estimates that show a drop of $1/2$ in the rate of
convergence in a mesh dependent norm when transitioning from a mesh
size larger than the diffusion parameter to a mesh size smaller than
the diffusion parameter. A numerical example supports our theoretical
analysis.

\appendix
\section{Proof of \cref{eq:projforstab}}
\label{s:projforstab}

The following results from \citet[Lemma B.7]{Sudirham:thesis} will be
useful in what follows: For
$\widetilde{u} \in H^{(k_t,k_s)}(\widetilde{\mathcal{K}})$,
$0 \le \alpha_t \le k_t$, $\alpha_s=(\alpha_1,\hdots,\alpha_d)$ with
$|\alpha_s| \le k_s$, and
$\widetilde{F}_{\mathcal{Q}} \subset
\mathcal{Q}_{\widetilde{\mathcal{K}}}$:
\begin{subequations}
  \label{eq:lb7sudir}
  \begin{align}
    \label{eq:lb7sudir-1}
    \norm[0]{\partial_{\widehat{t}}^{\alpha_t}\partial_{\widehat{x}}^{\alpha_s}\widehat{u}}_{\widehat{\mathcal{K}}}^2
    &=
    \del[1]{\tfrac{1}{2}}^{2(\alpha_t+|\alpha_s|)-d-1}
    \delta t_{\mathcal{K}}^{2\alpha_t-1}
    h_K^{2|\alpha_s|-d}
    \norm[0]{\partial_{\widetilde{t}}^{\alpha_t}\partial_{\widetilde{x}}^{\alpha_s}\widetilde{u}}_{\widetilde{\mathcal{K}}}^2,
    \\
    \label{eq:lb7sudir-2}
    \norm[0]{\partial_{\widehat{t}}^{\alpha_t}\partial_{\widehat{x}}^{\alpha_s}\widehat{u}}_{\widehat{F}_{\mathcal{Q}}}^2
    &= \del[1]{\tfrac{1}{2}}^{2(\alpha_t+|\alpha_s|)-d}
      \delta t_{\mathcal{K}}^{2\alpha_t-1}h_K^{2|\alpha_s|-d+1}
      \norm[0]{\partial_{\widetilde{t}}^{\alpha_t}\partial_{\widetilde{x}}^{\alpha_s}\widetilde{u}}_{\widetilde{F}_{\mathcal{Q}}}^2.
  \end{align}
\end{subequations}

To begin, we note that the projection operator on
$\widehat{\mathcal{K}}$ and $\widetilde{\mathcal{K}}$ are related to
the projection operator on $\mathcal{K}$ as follows
\citep[see][Definition 3.12]{Georgoulis:thesis}:
\begin{equation*}
  \widetilde{\Pi}_h
  \widetilde{u}
  =
  \del[0]{
    {\Pi_h}
    \del[0]{
      \widetilde{u}\circ {\phi}_{\mathcal{K}}^{-1}
    }
  }\circ{\phi}_{\mathcal{K}}
  \quad
  \forall\widetilde{u}\in L^2(\widetilde{\mathcal{K}}),
  \qquad
  \widehat{\Pi}
  \widehat{u}
  =
  \del[0]{
    \widetilde{\Pi}_h
    \del[0]{\widehat{u}\circ G_{\mathcal{K}}^{-1} }
  }\circ G_{\mathcal{K}}
  \quad
  \forall\widehat{u}\in L^2({\widehat{\mathcal{K}}}).
\end{equation*}
Similarly, on any $F\in\partial{\mathcal{K}}$, we have the following
relations:
\begin{equation*}
  \widetilde{\Pi}_h^{\mathcal{F}}
  \widetilde{u}
  =
  \del[0]{
    {\Pi_h^{\mathcal{F}}}
    \del[0]{
      \widetilde{u}\circ {\phi}_{F}^{-1}
    }
  }\circ{\phi}_{F}
  \
  \forall\widetilde{u}\in L^2(\widetilde{F}),
  \quad
  \widehat{\Pi}^{\mathcal{F}}
  \widehat{u}
  =
  \del[0]{
    \widetilde{\Pi}_h^{\mathcal{F}}
    \del[0]{\widehat{u}\circ G_{F}^{-1} }
  }\circ G_{F}
  \
  \forall\widehat{u}\in L^2({\widehat{F}}),
\end{equation*}
where $G_F$ and $\phi_F$ are the restrictions of $G_{\mathcal{K}}$ and
$\phi_{\mathcal{K}}$ on $F$, respectively.

To show \cref{eq:proj_spatial_grad} we first note that the following
projection estimate holds on $\widehat{\mathcal{K}}$ for any
$1\leq i\leq d$ \citep[see][lemma 3.7, eq. (3.12)]{Georgoulis:thesis}:
\begin{equation}
  \label{eq:projestpartialxhat}
  \norm[0]{\partial_{\widehat{x}_i}\del[0]{\widehat{u}-\widehat{\Pi}\widehat{u}}}_{\widehat{\mathcal{K}}}
  \leq
  c
  \norm[0]{\partial_{\widehat{x}_i}\widehat{u}}_{\widehat{\mathcal{K}}}.
\end{equation}
By the chain rule
$
{\partial_{\widetilde{x}_i}}\del[0]{\widetilde{u}-\widetilde{\Pi}_h\widetilde{u}}
= 2h_K^{-1} {\partial_{\widehat{x}_i}}\del[0]{
  \del[0]{\widetilde{u}-\widetilde{\Pi}_h\widetilde{u}} \circ
  {G}_{\mathcal{K}} } \circ {G}_{\mathcal{K}}^{-1} $, using that
$|\text{det }G_{\mathcal{K}}| = \delta t_{\mathcal{K}}h_K^d2^{-d-1}$,
and using \cref{eq:projestpartialxhat,eq:lb7sudir-1}, we find
\begin{equation}
  \label{eq:tildeprojectionestimatexi}
  \norm[0]{
    {\partial_{\widetilde{x}_i}}
    \del[0]{\widetilde{u}-\widetilde{\Pi}_h\widetilde{u}}
  }_{\widetilde{\mathcal{K}}}^2
  \leq
  c
  {\delta t_{\mathcal{K}}}
  h_K^{d-2}
  \norm[0]{
    {\partial_{\widehat{x}_i}}{\widehat{u}}
  }^2
  \leq
  c
  \norm[0]{
    {\partial_{\widetilde{x}_i}\widetilde{u}}
  }_{\widehat{\mathcal{K}}}^2.
\end{equation}
To obtain the result on the physical element, consider first that by
the chain rule,
\begin{equation*}
  {\partial_{{x}_i}}({u}-\Pi_{h}{u})
  =
  \sum_{1\leq j\leq d} \del[0]{
    {\partial_{\widetilde{x}_j}}\del[0]{ \del[0]{{u}-\Pi_{h}{u}} \circ
      \phi_\mathcal{K} } \circ \phi_\mathcal{K}^{-1} }
  \del[0]{ (-1)^{i+j}(\det
    J_{\phi_{\mathcal{K}}})^{-1}\det
    J_{\phi_{\mathcal{K}}\backslash ij} },
\end{equation*}
where we used that $\widetilde{t}$ only depends on $t$ in
$\phi_\mathcal{K}^{-1}$ and that
$\tfrac{\partial\widetilde{x}_j}{\partial{x}_i} = (-1)^{i+j}(\det
J_{\phi_{\mathcal{K}}})^{-1}\det J_{\phi_{\mathcal{K}}\backslash
  ij}$. By assumptions \cref{eq:diffeom_regular} and
\cref{eq:diffeom_regular_inv}, and using
\cref{eq:tildeprojectionestimatexi}, we therefore find that:
\begin{equation*}
  \norm[0]{{\partial_{{x}_i}}\del[0]{{u}-\Pi_{h}{u}}}_{\mathcal{K}}^2
  \le
  c
  \sum_{1\leq j\leq d}
  \norm[0]{
    {\partial_{\widetilde{x}_i}}
    \del[0]{\widetilde{u}-\widetilde{\Pi}_h\widetilde{u}}
  }_{\widetilde{\mathcal{K}}}^2
  \leq
  c
  \sum_{1\leq j\leq d}
  \norm[0]{
    \partial_{\widetilde{x}_j}\widetilde{u}
  }_{\widetilde{\mathcal{K}}}^2
  \le
  c
  \norm[0]{
    \overline{\nabla}u
  }_{\mathcal{K}}^2,
\end{equation*}
where the last step reverses the scaling arguments proving
\cref{eq:proj_spatial_grad}. The proof for
\cref{eq:proj_time_derivative} is similar and therefore omitted.

For \cref{eq:proj_diff_elem_facet_Q}, we consider a $d$-dimensional
hypersurface $F_{\mathcal{Q}}\in\mathcal{Q}_{\mathcal{K}}$. We first
map $F_{\mathcal{Q}}$ to the reference domain. For this, let
$\phi_{F_{\mathcal{Q}}}\circ
G_{F_{\mathcal{Q}}}(\widehat{F}_{\mathcal{Q}}) = F_{\mathcal{Q}}$,
i.e., the transformation of a face from the reference domain to the
physical domain. We then observe that one of the spatial coordinates,
which is denoted by $\widehat{x}_j$ without loss of generality, of
$\widehat{F}_\mathcal{Q}$ is fixed. We further consider a
decomposition of
$\widehat{\Pi}=\widehat{\pi}_{\widehat{t}} \Pi_{1\leq i\leq d}
\widehat{\pi}_{\widehat{x}_i}$ where $\widehat{\pi}_{\widehat{t}}$ and
$\widehat{\pi}_{\widehat{x}_i}$ are the one-dimensional
$L^2$-projection operators applied in the time direction and in the
spatial direction $\widehat{x}_i$, respectively. Similarly,
$\del[0]{\widehat{\Pi}^{\mathcal{F}}}|_{\widehat{F}_{\mathcal{Q}}}=\widehat{\pi}_{\widehat{t}}
\Pi_{1\leq i\leq d,i\neq j} \widehat{\pi}_{\widehat{x}_i}$. By
\citet[Definitions 3.1, 3.6]{Georgoulis:thesis}, we have:
\begin{equation}
  \label{eq:pihatpiFhat}
    \norm[0]{
      \widehat{\Pi}
      \widehat{u}
      -
      \widehat{\Pi}^{\mathcal{F}}
      \widehat{u}
    }_{\widehat{F}_\mathcal{Q}}
    =
    \norm[0]{
      \widehat{\pi}_{\widehat{t}}
      \Pi_{1\leq i\leq d,i\neq j}
      \widehat{\pi}_{\widehat{x}_i}
      \del[0]{
        \widehat{u}
        -
        \widehat{\pi}_{\widehat{x}_j}
        \widehat{u}
      }
    }_{\widehat{F}_\mathcal{Q}}
    \leq
    c
    \norm[0]{
      \widehat{u}
      -
      \widehat{\pi}_{\widehat{x}_j}
      \widehat{u}
    }_{\widehat{F}_\mathcal{Q}}
    \leq
    c
    \norm[0]{\partial_{\widehat{x}_j}\widehat{u}}_{\widehat{\mathcal{K}}},
\end{equation}
where the equality is by commutativity of
$\widehat{\pi}_{\widehat{x}_i}$ and $\widehat{\pi}_{\widehat{x}_j}$
($i\ne j$) and the last two inequalities are due to the boundedness of
any composition of projections $\widehat{\pi}$ and \citet[Lemma
3.3]{Georgoulis:thesis}. Next, using \cref{eq:lb7sudir-2},
\cref{eq:pihatpiFhat}, and \cref{eq:lb7sudir-1},
\begin{equation*}
  \norm[0]{
    \widetilde{\Pi}_h\widetilde{u}
    -
    \widetilde{\Pi}_h^{\mathcal{F}}\widetilde{u}
  }_{\widetilde{F}_{\mathcal{Q}_i}}^2
  \le
  c
  h_K
  \norm[0]{\partial_{\widetilde{x}_j}\widetilde{u}}_{\widetilde{\mathcal{K}}}^2.
\end{equation*}
Therefore, also using \cref{eq:k_surface} and
\cref{eq:diffeom_regular_d},
\begin{equation*}
    \norm[0]{
      \Pi_{h}{u}
      -
      \Pi_{h}^{\mathcal{F}}{u}
    }^2_{F_{\mathcal{Q}}}
    \leq
    c
    \norm[0]{
      \widetilde{\Pi}_h\widetilde{u}
      -
      \widetilde{\Pi}_h^{\mathcal{F}}\widetilde{u}
    }_{\widetilde{F}_\mathcal{Q}}^2
    \leq
    c
    h_K
    \norm[0]{\partial_{\widetilde{x}_j}\widetilde{u}}_{\widetilde{\mathcal{K}}}^2
    \le
    c	 h_K
    \norm[0]{\overline{\nabla}u}_{{\mathcal{K}}}^2,
\end{equation*}
where we reverse the scaling arguments in the final inequality,
proving \cref{eq:proj_diff_elem_facet_Q}.

\section{Proof of \cref{lem:proj_with_wt}}
\label{s:proof_prof_with_wt}

To prove \cref{lem:proj_with_wt} we use the following result from
\citet[Lemma B.6]{Sudirham:thesis}. For any
$\mathcal{K}\in\mathcal{T}_h$ and $v\in H^1(\mathcal{K})$, we consider
its affine domain
$\widetilde{\mathcal{K}}:=\phi_\mathcal{K}^{-1}\circ\mathcal{K}$ (with
$\mathcal{Q}_{\widetilde{\mathcal{K}}}:=\phi_\mathcal{K}^{-1}\circ\mathcal{Q}_\mathcal{K}$
and
$\mathcal{R}_{\widetilde{\mathcal{K}}}:=\phi_\mathcal{K}^{-1}\circ\mathcal{R}_\mathcal{K}$)
and
$\widetilde{v}:=v\circ\phi_\mathcal{K}\in
H^1(\widetilde{\mathcal{K}})$. The following trace inequalities hold:
\begin{subequations}
  \label{eq:refined_aniso_trace_aff}
  \begin{align}
    \label{eq:refined_aniso_trace_aff_1}
    \norm[0]{
    \widetilde{v}
    }_{\mathcal{Q}_{\widetilde{\mathcal{K}}}}^2
    &\leq
      c
    \del[0]{
    h_K^{-1}
    \norm[0]{
    \widetilde{v}
    }_{\widetilde{\mathcal{K}}}^2
    +
    \norm[0]{
    \widetilde{v}
    }_{\widetilde{\mathcal{K}}}
    \norm[0]{
    \widetilde{
    \overline{\nabla}
    }
    \widetilde{v}
    }_{\widetilde{\mathcal{K}}}
    }.
    \\
    \norm[0]{
    \widetilde{v}
    }_{\mathcal{R}_{\widetilde{\mathcal{K}}}}^2
    &
      \leq
      c
      \del[0]{
      \delta t_{\mathcal{K}}^{-1}
      \norm[0]{
      \widetilde{v}
      }_{\widetilde{\mathcal{K}}}^2
      +
      \norm[0]{
      \widetilde{v}
      }_{\widetilde{\mathcal{K}}}
      \norm[0]{
      \partial_{\widetilde{t}}
      \widetilde{v}
      }_{\widetilde{\mathcal{K}}}
      }.
      \label{eq:refined_aniso_trace_aff_2}
  \end{align}
\end{subequations}
We further observe that when $\varphi$ is a function of the time
variable only on $\mathcal{K}$, so are $\widetilde{\varphi}$ and
$\widehat{\varphi}$ on $\widetilde{\mathcal{K}}$ and
$\widehat{\mathcal{K}}$ respectively. Therefore, for
$w_h \in V_h^{(p_t,p_s)}$, and for $1 \le i,j \le d$, $j \ne i$,
\begin{equation}
  \label{eq:zeroequalities}
  \partial_{\widetilde{x}_i}^{p_s+1} \del[0]{\widetilde{\varphi
      w_h}}
  =
  \partial_{\widetilde{x}_i}^{p_s+1} \partial_{\widetilde{x}_j}
  \del[0]{\widetilde{\varphi w_h}}
  =
  \partial_{\widetilde{x}_i}^{p_s+1} \partial_{\widetilde{t}}
  \del[0]{\widetilde{\varphi w_h}}
  =
  0.
\end{equation}
The equivalent derivatives above are also zero on the reference domain
$\widehat{\mathcal{K}}$. Furthermore,
\begin{subequations}
  \label{eq:lebnizinverse}
  \begin{align}
    \norm[0]{
    \partial_{\widetilde{t}}^{p_t+1}
    \del[0]{\widetilde{\varphi w_h}}
    }_{\widetilde{\mathcal{K}}}
    &\leq
      c
      \delta t_{\mathcal{K}}^{-p_t}
      \norm[0]{\widetilde{w}_h}_{\widetilde{\mathcal{K}}},
      \label{eq:lebnizinverse_1}
    \\
    \norm[0]{
    \partial_{\widetilde{t}}^{p_t+1}
    \partial_{\widetilde{x}_i}
    \del[0]{\widetilde{\varphi w_h}}
    }_{\widetilde{\mathcal{K}}}
    &\leq
      c
      \delta t_{\mathcal{K}}^{-p_t}
      h_K^{-1}
      \norm[0]{\widetilde{w}_h}_{\widetilde{\mathcal{K}}}
      \label{eq:lebnizinverse_2}.
  \end{align}
\end{subequations}
\Cref{eq:lebnizinverse_1} can be shown using the general Leibniz rule,
that $\partial_{\widetilde{t}}^{p_t+1}\widetilde{w_h}=0$, that
$\norm[0]{\partial_{\widetilde{t}}^{j_t}\widetilde{\varphi}}_{\widetilde{\mathcal{K}}}
\le e$ for all $1\le j_t\le p_t+1$, that \cref{eq:eg_inv_1} reduces to
$\norm[0]{\partial_{\widetilde{t}}\widetilde{v}_h}_{\widetilde{\mathcal{K}}}
\leq c {\delta
  t_{\mathcal{K}}^{-1}}\norm[0]{\widetilde{v}_h}_{\widetilde{\mathcal{K}}}
$ on the axiparallel element $\widetilde{\mathcal{K}}$
\citep[see][Corollary 3.54]{Georgoulis:thesis} and using that
$\delta t_{\mathcal{K}}<1$. Similar arguments can be used to show
\cref{eq:lebnizinverse_2}.

To prove \cref{eq:proj_with_wt_1} we follow the proof of \citet[Lemma
4.2]{Ayuso:2009}, and apply the projection estimates in \citet[Lemma
3.4]{Houston:2002} when considered on the affine domain
$\widetilde{\mathcal{K}}$, \cref{eq:zeroequalities} and
\cref{eq:lebnizinverse_1},
\begin{equation}
  \label{eq:some_eq_102}
  \norm[0]{
    \widetilde{\varphi w_h}
    -
    \widetilde{\Pi}_h
    \del{\widetilde{\varphi w_h}}
  }_{\widetilde{\mathcal{K}}}
  \le
  c
  {{\delta t_{\mathcal{K}}}}^{p_t+1}
  \norm[0]{
    \partial_{\widetilde{t}}^{p_t+1}
    \del{\widetilde{\varphi w_h}}
  }_{\widetilde{\mathcal{K}}}
  \leq
  c
  \delta t_{\mathcal{K}}
  \norm[0]{
    \widetilde{w_h}
  }_{\widetilde{\mathcal{K}}}.
\end{equation}
A scaling argument applied to \cref{eq:some_eq_102} from
$\widetilde{\mathcal{K}}$ to $\mathcal{K}$ yields
\cref{eq:proj_with_wt_1}.

We next prove \cref{eq:proj_with_wt_2}. First note,
\begin{equation}
  \label{eq:some_eq_40}
  \norm[1]{
    \widetilde{\overline{\nabla}}
    \del[1]{
      \widetilde{\varphi w_h}
      -
      \widetilde{\Pi}_h
      \del[0]{\widetilde{\varphi w_h}}
    }
  }_{\widetilde{\mathcal{K}}}^2
  \le
  \sum_{1\leq i\leq d}
  c
  {h_K^{d-2}{{\delta t_{\mathcal{K}}}}}
  \norm[0]{
    \partial_{\widehat{x}_{i}}
    \del[0]{
      \widehat{\varphi w_h}
      -
      \widehat{\Pi}_h
      \del[0]{\widehat{\varphi w_h}}
    }
  }_{\widehat{\mathcal{K}}}^2.
\end{equation}
Following similar steps as in the proof of \citet[Lemma
7.5]{Georgoulis:2006}, the right-hand side of \cref{eq:some_eq_40} can
be bound further using the triangle inequality, commutativity of
$\partial_{\widehat{x}_i}$ with
$ \widehat{\pi}_{t} \Pi_{1\leq j\leq d,j\neq i} \widehat{\pi}_{x_j}$,
boundedness of
$ \widehat{\pi}_{t} \Pi_{1\leq j\leq d,j\neq i} \widehat{\pi}_{x_j}$,
the projection estimates in \citet[Lemma 3.4]{Houston:2002} and
\citet[Lemma 7.3]{Georgoulis:2006}, \cref{eq:zeroequalities},
\cref{eq:lb7sudir-1}, and \cref{eq:lebnizinverse_2},
\begin{equation}
  \label{eq:some_eq_103}
  \begin{split}
    \norm[0]{
      \partial_{\widehat{x}_{i}}
      \del[1]{
        \widehat{\varphi w_h}
        -
        \widehat{\Pi}_h
        \del[0]{\widehat{\varphi w_h}}
      }
    }_{\widehat{\mathcal{K}}}
    &\leq
    \norm[0]{
      \partial_{\widehat{x}_{i}}
      \widehat{\varphi w_h}
      -
      \widehat{\Pi}_h
      \del[0]{
        \partial_{\widehat{x}_{i}}
        \widehat{\varphi w_h}
      }
    }_{\widehat{\mathcal{K}}}
    +
    c
    \norm[0]{
      \widehat{\pi}_{x_i}
      \del[0]{
        \partial_{\widehat{x}_{i}}
        \widehat{\varphi w_h}
      }
      -
      \partial_{\widehat{x}_{i}}
      \del[0]{
        \widehat{\pi}_{x_i}
        \del[0]{\widehat{\varphi w_h}}
      }
    }_{\widehat{\mathcal{K}}}
    \\
    &\leq
    c
    \norm[0]{
      \partial_{\widehat{t}}^{p_t+1}
      \partial_{\widehat{x}_{i}}
      \widehat{\varphi w_h}
    }_{\widehat{\mathcal{K}}}
    \leq
    c
    h_K^{-d/2}\delta t_{\mathcal{K}}^{1/2}
    \norm[0]{ \widetilde{w_h} }_{\widetilde{\mathcal{K}}}.
  \end{split}
\end{equation}
Combining the right-hand side of \cref{eq:some_eq_103} with
\cref{eq:some_eq_40}, we find:
\begin{equation}
  \label{eq:some_eq_108}
  \norm[0]{
    \widetilde{\overline{\nabla}}
    \del[0]{
      \widetilde{\varphi w_h}
      -
      \widetilde{\Pi}_h
      \del[0]{\widetilde{\varphi w_h}}
    }
  }_{\widetilde{\mathcal{K}}}
  \leq
  c
  {\delta t_{\mathcal{K}}h_K^{-1}}
  \norm[0]{ \widetilde{w_h} }_{\widetilde{\mathcal{K}}}.
\end{equation}
A scaling argument applied to \cref{eq:some_eq_108} from
$\widetilde{\mathcal{K}}$ to $\mathcal{K}$ yields
\cref{eq:proj_with_wt_2}. With similar steps it can be shown that
\begin{equation}
  \label{eq:some_eq_109}
  \norm[0]{
    \partial_{\widetilde{t}}
    \del[0]{
      \widetilde{\varphi w_h}
      -
      \widetilde{\Pi}_h
      \del[0]{\widetilde{\varphi w_h}}
    }
  }_{\widetilde{\mathcal{K}}}
  \leq
  c
  \norm[0]{\widetilde{w}_h}_{\widetilde{\mathcal{K}}}.
\end{equation}

We next prove \cref{eq:proj_with_wt_5}. We start with a scaling
argument to transform the integral on a $\mathcal{Q}$-face
$\widetilde{F}_{\mathcal{Q},m}$ from the affine domain to the
reference domain. Note that, without loss of generality, subscript $m$
denotes the index of the spatial coordinate for which
$\widehat{x}_m\equiv 1$. Using \cref{eq:lb7sudir-2} we find
\begin{equation}
  \label{eq:some_eq_43}
  \norm[0]{
    \widetilde{\overline{\nabla}}
    \del[0]{
      \widetilde{\varphi w_h}
      -
      \widetilde{\Pi}_h
      \del[0]{\widetilde{\varphi w_h}}
    }
  }_{\widetilde{F}_{\mathcal{Q},m}}^2
  \le
  \sum_{1\leq i\leq d}
  c
  \delta t_{\mathcal{K}}h_K^{d-3}
  \norm[0]{
    \partial_{\widehat{x}_{i}}
    \del[0]{
      \widehat{\varphi w_h}
      -
      \widehat{\Pi}_h
      \del[0]{\widehat{\varphi w_h}}
    }
  }_{\widehat{F}_{\mathcal{Q},m}}^2.
\end{equation}
Consider now the right hand side term. Following \citet[Lemma
7.9]{Georgoulis:2006} we consider the cases $i=m$ and $i\neq m$
separately, starting with $i=m$. Using the commutativity of
$\partial_{\widehat{x}_m}$ with
$\widehat{\pi}_{t} \Pi_{1\leq j\leq d,j\neq m} \widehat{\pi}_{x_j}$,
the triangle inequality, and \citet[Lemma 3.47]{Georgoulis:thesis}
\citep[see also][Lemma 7.8]{Georgoulis:2006},
\begin{equation}
  \label{eq:some_eq_45}
  \begin{split}
    \norm[0]{
      \partial_{\widehat{x}_{m}}
      \del[0]{
        \widehat{\varphi w_h}
        -
        \widehat{\Pi}_h
        \del[0]{\widehat{\varphi w_h}}
      }
    }_{\widehat{F}_{\mathcal{Q},m}}
    \leq &
    \norm[0]{
      \partial_{\widehat{x}_{m}}
      \del[0]{ \widehat{\varphi w_h} }
      -
      \widehat{\pi}_{x_m}
      \partial_{\widehat{x}_{m}}
      \del[0]{ \widehat{\varphi w_h} }
    }_{\widehat{F}_{\mathcal{Q},m}}
    \\
    &
    +
    c
    \norm[0]{
      \widehat{\pi}_{x_m}
      \partial_{\widehat{x}_{m}}
      \del[0]{ \widehat{\varphi w_h} }
      -
      \widehat{\pi}_{t}
      \Pi_{1\leq j\leq d,j\neq m}
      \widehat{\pi}_{x_j}
      \del[0]{
        \widehat{\pi}_{x_m}
        \partial_{\widehat{x}_{m}}
        \del[0]{ \widehat{\varphi w_h} }
      }
    }_{\widehat{\mathcal{K}}}
    \\
    &
    +
    c
    \norm[0]{
      \widehat{\pi}_{t}
      \Pi_{1\leq j\leq d,j\neq m}
      \widehat{\pi}_{x_j}
      \del[0]{
        \widehat{\pi}_{x_m}
        \partial_{\widehat{x}_{m}}
        \del[0]{ \widehat{\varphi w_h} }
        -
        \partial_{\widehat{x}_{m}}
        \widehat{\pi}_{x_m}
        \del[0]{ \widehat{\varphi w_h} }
      }
    }_{\widehat{\mathcal{K}}}.
  \end{split}
\end{equation}
The first and third terms on the right-hand side of
\cref{eq:some_eq_45} vanish by \citet[Lemmas 7.2 and
7.3]{Georgoulis:2006} and \cref{eq:zeroequalities}.  The second term
on the right-hand side of \cref{eq:some_eq_45} is bounded using the
same argument as in the proof of \citet[Lemma 3.4]{Houston:2002} by
noting that $\widehat{\pi}_{x_m}$ and $\widehat{\pi}_t$ are
one-dimensional $L^2$-projections applied in the spatial direction
$\widehat{x}_m$ and time direction, respectively, the commutativity of
$\widehat{\pi}_{x_m}$ with $\partial_{\widehat{t}}^{p_t+1}$ and
$\partial_{\widehat{x}_{j}}^{p_s+1}$ ($j\ne m$), the boundedness of
$\widehat{\pi}_{x_m}$, and \cref{eq:zeroequalities},
\begin{equation}
  \label{eq:some_eq_47}
  \norm[0]{
    \widehat{\pi}_{x_m}
    \partial_{\widehat{x}_{m}}
    \del[0]{ \widehat{\varphi w_h} }
    -
    \widehat{\pi}_{t}
    \Pi_{1\leq j\leq d,j\neq m}
    \widehat{\pi}_{x_j}
    \del[0]{
      \widehat{\pi}_{x_m}
      \partial_{\widehat{x}_{m}}
      \del[0]{ \widehat{\varphi w_h} }
    }
  }_{\widehat{\mathcal{K}}}
  \leq
  c
  \norm[0]{
    \partial_{\widehat{t}}^{p_t+1}
    \partial_{\widehat{x}_{m}}
    \del[0]{ \widehat{\varphi w_h} }
  }_{\widehat{\mathcal{K}}},
\end{equation}
so that \cref{eq:some_eq_45} becomes:
\begin{equation}
  \label{eq:some_eq_49}
  \norm[0]{
    \partial_{\widehat{x}_{m}}
    \del[0]{
      \widehat{\varphi w_h}
      -
      \widehat{\Pi}_h
      \del[0]{\widehat{\varphi w_h}}
    }
  }_{\widehat{F}_{\mathcal{Q},m}}
  \leq
  c
  \norm[0]{
    \partial_{\widehat{t}}^{p_t+1}
    \partial_{\widehat{x}_{m}}
    \del[0]{ \widehat{\varphi w_h} }
  }_{\widehat{\mathcal{K}}}.
\end{equation}

We now consider the right hand side of \cref{eq:some_eq_43} with
$i\neq m$. We have by a triangle inequality
\begin{equation}
  \label{eq:some_eq_50}
  \begin{split}
    \norm[0]{
      \partial_{\widehat{x}_{i}}
      \del[0]{
        \widehat{\varphi w_h}
        -
        \widehat{\Pi}_h
        \del[0]{\widehat{\varphi w_h}}
      }
    }_{\widehat{F}_{\mathcal{Q},m}}
    \leq
    &
    \norm[0]{
      \partial_{\widehat{x}_{i}}
      \del[0]{ \widehat{\varphi w_h} }
      -
      \widehat{\pi}_{x_m}
      \partial_{\widehat{x}_{i}}
      \del[0]{ \widehat{\varphi w_h} }
    }_{\widehat{F}_{\mathcal{Q},m}}
    \\
    & +
    \norm[0]{
      \del[0]{
        I-
        \widehat{\pi}_{t}
        \Pi_{1\leq j\leq d,j\neq i,j\neq m}
        \widehat{\pi}_{x_j}
      }
      \widehat{\pi}_{x_m}
      \partial_{\widehat{x}_{i}}
      \del[0]{ \widehat{\varphi w_h} }
    }_{\widehat{F}_{\mathcal{Q},m}}
    \\
    & +
    \norm[0]{
      \partial_{\widehat{x}_{i}}
      \del[0]{ \widehat{\varphi w_h} }
      -
      \partial_{\widehat{x}_{i}}
      \widehat{\pi}_{x_i}
      \del[0]{ \widehat{\varphi w_h} }
    }_{\widehat{F}_{\mathcal{Q},m}}
    \\
    & +
    \norm[0]{
      \del[0]{
        I-
        \widehat{\pi}_{x_m}
      }
      \del[0]{
        \partial_{\widehat{x}_{i}}
        \del[0]{ \widehat{\varphi w_h} }
        -
        \partial_{\widehat{x}_{i}}
        \widehat{\pi}_{x_i}
        \del[0]{ \widehat{\varphi w_h} }
      }
    }_{\widehat{F}_{\mathcal{Q},m}}
    \\
    & +
    \norm[0]{
      \del[0]{
        I-
        \widehat{\pi}_{t}
        \Pi_{1\leq j\leq d,j\neq i,j\neq m}
        \widehat{\pi}_{x_j}
      }
      \widehat{\pi}_{x_m}
      \del[0]{
        \partial_{\widehat{x}_{i}}
        \del[0]{ \widehat{\varphi w_h} }
        -
        \partial_{\widehat{x}_{i}}
        \widehat{\pi}_{x_i}
        \del[0]{ \widehat{\varphi w_h} }
      }
    }_{\widehat{F}_{\mathcal{Q},m}}.
  \end{split}
\end{equation}
For the second and the fifth terms on the right-hand side, we observe
that the functions inside the norms are polynomials in the
$\widehat{x}_m$-direction. Therefore, \citet[Lemma
7.8]{Georgoulis:2006} gives us, using similar steps used to find
\cref{eq:some_eq_47},
\begin{equation*}
  \begin{split}
    &\norm[0]{
      \del[0]{
        I-
        \widehat{\pi}_{t}
        \Pi_{1\leq j\leq d,j\neq i,j\neq m}
        \widehat{\pi}_{x_j}
      }
      \widehat{\pi}_{x_m}
      \partial_{\widehat{x}_{i}}
      \del[0]{ \widehat{\varphi w_h} }
    }_{\widehat{F}_{\mathcal{Q},m}}
    \\
    &\quad
    +
    \norm[0]{
      \del[0]{
        I-
        \widehat{\pi}_{t}
        \Pi_{1\leq j\leq d,j\neq i,j\neq m}
        \widehat{\pi}_{x_j}
      }
      \widehat{\pi}_{x_m}
      \del[0]{
        \partial_{\widehat{x}_{i}}
        \del[0]{ \widehat{\varphi w_h} }
        -
        \partial_{\widehat{x}_{i}}
        \widehat{\pi}_{x_i}
        \del[0]{ \widehat{\varphi w_h} }
      }
    }_{\widehat{F}_{\mathcal{Q},m}}
    \\
    \leq
    &
    c
    \del[1]{
      \norm[0]{
        \partial_{\widehat{t}}^{p_t+1}
        \partial_{\widehat{x}_{i}}
        \del[0]{ \widehat{\varphi w_h} }
      }_{\widehat{\mathcal{K}}}
      +
      \norm[0]{
        \partial_{\widehat{t}}
        \del[0]{
          \partial_{\widehat{x}_{i}}
          \del[0]{ \widehat{\varphi w_h} }
          -
          \partial_{\widehat{x}_{i}}
          \widehat{\pi}_{x_i}
          \del[0]{ \widehat{\varphi w_h} }
        }
      }_{\widehat{\mathcal{K}}}
    }
    \\
    &\quad
    +
    c
    \del[0]{
      \sum_{1\leq j\leq d,j\neq m,j\neq i}
      \norm[0]{
        \partial_{\widehat{x}_{j}}
        \del[0]{
          \partial_{\widehat{x}_{i}}
          \del[0]{ \widehat{\varphi w_h} }
          -
          \partial_{\widehat{x}_{i}}
          \widehat{\pi}_{x_i}
          \del[0]{ \widehat{\varphi w_h} }
        }
      }_{\widehat{\mathcal{K}}}
    }.
  \end{split}
\end{equation*}
Next, using that $\widehat{\pi}_{x_i}$ and $\partial_{\widehat{x}_j}$
commute, using \citet[eq.(7.5) in Lemma 7.3]{Georgoulis:2006} and
\cref{eq:zeroequalities}, we find that for any $j\neq i$ where
$1\le i \le d$ and $0 \le j\le d$ with $\widehat{x}_0 = \widehat{t}$,
\begin{equation*}
  \norm[0]{
    \partial_{\widehat{x}_{j}}
    \del[0]{
      \partial_{\widehat{x}_{i}}
      \del[0]{ \widehat{\varphi w_h} }
      -
      \partial_{\widehat{x}_{i}}
      \widehat{\pi}_{x_i}
      \del[0]{ \widehat{\varphi w_h} }
    }
  }_{\widehat{\mathcal{K}}}^2
  \leq
  c
  \norm[0]{\partial_{\widehat{x}_{i}}^{p_s+1}
    \partial_{\widehat{x}_{j}}
    \del[0]{ \widehat{\varphi w_h} }
  }_{\widehat{\mathcal{K}}}^2
  =0.
\end{equation*}
Therefore, the second and fifth terms on the right hand side of
\cref{eq:some_eq_50} are bounded by
$c\norm[0]{ \partial_{\widehat{t}}^{p_t+1} \partial_{\widehat{x}_{i}}
  \del[0]{ \widehat{\varphi w_h} } }_{\widehat{\mathcal{K}}}$. All
remaining terms on the right hand side of \cref{eq:some_eq_50} vanish
by combining \citet[Lemmas 7.2 and 7.3]{Georgoulis:2006}
\cref{eq:refined_aniso_trace_aff,eq:zeroequalities} and so, for
$i\ne m$,
\begin{equation}
  \label{eq:some_eq_54}
  \norm[0]{
    \partial_{\widehat{x}_{i}}
    \del[0]{
      \widehat{\varphi w_h}
      -
      \widehat{\Pi}_h
      \del[0]{\widehat{\varphi w_h}}
    }
  }_{\widehat{F}_{\mathcal{Q},m}}
  \leq
  c
  \norm[0]{
    \partial_{\widehat{t}}^{p_t+1}
    \partial_{\widehat{x}_{i}}
    \del[0]{ \widehat{\varphi w_h} }
  }_{\widehat{\mathcal{K}}}.
\end{equation}

Combining \cref{eq:some_eq_43,eq:some_eq_49,eq:some_eq_54} and
applying \cref{eq:lb7sudir-1,eq:lebnizinverse_2} we obtain:
\begin{equation}
  \label{eq:some_eq_107}
  \norm[0]{
    \widetilde{\overline{\nabla}}
    \del[0]{
      \widetilde{\varphi w_h}
      -
      \widetilde{\Pi}_h
      \del[0]{\widetilde{\varphi w_h}}
    }
  }_{\widetilde{F}_{\mathcal{Q},m}}
  \leq
  c
  \delta t_{\mathcal{K}}h_K^{-3/2}
  \norm[0]{\widetilde{w}_h}_{\widetilde{\mathcal{K}}}.
\end{equation}
A scaling argument applied to \cref{eq:some_eq_107} from
$\widetilde{\mathcal{K}}$ to $\mathcal{K}$ yields
\cref{eq:proj_with_wt_5}.

\Cref{eq:proj_with_wt_3} follows directly by combining the local trace
inequality \cref{eq:refined_aniso_trace_aff_1} with
\cref{eq:some_eq_102,eq:some_eq_108}:
\begin{equation}
  \label{eq:some_eq_110}
  \norm[0]{
    \widetilde{\varphi w_h}
    -
    \widetilde{\Pi}_h
    \del[0]{\widetilde{\varphi w_h}}
  }_{\widetilde{\mathcal{Q}}_\mathcal{K}}
  \leq
  c
  h_K^{-1/2}
  \delta t_{\mathcal{K}}
  \norm[0]{\widetilde{w_h}}_{\widetilde{\mathcal{K}}},
\end{equation}
and a scaling argument applied to \cref{eq:some_eq_110} from
$\widetilde{\mathcal{K}}$ to $\mathcal{K}$. Lastly,
\cref{eq:proj_with_wt_4} follows by combining the local trace
inequality \cref{eq:refined_aniso_trace_aff_2} with
\cref{eq:some_eq_102,eq:some_eq_109}:
\begin{equation}
  \label{eq:some_eq_111}
  \norm[0]{
    \widetilde{\varphi w_h}
    -
    \widetilde{\Pi}_h
    \del[0]{\widetilde{\varphi w_h}}
  }_{\widetilde{\mathcal{R}}_\mathcal{K}}
  \leq
  c
  {\delta t_{\mathcal{K}}}^{1/2}
  \norm[0]{\widetilde{w_h}}_{\widetilde{\mathcal{K}}},
\end{equation}
and a scaling argument applied to \cref{eq:some_eq_111} from
$\widetilde{\mathcal{K}}$ to $\mathcal{K}$.

\section{Proof of \cref{lem:stab_fortin}}
\label{s:stab_fortin}

We start with volume terms in the definition of $\tnorm{\cdot}_v$
\cref{eq:v_norm}. Due to boundedness of $\Pi_h$, and using
\cref{eq:proj_spatial_grad},
\begin{equation}
  \label{eq:volumetermbound1}
  \sum_{\mathcal{K}\in\mathcal{T}_h}
  \norm[0]{\Pi_h\del[0]{\varphi w_h}}_\mathcal{K}^2
  +
  \sum_{\mathcal{K}\in\mathcal{T}_h}
  \varepsilon
  \norm[0]{
    \overline{\nabla}\del[0]{\Pi_h\del[0]{\varphi w_h}}
  }_{\mathcal{K}}^2
  \leq
  c
  \del[0]{eT+\chi}^2
  \tnorm{\boldsymbol{w}_h}_{v}^2.
\end{equation}
Next, the diffusive facet terms are bounded using a triangle
inequality, \cref{eq:proj_diff_elem_facet_Q}, and boundedness of
$\Pi_h^{\mathcal{F}}$:
\begin{equation}
  \label{eq:diftermbound3}
  \begin{split}
    &
    \sum_{\mathcal{K}\in\mathcal{T}_h}
    \varepsilon h_K^{-1}
    \norm[0]{
      \Pi_h\del[0]{\varphi w_h}
      -
      \Pi_h^{\mathcal{F}}
      \del[0]{\varphi\varkappa_h}
    }_{\mathcal{Q}_\mathcal{K}}^2
    \\
    &\quad \leq
    c
    \sum_{\mathcal{K}\in\mathcal{T}_h}
    \varepsilon h_K^{-1}
    h_K
    \norm[0]{
      \overline{\nabla}
      \del[0]{\varphi w_h}
    }_{\mathcal{K}}^2
    +
    c
    \sum_{\mathcal{K}\in\mathcal{T}_h}
    \varepsilon h_K^{-1}
    \norm[0]{
      \varphi \sbr{\boldsymbol{w}_h}
    }_{\mathcal{Q}_\mathcal{K}}^2
    \leq
    c\del[0]{eT+\chi}^2
    \tnorm{\boldsymbol{w}_h}_v^2.
  \end{split}
\end{equation}
For the Neumann boundary term in the definition of $\tnorm{\cdot}_v$,
consider first a single facet $F \in \partial\mathcal{E}_N$. Then,
\begin{multline*}
  \norm[0]{ \envert[0]{ \tfrac{1}{2} \beta\cdot{n} }^{1/2} \Pi_h^{\mathcal{F}} \del[0]{\varphi\varkappa_h} }_F
  \\
  \le
  \norm[0]{ \envert[0]{ \tfrac{1}{2} \beta\cdot{n} }^{1/2} \del[1]{ \Pi_h^{\mathcal{F}} \del[0]{\varphi\varkappa_h}
      - \Pi_h \del[0]{\varphi w_h} } }_F + \norm[0]{ \envert[0]{ \tfrac{1}{2} \beta\cdot{n} }^{1/2} \Pi_h \del[0]{\varphi w_h} }_F
  :=I + II.
\end{multline*}
For term $I$, using that
$\envert[0]{ \beta\cdot{n} } \le \del[0]{ \max_{(t,x)\in F}
  |\beta\cdot{n}| }$, that
$\Pi_h^{\mathcal{F}}\Pi_h(\varphi w_h) = \Pi_h(\varphi w_h)$ on $F$,
boundedness of $\Pi_h^{\mathcal{F}}$ and a triangle inequality we have:
\begin{equation*}
  \begin{split}
    \norm[0]{ \envert[0]{ \tfrac{1}{2} \beta\cdot{n} }^{1/2}
      \del[1]{ \Pi_h^{\mathcal{F}} \del[0]{\varphi\varkappa_h} - \Pi_h \del[0]{\varphi w_h} } }_F
    \le&
    \del[0]{ \tfrac{1}{2} \max_{(t,x)\in F} |\beta\cdot{n}| }^{1/2}
    \norm[0]{\varphi\varkappa_h - \Pi_h \del[0]{\varphi w_h} }_F
    \\
    \le&
    \del[0]{ \tfrac{1}{2} \max_{(t,x)\in F} |\beta\cdot{n}| }^{1/2}
    \norm[0]{\varphi\varkappa_h - \varphi w_h }_F
    \\
    &+
    \del[0]{ \tfrac{1}{2} \max_{(t,x)\in F} |\beta\cdot{n}| }^{1/2}
    \norm[0]{(I- \Pi_h)\del[0]{\varphi w_h} }_F.
  \end{split}
\end{equation*}
Using that $|\varphi| \le eT+\chi$ and \cref{eq:betasinfmax} for the
first term on the right hand side, and using \cref{eq:proj_with_wt_3}
and that $\delta t_{\mathcal{K}}\leq h_K$ for the second term, we
obtain
\begin{equation}
  \label{eq:termIbound}
  I \le
  \del[0]{eT+\chi}
  \norm[0]{ \del[0]{\beta_s - \tfrac{1}{2}\beta\cdot n}^{1/2}(\varkappa_h - w_h) }_F
  +
  c
  \norm[0]{w_h}_{\mathcal{K}_F},
\end{equation}
where $\mathcal{K}_F$ is the space-time element of which $F$ is a
facet. Next, for term $II$, by a triangle inequality, using
\cref{eq:proj_with_wt_3}, that $\delta t_{\mathcal{K}}\leq h_K$,
$|\varphi| \le eT+\chi$, and \cref{eq:betasinfmax},
\begin{equation*}
  \begin{split}
    II
    &
    \le
    \norm[0]{ \envert[0]{ \tfrac{1}{2} \beta\cdot{n} }^{1/2}
      (I-\Pi_h) (\varphi w_h) }_F
    + \norm[0]{ \envert[0]{ \tfrac{1}{2} \beta\cdot{n} }^{1/2} \varphi (w_h-\varkappa_h) }_F
    +
    \del[0]{eT+\chi}
    \norm[0]{ \envert[0]{ \tfrac{1}{2} \beta\cdot{n} }^{1/2} \varkappa_h }_F
    \\
    &
    \le
    c
    \norm[0]{w_h}_{\mathcal{K}}
    +
    \del[0]{eT+\chi}
    \norm[0]{ \envert[0]{ \beta_s - \tfrac{1}{2} \beta\cdot{n} }^{1/2} (w_h-\varkappa_h) }_F
    +
    \del[0]{eT+\chi}
    \norm[0]{ \envert[0]{ \tfrac{1}{2} \beta\cdot{n} }^{1/2} \varkappa_h }_F.
  \end{split}
\end{equation*}
For a facet $F \in \partial\mathcal{E}_N$ we therefore conclude that
\begin{multline}
  \label{eq:singleboundFinEN}
  \norm[0]{ \envert[0]{ \tfrac{1}{2} \beta\cdot{n} }^{1/2} \Pi_h^{\mathcal{F}} \del[0]{\varphi\varkappa_h} }_F
  \le
  c
  \del[0]{eT+\chi}
  \norm[0]{ \del[0]{\beta_s - \tfrac{1}{2}\beta\cdot n}^{1/2}(\varkappa_h - w_h) }_F
  \\
  +
  c
  \norm[0]{w_h}_{\mathcal{K}}
  + \del[0]{eT+\chi}
  \norm[0]{ \envert[0]{ \tfrac{1}{2} \beta\cdot{n} }^{1/2} \varkappa_h }_F.
\end{multline}
We find for the Neumann term in the definition of $\tnorm{\cdot}_v$:
\begin{multline}
  \label{eq:boundFinEN}
  \sum_{F \in \partial\mathcal{E}_N}\norm[0]{ \envert[0]{ \tfrac{1}{2} \beta\cdot{n} }^{1/2}
    \Pi_h^{\mathcal{F}} \del[0]{\varphi\varkappa_h} }_F^2
  \le
  c
  \norm[0]{\beta}_{L^{\infty}(\mathcal{E})}
  \sum_{\mathcal{K} \in \mathcal{T}_h}
  \norm[0]{w_h}_{\mathcal{K}}^2
  \\
  +
  \del[0]{eT+\chi}^2
  \del[1]{
    \sum_{\mathcal{K} \in \mathcal{T}_h}
    \norm[0]{
      \del[0]{
        \beta_s - \tfrac{1}{2}\beta\cdot n
      }^{1/2}
      (\varkappa_h - w_h)
    }_{\partial \mathcal{K}}^2
    + \sum_{F \in \partial \mathcal{E}_N}
    \norm[0]{
      \envert[0]{ \tfrac{1}{2} \beta\cdot{n} }^{1/2}
      \varkappa_h
    }_F^2
  }.
\end{multline}
Finally, we consider the advective facet terms in the definition of
$\tnorm{\cdot}_v$. On a single facet we have:
\begin{equation*}
  \norm[0]{
    \envert[0]{\beta_s-\tfrac{1}{2}\beta\cdot{n}}^{1/2}
    \del[0]{
      \Pi_h\del[0]{\varphi w_h}
      -
      \Pi_h^{\mathcal{F}}
      \del[0]{\varphi\varkappa_h}
    }
  }_{F}
  \le
  c
  (
  \max_{(t,x)\in F} |\beta\cdot n|)^{1/2}
  \norm[0]{
    \del[0]{
      \Pi_h\del[0]{\varphi w_h}
      -
      \Pi_h^{\mathcal{F}}
      \del[0]{\varphi\varkappa_h}
    }
  }_{F}.
\end{equation*}
Using identical steps as used to find the bound for $I$ in
\cref{eq:termIbound}, we find:
\begin{multline*}
  \norm[0]{
    \envert[0]{\beta_s-\tfrac{1}{2}\beta\cdot{n}}^{1/2}
    \del[0]{
      \Pi_h\del[0]{\varphi w_h}
      -
      \Pi_h^{\mathcal{F}}
      \del[0]{\varphi\varkappa_h}
    }
  }_{F}
  \\
  \le
  c
  \del[0]{eT+\chi}
  \norm[0]{ \del[0]{\beta_s - \tfrac{1}{2}\beta\cdot n}^{1/2}(\varkappa_h - w_h) }_F
  +
  c \norm[0]{w_h}_{\mathcal{K}},
\end{multline*}
so that
\begin{multline}
  \label{eq:FboundinAT}
    \sum_{\mathcal{K} \in \mathcal{T}_h}
    \norm[0]{
      \envert[0]{\beta_s-\tfrac{1}{2}\beta\cdot{n}}^{1/2}
      \del[0]{
        \Pi_h\del[0]{\varphi w_h}
        -
        \Pi_h^{\mathcal{F}}
        \del[0]{\varphi\varkappa_h}
      }
    }_{\partial \mathcal{K}}^2
    \\
    \le
    c \del[0]{eT+\chi}^2
    \sum_{\mathcal{K} \in \mathcal{T}_h}
    \norm[0]{ \del[0]{\beta_s - \tfrac{1}{2}\beta\cdot n}^{1/2}
      (\varkappa_h - w_h)
    }_{\partial\mathcal{K}}^2
    +
    c \del[1]{
      \sum_{\mathcal{K} \in \mathcal{T}_h}
      \norm[0]{w_h}_{\mathcal{K}}^2 }.
\end{multline}
The result follows after collecting the bounds in
\cref{eq:volumetermbound1,eq:diftermbound3,eq:boundFinEN,eq:FboundinAT}.

\subsubsection*{Acknowledgements}

SR gratefully acknowledges support from the Natural Sciences and
Engineering Research Council of Canada through the Discovery Grant
program (RGPIN-05606-2015).

This research was enabled in part by support provided by Simon
Fraser University
(\url{https://www.sfu.ca/research/supercomputer-cedar}), Compute
Ontario (\url{https://www.computeontario.ca/}) and the Digital
Research Alliance of Canada (\url{https://alliancecan.ca}). We
furthermore acknowledge the support provided by the Math Faculty
Computing Facility at the University of Waterloo
(\url{https://uwaterloo.ca/math-faculty-computing-facility/}).

\bibliographystyle{abbrvnat}
\bibliography{references}
\end{document}